\documentclass{imsart}
\RequirePackage{amsthm,amsmath,amsfonts,amssymb}
\RequirePackage[numbers]{natbib}
\RequirePackage[colorlinks,citecolor=blue,urlcolor=blue]{hyperref}
\RequirePackage{graphicx}

\usepackage{bbm}
\usepackage{float}
\usepackage{comment}
\usepackage{mathtools}
\usepackage{xr}
\usepackage{placeins}
\usepackage{algorithm}

\usepackage{newcommands}
\usepackage{tikz}

\startlocaldefs
\theoremstyle{plain}

\newtheorem{theorem}{Theorem}[section]
\newtheorem{lemma}[theorem]{Lemma}

\newtheorem{proposition}{Proposition}[section]
\theoremstyle{remark}
\newtheorem{remark}{Remark}[section]
\newtheorem{definition}[theorem]{Definition}

\endlocaldefs

\begin{document}
	
	\begin{frontmatter}
		\title{Minmax Trend Filtering: Generalizations of Total Variation Denoising via a Pointwise Formula}
		\runtitle{Minmax Trend Filtering}
		
		\begin{aug}
			\author{\fnms{Sabyasachi} \snm{Chatterjee}\ead[label=e2,mark]{sc1706@illinois.edu}}
			
			\address{Department of Statistics,
				University of Illinois at Urbana Champaign,
				\printead{e2}}
		\end{aug}
		
		\begin{abstract}
Total Variation Denoising (TVD) is a fundamental method for denoising and nonparametric regression. In this article, we identify a new exact pointwise representation of the univariate TVD
estimator: each fitted value can be expressed as a minmax/maxmin of a simple function of local
averages over intervals containing the target point. 
This minmax/maxmin formulation is generalizeable and provides a framework for constructing other locally adaptive estimators. Building on this viewpoint, we propose and study higher-order polynomial
generalizations of TVD, defined pointwise as values lying between minmax and maxmin
optimizations of simple functions of local polynomial regressions over intervals at
multiple scales. These estimators, which we call Minmax Trend Filtering (MTF), are
distinct from classical trend filtering and other existing methods in the nonparametric
regression literature. To address computational considerations, we also introduce a
dyadic variant of MTF that restricts attention to a multiscale dyadic family of intervals,
yielding a near-linear-time implementation while preserving the essential statistical
properties. The proposed local formulation of TVD and MTF makes it tractable to bound pointwise
estimation errors in terms of a transparent local bias--variance–type tradeoff. This form
of pointwise analysis is new and simpler than existing analyses of TVD and Trend
Filtering. In addition to minimax rate optimality over bounded variation and piecewise
polynomial classes, our results also yield near-optimal local rates of convergence for locally
H\"older smooth signals. The resulting pointwise risk bounds exhibit a clean dependence on the
tuning parameter $\lambda$ and provide a new insight on why TVD and MTF
exhibit stronger local adaptivity than classical linear smoothers.
\end{abstract}

		\begin{keyword}[class=MSC]
			\kwd[Primary ]{62G05}
			\kwd{62G08}
		\end{keyword}
		
		\begin{keyword}
			\kwd{Locally Adaptive Nonparametric Regression}
			\kwd{Pointwise Bounds}
			\kwd{Local Bias Variance Tradeoff}
			\kwd{Piecewise Polynomial Function Estimation}
			\kwd{Fused Lasso}
			\kwd{Total Variation Denoising}
			\kwd{Trend Filtering}
			\kwd{Minmax Optimization}
			\kwd{Dyadic Interval Search}
			\kwd{Boundary Consistency}
		\end{keyword}
		
	\end{frontmatter}
	
	
\section{Introduction}

\subsection{Nonparametric Regression and Local Adaptivity}

Nonparametric regression is a fundamental problem in statistics; see
\cite{Gyorfi, WassermanNonpar, Tsybakovbook} for general introductions. The goal
is to estimate the regression function
$f(x) = \E(Y \mid X = x)$ from observations $\{(x_i,y_i)\}_{i=1}^n$, under weak
structural assumptions on $f$, such as Lipschitz or H\"older smoothness, or
bounded variation.

In this article, we revisit \emph{total variation denoising} (TVD), a canonical
method for nonparametric regression, and study new univariate regression
procedures that generalize the classical TVD estimator.

Throughout most of the paper, we consider a fixed-design setting in which the
design points lie on a regular grid in $[0,1]$,
\[
x_1 < x_2 < \cdots < x_n, \qquad x_i = \frac{i}{n}.
\]
Letting $\theta_i^* = f(i/n)$, we obtain the signal-plus-noise model
\begin{equation}\label{eq:thetaep}
    y_i = \theta_i^* + \epsilon_i, \qquad i = 1,\dots,n .
\end{equation}
We assume that the noise variables $\epsilon_1,\dots,\epsilon_n$ are independent,
mean-zero, sub-Gaussian with variance proxy bounded by $\sigma^2 > 0$, that is,
\begin{equation}\label{eq:psi2norm}
\sup_{i \in [n]} \E\!\left[e^{t\epsilon_i}\right]
\le e^{\sigma^2 t^2 / 2}
\qquad \text{for all } t \in \R,
\end{equation}
see, for example,~\cite[Section~5.2.3]{vershynin2018high}. Under this model, the
objective is to estimate the unknown signal $\theta^*$, or equivalently the
regression function $f$ at the design points, from the observed data vector $y$.

In many applications, the regression function is not uniformly smooth, and its
local regularity may vary across the domain. This motivates the study of
\emph{locally adaptive} methods, which adapt their estimation accuracy to the
local smoothness of the underlying function.

Although the class of nonparametric regression methods is broad
(see, e.g.,~\cite{cover1967nearest, fan1996framework, wand1994kernel,
de1978practical, green1993nonparametric, wahba1990spline, smola1998learning,
donoho1994ideal, tibshirani2020divided, breiman2017classification,
bishop1994neural}), not all methods are locally adaptive in this sense. In
particular, it is well known
(see, e.g.,~\cite{donoho1998minimax, sadhanala2016total, sadhanala2017higher})
that linear smoothers, including kernel smoothing, local polynomial regression,
and smoothing splines, fail to achieve local adaptivity in a minimax sense.

In contrast, total variation denoising is a nonlinear estimator and is widely regarded as a
locally adaptive method. It can be viewed as the $0$th order case of \emph{trend
filtering} (TF), a class of estimators introduced in the optimization literature;
see~\cite{steidl2006splines, kim2009ell_1}. Trend filtering may be interpreted as
a discrete analogue of locally adaptive regression splines;~\cite{koenker1994quantile, mammen1997locally}. A comprehensive statistical and
computational analysis of TF was provided in~\cite{tibshirani2014adaptive},
where it was argued that TF offers a computationally efficient approximation to
locally adaptive regression splines while retaining their adaptivity properties.
Subsequent work has further investigated TF from various perspectives; see, for
example,~\cite{tibshirani2020divided, guntuboyina2020adaptive,
ortelli2021prediction} and references therein. \emph{In this article, we focus on
trend filtering of order $0$, corresponding to total variation denoising, and
revisit its properties from a new perspective.}

\subsection{Main result for total variation denoising}

Total variation denoising (TVD) is a classical nonparametric regression and
denoising method that has been extensively studied in statistics, applied
mathematics, and signal processing. It was introduced in the seminal work of
\cite{rudin1992nonlinear}, which proposed total variation regularization for
two-dimensional image denoising. In the univariate setting, TVD is also known
as the fused lasso. Given a data vector $y \in \R^n$, the univariate TVD
estimator is defined as
\begin{equation}\label{eq:tvd}
    \hat{\theta}^{(\lambda)}
    = \arg\min_{\theta \in \R^n}
    \frac{1}{2} \sum_{i=1}^n (y_i - \theta_i)^2 + \lambda \, TV(\theta),
\end{equation}
where $TV(\theta) = \sum_{i=1}^{n-1} |\theta_{i+1} - \theta_i|$ and $\lambda \ge 0$
is a tuning parameter.

The total variation penalty promotes sparsity in the first-order differences
of the fitted signal, and consequently the estimator $\hat{\theta}^{(\lambda)}$
is piecewise constant. Moreover, the univariate TVD estimator can be computed
in linear time using specialized algorithms; see, for example,
\cite{johnson2013dynamic, condat2013direct, makovetskii2020tube}.

Despite its favorable properties, including local adaptivity and fast
computation, no closed-form expression for the fitted value
$\hat{\theta}^{(\lambda)}_i$ at a given location $i$ is known. In this article, we
derive an exact pointwise minmax and maxmin representation for the TVD fit.
This representation enables a refined analysis of the estimator, including
the study of local risk behavior, location-dependent rates of convergence
adapted to the local smoothness of the underlying regression function, and
explicit pointwise risk bounds as tractable functions of the tuning parameter
$\lambda$. In turn, these results provide transparent guidance for the choice
of $\lambda$ for each location.

Our main result shows that the TVD fit at a given index can be expressed as a
min--max or max--min over simple functions of local averages of the data. The
following theorem is one of the central results of the paper.

\begin{theorem}[A pointwise formula for TVD]\label{thm:ptwiseformula}
Fix any $i \in [n]$, and let $\mathcal{I}$ denote the collection of all discrete
intervals of $[n]$. For the TVD estimator $\hat{\theta}^{(\lambda)}$ defined
in~\eqref{eq:tvd}, the following identity holds:
\begin{equation}\label{eq:fldefn}
\max_{J \in \mathcal{I}: i \in J}
\min_{I \in \mathcal{I}: I \subseteq J,\, i \in I}
\Bigl[ \bar y_I + C_{I,J} \frac{2\lambda}{|I|} \Bigr]
=
\hat{\theta}^{(\lambda)}_i
=
\min_{J \in \mathcal{I}: i \in J}
\max_{I \in \mathcal{I}: I \subseteq J,\, i \in I}
\Bigl[ \bar y_I - C_{I,J} \frac{2\lambda}{|I|} \Bigr],
\end{equation}
where $\bar y_I = |I|^{-1} \sum_{j \in I} y_j$, and the constants $C_{I,J}$ are
defined in Definition~\ref{def:cij} below.
\end{theorem}

\begin{definition}\label{def:cij}
Fix an interval $J \subseteq [n]$ and a subinterval $I \subseteq J$. Write
$J = [j_1 : j_2]$ and $I = [s : t]$. Throughout, $I \subset J$ denotes that $I$ lies strictly in the interior of $J$, that is, $I$ does not contain either boundary point of $J.$ There are different cases depending on whether $J$ contains none, one or both global boundary points $\{1,n\}$.

If $1 < j_1 \le j_2 < n$, then
\[
C_{I,J} =
\begin{cases}
1, & \text{if } I \subset J, \\
-1, & \text{if } I = J, \\
0, & \text{otherwise}.
\end{cases}
\]

If $1 = j_1 \le j_2 < n$, then
\[
C_{I,J} =
\begin{cases}
1, & \text{if } I \subset J, \\
-1/2, & \text{if } I = J, \\
1/2, & \text{if } 1 = s \le t < j_2, \\
0, & \text{if } 1 < s \le t = j_2 .
\end{cases}
\]

If $1 < j_1 \le j_2 = n$, then
\[
C_{I,J} =
\begin{cases}
1, & \text{if } I \subset J, \\
-1/2, & \text{if } I = J, \\
1/2, & \text{if } j_1 < s \le t = n, \\
0, & \text{if } j_1 = s \le t < n .
\end{cases}
\]

If $J = [1 : n]$, then
\[
C_{I,J} =
\begin{cases}
1, & \text{if } I \subset J, \\
0, & \text{if } I = J, \\
1/2, & \text{otherwise}.
\end{cases}
\]

\end{definition}

The representation in~\eqref{eq:fldefn} holds for all indices $i \in [n]$, all
data vectors $y \in \R^n$, and all tuning parameters $\lambda \ge 0$. Although
the TVD estimator is defined as the solution to a global convex optimization
problem, Theorem~\ref{thm:ptwiseformula} shows that each fitted value
$\hat{\theta}^{(\lambda)}_i$ admits an exact local characterization in terms of
averages of the data over intervals containing $i$. Specifically,
$\hat{\theta}^{(\lambda)}_i$ is obtained by a nested minmax or maxmin
operation over intervals: one first fixes an interval $J$ containing $i$, then
optimizes over all subintervals $I \subseteq J$ that also contain $i$ via
modified local averages of the form
\[
\bar y_I \;\pm\; \frac{2\lambda}{|I|} C_{I,J}.
\]
In Section~\ref{sec:unitvd}, we provide a proof sketch of
Theorem~\ref{thm:ptwiseformula}, which explains the origin of this minmax
structure and the role of the constants $C_{I,J}$.

Univariate total variation denoising admits a classical geometric
\emph{taut string} characterization. One considers the cumulative sums
$Y_0 = 0, Y_k = \sum_{j=1}^k y_j$ and draws the shortest path (pinned at the endpoints)
that remains within a tube of radius $\lambda$ around $\{Y_k\}_{k=0}^n$; the fitted
signal $\hat\theta^{(\lambda)}$ is then obtained by taking discrete differences
of this path. This viewpoint is classical; see, for example,
\cite{mammen1997locally, davies2001local}. While this characterization provides a
global and geometric viewpoint of the TVD solution, it does not by itself yield
an explicit pointwise min--max/max--min representation such as
Theorem~\ref{thm:ptwiseformula}, to the best of our knowledge. 

\subsection{Connection with isotonic regression}

The pointwise min--max characterization of TVD is reminiscent of classical
results for isotonic regression. Isotonic regression~\cite{RWD88} is a
fundamental nonparametric regression method that solves
\[
\min_{\theta_1 \le \cdots \le \theta_n}
\sum_{i=1}^n (y_i - \theta_i)^2 .
\]
Its solution admits the well-known min--max representation
\[
\max_{u \le i} \min_{l \ge i}
\frac{1}{l-u+1} \sum_{j=u}^l y_j
=
\hat{\theta}_i
=
\min_{l \ge i} \max_{u \le i}
\frac{1}{l-u+1} \sum_{j=u}^l y_j .
\]

Many properties of isotonic regression are most transparently derived using
this formula; see, for example, the statistical analysis in~\cite{Zhang02}.
There are other connections between isotonic regression and TVD. The TVD
objective can be written as a generalized lasso problem, while isotonic
regression corresponds to a related nonnegative lasso formulation. Both methods
promote sparsity in first-order differences, yield piecewise constant fits, and
are regarded as locally adaptive. Indeed,~\cite{mammen1997locally} notes that
TVD behaves similarly to isotonic regression in regions where the underlying
signal is monotone for sufficiently large values of $\lambda$.

Despite these similarities, the minmax structure for TVD differs
fundamentally from that of isotonic regression. In isotonic regression, the
minmax/maxmin operations are taken over lower and upper intervals containing the
index $i$, whereas in the TVD representation of
Theorem~\ref{thm:ptwiseformula}, the optimization is over all intervals and their
subintervals containing $i$.

\subsection{Role of the min--max representation}

A central contribution of this article is the use of the minmax/maxmin
representation to analyze pointwise estimation errors of the TVD estimator.
Specifically, the min--max structure leads naturally to a bias--variance--type
decomposition of the local risk; see Theorem~\ref{thm:maingeneral}.

The literature on the statistical accuracy of univariate TVD is extensive; see,
for example,~\cite{mammen1997locally, Harchaoui2010multiple,
dalalyan2017tvd, lin2017sharp, Ortelli2018, ortelli2021prediction,
guntuboyina2020adaptive, madrid2022risk}. However, these works primarily study
global risk measures such as mean squared error. The study of pointwise
estimation error for TVD was initiated more recently in~\cite{zhang2023element},
but the results there are limited to piecewise constant signals with a small
number of segments. In contrast, we derive general pointwise risk bounds for
TVD; see Theorem~\ref{thm:mainada}. These results substantially extend the scope
of~\cite{zhang2023element}; see Section~\ref{sec:previous} for further
discussion. The key technical tool enabling this analysis is the proposed
min--max/max--min representation.

We now briefly sketch how the min--max/max--min formula yields deterministic
pointwise error bounds. Fix an index $i \in [n]$. Using the min--max
representation in~\eqref{eq:fldefn}, for any interval $J$ containing $i$ we have
\begin{align*}
\hat{\theta}^{(\lambda)}_i
&\le
\max_{I \subseteq J: i \in I}
\Bigl[ \bar y_I - C_{I,J}\frac{2\lambda}{|I|} \Bigr] \\
&=
\max_{I \subseteq J: i \in I}
\Bigl[ \bar{\theta}^*_I + \bar{\epsilon}_I - C_{I,J}\frac{2\lambda}{|I|} \Bigr] \\
&\le
\max_{I \subseteq J: i \in I} \bar{\theta}^*_I
+
\max_{I \subseteq J: i \in I}
\Bigl[ \bar{\epsilon}_I - C_{I,J}\frac{2\lambda}{|I|} \Bigr].
\end{align*}
Consequently,
\[
\hat{\theta}^{(\lambda)}_i - \theta^*_i
\le
\underbrace{
\max_{I \subseteq J: i \in I}
\bigl( \bar{\theta}^*_I - \theta^*_i \bigr)
}_{\mathrm{Bias}_{+}(i,J,\theta^*)}
+
\underbrace{
\max_{I \subseteq J: i \in I}
\Bigl[ \bar{\epsilon}_I - C_{I,J}\frac{2\lambda}{|I|} \Bigr]
}_{\mathrm{SD}(i,J,\lambda)} .
\]

Similarly, using the max--min representation in~\eqref{eq:fldefn} yields the
lower bound
\[
\hat{\theta}^{(\lambda)}_i - \theta^*_i
\ge
\underbrace{
\min_{I \subseteq J: i \in I}
\bigl( \bar{\theta}^*_I - \theta^*_i \bigr)
}_{\mathrm{Bias}_{-}(i,J,\theta^*)}
-
\mathrm{SD}(i,J,\lambda) .
\]

Since these inequalities hold for every interval $J$ containing $i$, we obtain
\[
\max_{J \in \mathcal{I}: i \in J}
\bigl( \mathrm{Bias}_{-}(i,J,\theta^*) - \mathrm{SD}(i,J,\lambda) \bigr)
\le
\hat{\theta}^{(\lambda)}_i - \theta^*_i
\le
\min_{J \in \mathcal{I}: i \in J}
\bigl( \mathrm{Bias}_{+}(i,J,\theta^*) + \mathrm{SD}(i,J,\lambda) \bigr),
\]
where the notation $\mathrm{Bias}_{\pm}$ and $\mathrm{SD}$ is suggestive of a
bias--variance decomposition and will be formalized later.

This establishes deterministic upper and lower bounds on the pointwise
estimation error, expressed as an optimal bias--variance--type tradeoff over all
intervals $J$ containing the index $i$. Our main pointwise risk bound,
Theorem~\ref{thm:maingeneral}, is of this form. The remaining analysis consists
of bounding the stochastic term $\mathrm{SD}(i,J,\lambda)$ and optimizing the
tradeoff for specific classes of signals $\theta^*$.

\subsection{New Higher Order Generalizations of Total Variation Denoising}
\label{sec:newgeneral}

In total variation denoising, the regularization term is the $\ell_1$ norm of the
first-order differences,
\[
|D\theta|_1
\quad\text{where}\quad
D\theta = (\theta_2-\theta_1,\dots,\theta_n-\theta_{n-1}).
\]
Trend filtering generalizes TVD by retaining the same penalized least squares
formulation as in~\eqref{eq:tvd} but replacing first-order differences by
$r$th-order differences for a degree $r\ge1$. Specifically, trend filtering (of order $r - 1$) is
defined as
\[
\hat{\theta}^{\mathrm{TF}}_{(r - 1),\lambda}
=
\arg\min_{\theta\in\mathbb{R}^n}
\frac12\sum_{i=1}^n (y_i-\theta_i)^2
+
\lambda\,|D^{(r)}\theta|_1,
\]
where $D^{(r)}$ denotes the $r$th-order discrete difference operator, defined
recursively by $D^{(r)} = D D^{(r-1)}$. Trend filtering of order $r - 1$, defined above, promotes sparsity in the
$r$th-order differences of the fitted signal; for example, when $r=2$, the
resulting estimator is piecewise linear and continuous. For general orders, it is well known that Trend Filtering
produces discrete splines that are piecewise polynomials with regularities at the knots; see, for example,
\cite{tibshirani2020divided, guntuboyina2020adaptive}.

The pointwise minmax/maxmin representation in Theorem~\ref{thm:ptwiseformula}
suggests a different route to higher-order generalization that does not start
from a global penalized least-squares program. The guiding idea is to retain the
same nested interval optimization structure as in \eqref{eq:fldefn}, but to
replace local averages by local polynomial fits of degree $r\ge 0$. This leads to a new class of locally adaptive
estimators, which we call \emph{Minmax Trend Filtering} (MTF).

To formalize this, let $I\subseteq[n]$ be an interval and write $y_I\in\R^{|I|}$
for the restriction of $y$ to $I$. Let $P^{(|I|,r)}$ denote the (unweighted)
least-squares projection matrix onto discrete polynomials of degree at most $r$ on the
grid $\{j/n: j\in I\}$, and let $(P^{(|I|,r)}y_I)_i$ denote the fitted value at
index $i\in I$.

\begin{definition}[Minmax trend filtering of degree $r$]\label{def1:mtf}
Fix $r\ge 0$ and $\lambda\ge 0$. For each $i\in[n]$, define the lower and upper
estimators
\begin{align}
L^{(r,\lambda)}_i
&:=
\max_{J \in \mathcal{I}: i \in J}
\min_{I \in \mathcal{I}: I \subseteq J,\, i \in I}
\Bigl[
(P^{(|I|,r)}y_I)_i + C_{I,J}\frac{2\lambda}{|I|}
\Bigr],
\label{eq:mtf_lower}\\
U^{(r,\lambda)}_i
&:=
\min_{J \in \mathcal{I}: i \in J}
\max_{I \in \mathcal{I}: I \subseteq J,\, i \in I}
\Bigl[
(P^{(|I|,r)}y_I)_i - C_{I,J}\frac{2\lambda}{|I|}
\Bigr],
\label{eq:mtf_upper}
\end{align}

where $C_{I,J}$ is the same as in Definition~\ref{def:cij}.

Then it turns out that $L^{(r,\lambda)}_i \le U^{(r,\lambda)}_i$ for all $i$ and we define the \emph{minmax trend filtering} estimator by
\begin{equation}\label{eq:mtf_def}
\hat{\theta}^{(r,\lambda)}_i
:=
\frac12\Bigl(L^{(r,\lambda)}_i + U^{(r,\lambda)}_i\Bigr),
\qquad i\in[n].
\end{equation}
\end{definition}

\begin{remark}
When $r=0$, the projection $P^{(|I|,0)}$ reduces to averaging, and
\eqref{eq:mtf_lower}--\eqref{eq:mtf_upper} coincide with \eqref{eq:fldefn}; hence
$\hat{\theta}^{(0,\lambda)}$ is the same as the TVD estimator. 
\end{remark}

For $r\ge 1$, the estimator \eqref{eq:mtf_def} is no longer defined via a convex
optimization program. Nevertheless, in Section~\ref{sec:well} we show that the
upper minmax estimator is no smaller than the lower maxmin estimator, i.e., $L^{(r,\lambda)}_i \le U^{(r,\lambda)}_i$ for all $i$, so the
definition is well posed.

Minmax Trend Filtering differs from classical Trend Filtering when
$r\ge1.$ Although not obvious from its formulation, it appears to retain some key features of locally
adaptive estimation. Leveraging the minmax/maxmin formulation, we derive
pointwise bias variance type decompositions for the local risk of MTF, extending
the analysis developed for TVD to arbitrary polynomial degrees. Using this
framework, we show that MTF achieves minimax-optimal global rates comparable to
those of Trend Filtering (Section~\ref{sec:global}). We also give a refined analysis of local rates of convergence that depend explicitly on the
local smoothness of the underlying regression function
(Section~\ref{sec:local}). Such local analyses of classical Trend Filtering seem difficult, due to the lack of a pointwise representation. These results illustrate how the local risk of MTF can adapt to the local smoothness. They also indicate how the local risk varies as a
function of the tuning parameter $\lambda$ and reveal a degree of robustness
to over-smoothing not shared by standard linear smoothers; see Section~\ref{subsec:riskcurves}.

From a computational standpoint,
a direct evaluation of \eqref{eq:mtf_lower}--\eqref{eq:mtf_upper} is expensive.
We therefore introduce a dyadic variant that restricts to intervals of dyadic
lengths, significantly reducing the complexity to $O(n (\log n)^2)$. This variant is efficiently computable and we have developed a R code that is publicly available; see Section~\ref{sec:simu}. Our numerical experiments in Section~\ref{sec:simu} show that this variant can even outperform classical trend filtering for signals with highly heterogeneous local smoothness levels. 

We now summarize the main contributions of this article.

\begin{enumerate}
\item \textbf{New Pointwise Formula for TVD.} We give a pointwise formula for the TVD estimator in terms of minmax or maxmin of a simple function of local averages.
In spite of a long history and substantial literature on analyzing TVD, this pointwise formula appears to be new.

\item \textbf{Well Posedness of the Minmax Formula.}
We recognize that the minmax/maxmin formula for TVD is well posed, significantly generalizable, and gives a new and interesting way to define other locally adaptive estimators.

\item \textbf{Minmax Trend Filtering.} We propose higher degree polynomial generalizations of TVD via the pointwise minmax/maxmin representation developed here. These estimators are in general different from trend filtering of order $r \geq 1$. These estimators appear to be new and combine the strengths of linear and nonlinear smoothers by admitting a pointwise representation and by being locally adaptive.

\item \textbf{Local Risk Analysis for general degrees.}
We give pointwise estimation errors for TVD and Minmax Trend Filtering (of any order $r \geq 1$) which are clearly interpretable as a tradeoff of (local) bias + (local) standard error. We show that the notion of (local) bias and (local) standard error tradeoff developed here is a stronger notion than the existing minimax rate optimality notions of local adaptivity usually cited for trend filtering; see Section~\ref{sec:global}. Additionally, we derive pointwise estimation error bounds for the entire risk function (as a function of $\lambda$) of MTF at a point where the underlying function is locally H\"older smooth with a given smoothness exponent. These local rates of convergence clearly reveal the optimal choice of the tuning parameter $\lambda$ and consequences for undersmoothing/oversmoothing. The proof technique is arguably simpler (than existing proofs for trend filtering); it does not rely on local entropy bounds as in~\cite{guntuboyina2020adaptive} or the notion of interpolating vectors as in~\cite{ortelli2021prediction}.
\end{enumerate}

\subsection{Notation}

For a positive integer $n$, we write $[n] := \{1,2,\dots,n\}$.  
For integers $1 \le a \le b \le n$, we write
\[
[a:b] := \{a,a+1,\dots,b\}.
\]
A (discrete) \emph{interval} of $[n]$ is any set of the form $I = [a:b]$.  
We denote by $|I| = b-a+1$ the cardinality of an interval $I$, and by
$\mathcal{I}$ the collection of all intervals of $[n]$.

For any vector $v \in \mathbb{R}^n$ and any interval $I \in \mathcal{I}$, we write
$v_I \in \mathbb{R}^{|I|}$ for the restriction of $v$ to the coordinates indexed
by $I$.

For two intervals $I,J \in \mathcal{I}$, we write $I \subset J$ to mean that $I$ is
strictly contained in the interior of $J$, i.e., $I$ does not include either
endpoint of $J$.  
We write $I \subseteq J$ when $I$ is an arbitrary subinterval of $J$, possibly
including one or both endpoints.

Throughout the paper, $r \ge 0$ denotes an integer polynomial degree.
We use $C_r$ to denote a generic positive constant depending only on $r$; its
value may change from line to line.

We use the term \emph{interval partition} to refer to a partition of $[n]$ into
contiguous discrete intervals.

\subsection{Outline}

The remainder of the paper is organized as follows. In Section~\ref{sec:unitvd}, we derive and explain the pointwise minmax/maxmin
representation for total variation denoising (TVD), including a proof sketch
highlighting the role of interval averaging and boundary effects. Section~\ref{sec:well} studies the general minmax/maxmin structure underlying our
representation and establishes a general well-posedness result. This observation
is essential for defining new estimators beyond TVD. In Section~\ref{sec:minmaxtf}, we introduce \emph{Minmax Trend Filtering} (MTF),
a family of locally adaptive estimators of arbitrary polynomial degree $r \ge 0$,
defined pointwise via nested interval optimizations. We also discuss computational
variants, including a dyadic symmetric restriction that leads to efficient
implementations. Section~\ref{sec:mainresult} presents our main theoretical result: a simultaneous
pointwise estimation error bound for MTF of general degree. This bound takes the
form of a local bias--variance tradeoff optimized over intervals. In Section~\ref{sec:local}, we investigate the local adaptivity of MTF.
We derive local rates of convergence at a point where the underlying regression
function is locally H\"older smooth and show how these rates depend explicitly on
the tuning parameter $\lambda$. Section~\ref{sec:global} studies global performance under mean squared error loss.
We show that MTF achieves near-minimax optimal global rates over classical function
classes, despite being defined pointwise rather than through a global optimization
problem. In Section~\ref{sec:simu}, we discuss the computational complexity and present numerical experiments illustrating the
finite-sample behavior of MTF and its dyadic variant, and we compare their
performance with classical trend filtering. Finally, Section~\ref{sec:discuss} discusses related work, extensions, and open
questions.

The supplementary material contains proofs of the main results.
In particular, Sections~\ref{sec:ptwiseformulaproof},
\ref{sec:mainthmproof}, and \ref{sec:localrateproofs} provide proofs of the main
theorems, while Sections~\ref{sec:fact1} and \ref{sec:fact2} collect auxiliary
lemmas and technical results.

\section{The minmax/maxmin representation for TVD}
\label{sec:unitvd}

In this section, we first provide a proof sketch of
Theorem~\ref{thm:ptwiseformula}. The goal is to convey the main ideas underlying
the proof and to explain the origin and interpretation of the constants $C_{I,J}$. 

\begin{enumerate}
\item \textbf{Optimality Conditions}

The TVD estimator is the unique solution of a strictly convex optimization
problem with an $\ell_1$ penalty on first differences. Its optimality conditions
imply the existence of a vector $z = (z_0,\ldots,z_n)$ satisfying
\begin{equation}\label{eq:dualconstr}
z_0 = z_n = 0, \qquad
|z_k| \le \lambda \ \text{for all } k, \qquad
z_k = \lambda \operatorname{sign}(\hat\theta_k - \hat\theta_{k+1})
\:\:\text{if}\:\: \hat\theta_{k+1} \neq \hat\theta_k
\end{equation}
together with the identity
\begin{equation}{\label{eq:ident}}
\hat\theta^{(\lambda)}_i = y_i + z_{i-1} - z_i,
\qquad i = 1,\ldots,n.
\end{equation}

The vector $z$ can be identified with the taut string associated with total variation
denoising: it is the cumulative residual process constrained to lie within
$[-\lambda,\lambda]$ and pinned at zero at both endpoints.

\item \textbf{Interval averaging identity.}

Averaging the identity~\eqref{eq:ident} over an arbitrary interval
$I = [a:b] \subseteq [n]$ telescopes the $z$ vector and yields 
\begin{equation}\label{eq:interval_identity_appendix_new}
\overline{\hat\theta^{(\lambda)}_I}
=
\bar y_I
-
\frac{z_b - z_{a-1}}{|I|}.
\end{equation}
Although elementary, this identity is the key driver of our proof. It shows
that the effect of total variation regularization on any local average is
entirely determined by boundary terms of the vector $z$. 

\item \textbf{Maximal Plateau, boundary analysis, and the emergence of
$C_{I,J}$.}

We first establish the upper bound in the min--max representation. Fix any
interval $J = [a:b] \ni i$. We show that
\begin{equation}\label{eq:minmaxineq}
\hat\theta^{(\lambda)}_i
\le
\max_{\substack{I \subseteq J \\ i \in I}}
\Bigl(
\bar y_I - \frac{2\lambda}{|I|} C_{I,J}
\Bigr).
\end{equation}

The key step is to identify a \emph{maximal interval where the fitted value is not below $\hat\theta^{(\lambda)}_i$} (or \emph{plateau})
within $J$. Specifically, let $I = [c:d] \subseteq J$ be the largest subinterval
of $J$ containing $i$ such that
\[
\hat\theta^{(\lambda)}_u \ge \hat\theta^{(\lambda)}_i
\qquad \text{for all } u \in I.
\]
It suffices to show that
\begin{equation}\label{eq:sufficeup}
\hat\theta^{(\lambda)}_i
\le
\bar y_I - \frac{2\lambda}{|I|} C_{I,J},
\end{equation}
which then implies~\eqref{eq:minmaxineq}.

By construction of $I$,
\[
\hat\theta^{(\lambda)}_i
\le
\overline{\hat\theta^{(\lambda)}_I}
=
\bar y_I - \frac{z_d - z_{c-1}}{|I|},
\]
where the equality follows from
\eqref{eq:interval_identity_appendix_new}. The boundary term
$z_d - z_{c-1}$ is controlled using the constraints in
\eqref{eq:dualconstr}. Its magnitude and sign depend on how the endpoints of $I$
interact with those of $J$, as well as on whether $J$ intersects the global
boundaries $\{1,n\}$.

For example, if $J \subseteq [2:(n-1)]$ and $I$ lies strictly in the interior of
$J$, then maximality of $I$ implies $z_d = \lambda$ and
$z_{c-1} = -\lambda$, so that $z_d - z_{c-1} = 2\lambda$, corresponding to
$C_{I,J} = 1$. A complete case-by-case analysis over all configurations of $I$
and $J$ yields the full collection of constants $C_{I,J}$.

\item \textbf{Exact min--max identity via the maximal valley.}

We next show that the upper bound is tight. That is, there exists an interval
$J = [a:b] \ni i$ such that
\begin{equation}\label{eq:minmaxeq}
\hat\theta^{(\lambda)}_i
\ge
\max_{\substack{I \subseteq J \\ i \in I}}
\Bigl(
\bar y_I - \frac{2\lambda}{|I|} C_{I,J}
\Bigr).
\end{equation}

This interval is chosen as a \emph{maximal interval where the fitted value is not above $\hat\theta^{(\lambda)}_i$} (or \emph{valley}). Specifically, let
$J = [a:b]$ be the largest interval containing $i$ such that
\[
\hat\theta^{(\lambda)}_u \le \hat\theta^{(\lambda)}_i
\qquad \text{for all } u \in J.
\]
For any subinterval $I = [c:d] \subseteq J$ containing $i$, we then have
\[
\hat\theta^{(\lambda)}_i
\ge
\overline{\hat\theta^{(\lambda)}_I}
=
\bar y_I - \frac{z_d - z_{c-1}}{|I|},
\]
where the inequality follows from the definition of $J$ and the equality again
from~\eqref{eq:interval_identity_appendix_new}. A case-by-case analysis shows
that
\[
z_{c-1} - z_d \ge -2\lambda\, C_{I,J},
\]
uniformly for all $I \subseteq J$. Importantly, and perhaps surprisingly, the same constants $C_{I,J}$
appear here as in the plateau argument, essentially since both analyses rely on the same constraints on the $z$ vector with reversed sign patterns. This establishes~\eqref{eq:minmaxeq} and completes
the proof of the min--max identity.

In fact, this argument shows that equality in the min--max formula is attained by taking $J^\star$ to be
the maximal valley containing $i$. For this choice, an inner maximizer is given
by the maximal plateau $I^\star\subseteq J^\star$ containing $i$, which
coincides with the constant fitted block containing $i$

\item \textbf{Symmetry of the TVD objective.}

Finally, the TVD objective is invariant under the transformation
$(\theta,y) \mapsto (-\theta,-y)$, implying
$\hat\theta^{(\lambda)}(-y) = -\hat\theta^{(\lambda)}(y)$. The maxmin
representation therefore follows immediately from the minmax identity applied
to $-y$, completing the proof of
Theorem~\ref{thm:ptwiseformula}.
\end{enumerate}

\medskip
\noindent

To summarize, the essential insight is the use of the interval averaging
identity~\eqref{eq:interval_identity_appendix_new} on two maximal intervals
(plateaus and valleys) to derive a \emph{min--max structure} for the fitted
values. The constants $C_{I,J}$ arise as boundary interaction coefficients,
encoding how the global total variation constraint translates into explicit,
interval-dependent corrections of local averages.

%


\subsection{Local Geometry of the TVD Fit}

The proof above reveals that equality in the min--max formula is achieved by
choosing $J^\star$ as the maximal valley containing $i$, and $I^\star$ as the
constant fitted block containing $i$. The fitted value admits the explicit
representation
\[
\hat\theta_i
=
\bar y_{I^\star}
-
\frac{2\lambda}{|I^\star|}\,C_{I^\star,J^\star},
\]
where $C_{I^\star,J^\star}$ depends only on the relative position of $I^\star$
within $J^\star$.

Assuming for simplicity that $I^\star$ and $J^\star$ are interior intervals,
three qualitatively distinct cases arise.

If $I^\star$ is a \emph{local maximum block}, then $C_{I^\star,J^\star}=1$ and
\[
\hat\theta_i
=
\bar y_{I^\star}-\frac{2\lambda}{|I^\star|},
\]
so the fitted value is a \emph{shrunk local average}.

If $I^\star$ is a \emph{local minimum block}, then $I^\star=J^\star$ and
$C_{I^\star,J^\star}=-1$, yielding
\[
\hat\theta_i
=
\bar y_{I^\star}+\frac{2\lambda}{|I^\star|},
\]
a \emph{lifted local average}.

Finally, if $I^\star$ lies in a monotone stretch and is not an extremum block,
then $C_{I^\star,J^\star}=0$ and
\[
\hat\theta_i=\bar y_{I^\star}.
\]

Thus, total variation denoising selectively shrinks towards local averages on extremum
blocks while leaving monotone regions unchanged. The amount of shrinkage scales
inversely with the block length, so smaller blocks are penalized more heavily.
In this sense, the coefficients $C_{I,J}$ can be thought of as controlling the local shrinkage factor;
they take values in $\{0,\pm1\}$ for interior intervals $J$, $\{0,\pm\frac12,1\}$
when $J$ contains one global boundary point and $\{0,\frac12,1\}$ when $J = [n].$
	
\section{Well Posedness}\label{sec:well}
In this section, we explore whether the minmax/maxmin expressions can give rise to more general estimators. \textit{It turns out that the maxmin expression never exceeds the minmax expression much more generally. This is what makes the definition of Minmax Trend Filtering in~\eqref{def1:mtf} well posed}. This is an important observation in the context of this article and we state and prove this in a general form as a separate proposition.

\begin{proposition}[Well Posedness]\label{prop:welldefn}
Fix any $i \in [n].$ Let $\mathcal{S} \subseteq \mathcal{I}$ be any non-empty class of (discrete) intervals of $[n]$ containing $i$ and closed under intersection. For any set function $f: \mathcal{S} \rightarrow \mathbb{R}$ and any non-negative set function $g: \mathcal{S} \rightarrow \mathbb{R}$, we have the following inequality:
\begin{equation}\label{eq:welldfn}
			\max_{J \in \mathcal{S}} \min_{I \in \mathcal{S}: I \subseteq J} \big[f(I)  + c_{I,J}\:g(I)\big] \leq \min_{J \in \mathcal{S}} \max_{I \in \mathcal{S}: I \subseteq J} \big[f(I)  - c_{I,J}\:g(I)\big]
		\end{equation}
        for any constants $c_{I,J}$ satisfying the following conditions:
        \begin{equation}\label{eq:CIJ_condition}
        c_{I,J_1} + c_{I,J_2} \le 0
        \qquad\text{for every } J_1,J_2\in\mathcal{S}\text{ and } I = J_1\cap J_2.
        \end{equation}
\end{proposition}

\begin{proof}[Proof of Proposition~\ref{prop:welldefn}]

For any $J \in \mathcal{S}$, let's define the two quantities
\[
LH(J) := \min_{I \in \mathcal{S}: I \subseteq J} \big[f(I)  + c_{I,J}\:g(I)\big],
\qquad
RH(J) :=  \max_{I \in \mathcal{S}: I \subseteq J} \big[f(I)  - c_{I,J}\:g(I)\big].
\]

To show that~\eqref{eq:welldfn} holds, it is equivalent to show that for any $J_1,J_2 \in \mathcal{S}$,
\begin{equation}\label{eq:req1}
 LH(J_1) \leq RH(J_2).
\end{equation}

The left hand side above is a minimum of a list of numbers (indexed by $I \in \mathcal{S}: I \subseteq J_1$) and the right hand side is a maximum of a list of possibly different numbers (indexed by $I \in \mathcal{S}: I \subseteq J_2$). To show~\eqref{eq:req1}, it suffices to show that one number is common in the two lists of numbers. The main observation is that we can always consider the number corresponding to $J_1 \cap J_2 \in \mathcal{S}$ which is common to both the lists. 

Let $I := J_1 \cap J_2 \in \mathcal{S}$. Since $I \subseteq J_1$ and $I \subseteq J_2$, we may use $I$ as a feasible choice in both $LH(J_1)$ and $RH(J_2)$ to obtain
\begin{align*}
LH(J_1)
&\le f(I) + c_{I,J_1}\,g(I),\\
RH(J_2)
&\ge f(I) - c_{I,J_2}\,g(I).
\end{align*}
By \eqref{eq:CIJ_condition} and $g(I)\ge 0$, we have $c_{I,J_1}\,g(I)\le -c_{I,J_2}\,g(I)$, and hence
\[
f(I) + c_{I,J_1}\,g(I) \le f(I) - c_{I,J_2}\,g(I).
\]
Combining the last three displays yields \eqref{eq:req1}.
\end{proof}

Now we show that the particular constants $C_{I,J}$ in~\eqref{def:cij}, used to define Minmax Trend
Filtering of a general order $r \geq 0$, satisfy the conditions
in~\eqref{eq:CIJ_condition}.

\begin{lemma}\label{lem:cij_condition}
Let $C_{I,J}$ be as in Definition~\ref{def:cij}. For any two intervals
$J_1,J_2\subseteq[n]$ and their intersection $I:=J_1\cap J_2$, (which is again an interval, possibly empty), we have
\begin{equation}\label{eq:cij_submod}
C_{I,J_1}+C_{I,J_2}\le 0,
\end{equation}
with the convention that $C_{\varnothing,J}:=0$.
\end{lemma}

\begin{proof}
If $I=\varnothing$ there is nothing to prove, so assume $I\neq\varnothing$.
Write $J_1=[a:b]$, $J_2=[c:d]$, and hence
\[
I=[\max\{a,c\}:\min\{b,d\}].
\]
There are two mutually exclusive possibilities: either the intervals are nested,
or they are not.

\smallskip\noindent
\textbf{Case 1 (nested intervals):} $J_1\subseteq J_2$ or $J_2\subseteq J_1$.
By symmetry it suffices to treat $J_1\subseteq J_2$, in which case $I=J_1$.
We must show $C_{J_1,J_1}+C_{J_1,J_2}\le 0$.

\smallskip\noindent
\emph{(1a) $J_1$ is an interior interval, i.e.\ $1<a\le b<n$.}
Then Definition~\ref{def:cij} gives $C_{J_1,J_1}=-1$.
Since $C_{J_1,J_2}\le 1$, we obtain $C_{J_1,J_1}+C_{J_1,J_2}\le 0.$

\smallskip\noindent
\emph{(1b) $J_1$ touches exactly one global boundary, i.e.\ $a=1<b<n$ or
$1<a\le b=n$.}
Then Definition~\ref{def:cij} gives $C_{J_1,J_1}=-\tfrac12$.
Depending on $J_2$, we have $C_{J_1,J_2}\in\{\pm \tfrac12\}$, and hence $C_{J_1,J_1}+C_{J_1,J_2}\le 0.$

\smallskip\noindent
\emph{(1c) $J_1=[1:n]$.}
Then Definition~\ref{def:cij} gives $C_{J_1,J_1}=0$, and since $J_1\subseteq J_2$
forces $J_2=[1:n]$, we also have $C_{J_1,J_2}=0$.

\smallskip
Thus \eqref{eq:cij_submod} holds in all nested cases.

\smallskip\noindent
\textbf{Case 2 (non-nested intervals):} neither $J_1\subseteq J_2$ nor
$J_2\subseteq J_1$.
Then $I=J_1\cap J_2$ is a strict subinterval of both $J_1$ and $J_2$. Equivalently,
$I$ touches a boundary point of $J_1$ and a boundary point of $J_2$.

We claim that in this configuration
\begin{equation}\label{eq:proper_intersection_zero}
C_{I,J_1}=C_{I,J_2}=0.
\end{equation}

If both $J_1$ and $J_2$ are strictly interior intervals, this follows immediately
from Definition~\ref{def:cij}. Otherwise, since neither $J_1$ nor $J_2$ can equal
$[1:n]$, without loss of generality the non-nested configuration must be of
the form
\[
J_1=[1:j_1], \qquad J_2=[j_2:n],
\qquad 1<j_2\le j_1<n.
\]
In this situation, the intersection $I=[j_2:j_1]$ does not share
a global boundary point with either $J_1$ or $J_2$. Since the value
$C_{I,J}=\tfrac12$ in Definition~\ref{def:cij} can occur only when $I$ shares a
global boundary with $J$, it follows that $C_{I,J_1}=C_{I,J_2}=0$, proving
\eqref{eq:proper_intersection_zero}.

Therefore \eqref{eq:cij_submod} holds in Case~2 as
well.
\end{proof}

	
\section{Minmax Trend Filtering of General Degree}
\label{sec:minmaxtf}

In this section, we develop higher--degree polynomial generalizations of
univariate total variation denoising (TVD), also known as the fused lasso,
through the pointwise minmax/maxmin representation introduced
in~\eqref{eq:fldefn}. The guiding principle is to retain the same nested
interval optimization structure as in TVD, while replacing local averages by
local polynomial fits of arbitrary degree.

We work throughout in the sequence model and use discrete polynomial sequences.
We begin by introducing the necessary notation.

Fix a nonnegative integer $r \ge 0$. Define the linear subspace of
$n$--dimensional \emph{discrete polynomial vectors} of degree at most $r$ by
\[
\mathcal{P}^{(r)}_n
:=
\Bigl\{
\theta \in \R^n :
(\theta_1,\dots,\theta_n) = (f(1/n),\dots,f(1))
\text{ for some polynomial $f$ of degree at most $r$}
\Bigr\}.
\]

Given an interval $I = [a:b] \subseteq [n]$, define the corresponding subspace of
degree--$r$ discrete polynomials on $I$ by
\[
S^{(I,r)}
:=
\bigl\{
\theta \in \R^{|I|} :
\theta = v_I \text{ for some } v \in \mathcal{P}^{(r)}_n
\bigr\}.
\]

Let $P^{(I,r)} \in \R^{|I| \times |I|}$ denote the orthogonal projection matrix
onto $S^{(I,r)}$. Since the design points are equally spaced, the subspace
$S^{(I,r)}$ depends on the interval $I$ only through its length $|I|$; we
therefore write $P^{(|I|,r)}$ throughout.

We now give a definition of the Minmax Trend Filtering estimator that is slightly
more general than~\eqref{def1:mtf}, allowing for computationally restricted
variants.

\begin{definition}[Minmax Trend Filtering (MTF) of general degree]
\label{defn:minmaxfgen}
Fix a nonnegative integer $r$. For each $i \in [n]$, let
$\mathcal{I}_i \subseteq \mathcal{I}$ be a collection of intervals such that
$i \in I$ for all $I \in \mathcal{I}_i$ and $\mathcal{I}_i$ is closed under
intersection. Given data $y \in \R^n$ and $\lambda \ge 0$, define an estimator
$\hat{\theta}^{(r,\lambda)} \in \R^n$ satisfying, for each $i \in [n]$,
\begin{equation}\label{eq:defn}
\max_{J \in \mathcal{I}_i}\ \min_{I \in \mathcal{I}_i:\, I \subseteq J}
\Bigl[
(P^{(|I|,r)} y_{I})_{i}
+ \frac{2\lambda\, C_{I,J}}{|I|}
\Bigr]
\;\le\;
\hat{\theta}^{(r,\lambda)}_i
\;\le\;
\min_{J \in \mathcal{I}_i}\ \max_{I \in \mathcal{I}_i:\, I \subseteq J}
\Bigl[
(P^{(|I|,r)} y_{I})_{i}
- \frac{2\lambda\, C_{I,J}}{|I|}
\Bigr],
\end{equation}
where $C_{I,J}$ is as in Definition~\ref{def:cij}.
\end{definition}

\begin{remark}
The difference between Definition~\ref{defn:minmaxfgen} and the earlier
definition~\eqref{def1:mtf} is that the interval class $\mathcal{I}_i$ is allowed
to depend on the location $i \in [n]$. This flexibility is required for the
dyadic variant introduced below.
\end{remark}

We refer to Section~\ref{sec:simu} for plots of the resulting estimators. We now
summarize several key aspects of Definition~\ref{defn:minmaxfgen}.

\begin{enumerate}
\item \textbf{Well--posedness.}
For each $i \in [n]$, the lower and upper bounds in~\eqref{eq:defn} are well
defined. This follows directly from Proposition~\ref{prop:welldefn} with
$\mathcal{S}=\mathcal{I}_i$, $f(I) = (P^{(|I|,r)} y_{I})_i$, and
$g(I)=2\lambda/|I|$.

\item \textbf{Non--uniqueness.}
The estimator $\hat{\theta}^{(r,\lambda)}$ is not uniquely determined by
\eqref{eq:defn}: any value between the two bounds is admissible. For
concreteness, we take the midpoint of the bounds as in~\eqref{def1:mtf}. All
theoretical results developed below apply uniformly to any admissible choice.

\item \textbf{Indexing.}
For $I=[a:b]\in\mathcal{I}_i$, the vector
$P^{(|I|,r)}y_I\in\R^{|I|}$ is indexed by the points in $I$. Accordingly,
$(P^{(|I|,r)}y_I)_i$ refers to its $(i-a+1)$--st coordinate.

\item \textbf{Full MTF.}
A canonical choice is $\mathcal{I}_i = \{I \in \mathcal{I}: i \in I\}$, which
yields the \emph{full} MTF estimator defined in~\eqref{def1:mtf}. For $r \ge 1$,
direct evaluation of~\eqref{eq:defn} in this case requires $O(n^5)$ operations
and is therefore computationally infeasible.

\item \textbf{Dyadic symmetric variant.}
To reduce computational cost, we restrict $\mathcal{I}_i$ to a dyadic symmetric
family of intervals. For $h \ge 0$, define
\[
[i \pm h] := [\max\{i-h,1\},\ \min\{i+h,n\}],
\]
and set
\begin{equation}\label{eq:dsmtfclass}
\mathcal{I}_i
:=
\{ \{i\} \} \cup \{ [i \pm 2^j] \subseteq [n] : j \in \mathbb{Z}_{\ge 0} \}.
\end{equation}
This collection is closed under intersection, since it forms a nested chain.
The resulting estimator is computable in $O(n(\log n)^2)$ time; see
Section~\ref{sec:simu}. We refer to this variant as \emph{Dyadic Symmetric Minmax
Trend Filtering} (DSMTF).

All pointwise risk bounds derived later apply to both the full MTF and DSMTF
estimators; also see Remark~\ref{rem:mtfvsdsmtf}.
\end{enumerate}

\section{Pointwise Estimation Error Bound for Minmax Trend Filtering}
\label{sec:mainresult}

In this section, we state our main pointwise estimation error bound for Minmax
Trend Filtering (MTF) of arbitrary degree. The bound holds
\emph{simultaneously} over all locations $i\in[n]$ and \emph{uniformly} over all
tuning parameters $\lambda>0$. It takes the form of an oracle inequality: for
each $i$, the estimation error is controlled by the best (data--independent)
interval $J\in\mathcal I_i$ balancing a local approximation (bias) term and a
$\lambda$--dependent stochastic term.

This result is the main technical input for our local analysis in
Section~\ref{sec:local} and for the global risk bounds derived in
Section~\ref{sec:global}.

\subsection{Local bias functionals}

Fix a signal $\theta^* \in \R^n$ and an integer $r \ge 0$. Let $i \in [n]$ be any
location and let $J \subseteq [n]$ be any interval with $J \in \mathcal{I}_i$.
We define the \emph{local positive} and \emph{local negative} $r$th--order bias
functionals associated with $J$ by
\[
Bias^{(r)}_{+}(i,J,\theta^*)
:=
\max_{I \in \mathcal{I}_i:\, I \subseteq J}
\bigl[(P^{(|I|,r)} \theta^*_{I})_i - \theta^*_i\bigr],
\]
and
\[
Bias^{(r)}_{-}(i,J,\theta^*)
:=
\min_{I \in \mathcal{I}_i:\, I \subseteq J}
\bigl[(P^{(|I|,r)} \theta^*_{I})_i - \theta^*_i\bigr].
\]

If the singleton interval $\{i\}$ belongs to $\mathcal{I}_i$ (as is the case for
both MTF and DSMTF), then
\[
Bias^{(r)}_{+}(i,J,\theta^*) \ge 0,
\qquad
Bias^{(r)}_{-}(i,J,\theta^*) \le 0.
\]

\subsection{Local variance term}

For any $i \in [n]$ and any interval $J \subseteq [n]$ with $J \in \mathcal{I}_i$,
define the local $r$th--order standard deviation term
\[
SD^{(r)}(i,J,\lambda)
=
C_r \sigma \sqrt{\log n}
\left(
\frac{\mathbf{1}(i \notin \{1,n\})}{\sqrt{Dist(i,\partial J)}}
+
\frac{1}{\sqrt{|J|}}
\right)
+
\frac{C_r \sigma^2 \log n}{\lambda}
+
\frac{2\lambda}{|J|}.
\]

Here, for an interval $J=[j_1:j_2]\subseteq[n]$, we denote its boundary points by
$\partial J=\{j_1,j_2\}$ and define
\[
Dist(i,\partial J)
=
\min\{\, i-j_1+1,\; j_2-i+1 \,\}.
\]

The standard deviation term $SD^{(r)}(i,J,\lambda)$ consists of three components.
The first term corresponds to the stochastic fluctuation of a local polynomial
fit on the largest symmetric interval about $i$ contained in $J$. The remaining
two terms capture the explicit dependence on the tuning parameter $\lambda$.
Optimizing these two terms over $\lambda$ yields a quantity of the same order as
the first term. Accordingly, the first term may be viewed as the optimally tuned
local stochastic fluctuation, while the remaining terms quantify the additional
variability induced by the choice of $\lambda$.

\begin{remark}
For Theorem~\ref{thm:maingeneral} to hold, one could equivalently define the local
standard deviation term as the maximum of the three components above. We use
their sum for notational simplicity.
\end{remark}

\subsection{Main simultaneous pointwise bound}

\begin{theorem}[Simultaneous Bias--Variance Tradeoff]
\label{thm:maingeneral}
Fix any degree $r \ge 0$. There exist constants $c>1$ and $C_r>0$ (depending only
on $r$ and $c$) such that, with probability at least $1-n^{-c}$, the following
bound holds for the MTF estimator simultaneously for all $i\in[n]$ and all
$\lambda>0$:
\begin{equation}
\label{eq:thmbound}
\max_{J \in \mathcal{I}_i}
\bigl(
Bias^{(r)}_{-}(i,J,\theta^*)
-
SD^{(r)}(i,J,\lambda)
\bigr)
\;\le\;
\hat{\theta}^{(r,\lambda)}_i - \theta^*_i
\;\le\;
\min_{J \in \mathcal{I}_i}
\bigl(
Bias^{(r)}_{+}(i,J,\theta^*)
+
SD^{(r)}(i,J,\lambda)
\bigr).
\end{equation}
\end{theorem}

\begin{itemize}

\item
The bound in~\eqref{eq:thmbound} provides deterministic upper and lower bounds on
the random estimation error that hold simultaneously over all locations
$i\in[n]$, with high probability. The bounds are uniform over all signals
$\theta^*$ and all $\lambda>0$. Since
$Bias^{(r)}_{+}(i,J,\theta^*)\ge0$ and
$Bias^{(r)}_{-}(i,J,\theta^*)\le0$, the right--hand side controls the positive
deviation of the estimator, while the left--hand side controls the negative
deviation.

\item
The result is nonasymptotic and takes the form of an oracle inequality: the error
is bounded by the optimal (over $J\in\mathcal{I}_i$) tradeoff between a local bias
and a local variance term. In particular, the estimator adapts implicitly to the
unknown local regularity of $\theta^*$.

\item
Because the bound is pointwise, it enables a precise analysis of local rates of
convergence, which we pursue in Section~\ref{sec:local}. Comparable pointwise
bounds are not available for classical trend filtering of general degree, likely
due to the absence of an exact pointwise representation.

\item
An attractive feature of the bound is that the dependence on $\lambda$ appears
only in the final two terms of $SD^{(r)}(i,J,\lambda)$ and is fully explicit. This
makes it possible to analyze the local risk as a function of $\lambda$,
including the effects of under-- and over--smoothing; see
Section~\ref{sec:local}.

\item
Comparable simultaneous pointwise bounds are not known for classical trend
filtering of order $r\ge1$. For TVD ($r=0$), the closest related result is
Theorem~1 of~\cite{zhang2023element}, which applies only to piecewise constant
signals with sufficiently long segments. In contrast, the bound here applies to
a substantially broader class of signals and to all degrees $r\ge0$.

\item
The bound also enables analysis at the boundary points $\{1,n\}$. For
$i\in\{1,n\}$, the leading term in $SD^{(r)}(i,J,\lambda)$ does not involve
$Dist(i,\partial J)$, preventing degeneration at the boundary. Consequently, we
obtain boundary consistency under local smoothness assumptions; see
Section~\ref{sec:local}. To the best of our knowledge, such boundary results are
not available for classical trend filtering.

\item
At a high level, this result provides new insight into the local adaptivity of
TVD/MTF. Local smoothness is captured through the bias--variance tradeoff above,
revealing a sense in which TVD/MTF can be more locally adaptive than canonical
linear smoothers; see Section~\ref{sec:local}. The result also recovers the known
minimax rates for trend filtering discussed in Section~\ref{sec:msebds}.

\end{itemize}
	
\section{Local Rates}\label{sec:local}

In this section, we explore concrete consequences of the simultaneous pointwise error bound
established in Theorem~\ref{thm:maingeneral}. In particular, we investigate the local rate of
convergence of TVD and MTF at points where the true signal $\theta^*$ is locally H\"older smooth.
Throughout this section, we view
\[
\theta^*_i = f^*\!\left(\frac{i}{n}\right)
\]
as evaluations of an underlying function $f^*:[0,1]\to\mathbb{R}$ on the equally spaced grid
$\{i/n:1\le i\le n\}$.

We now formally introduce the H\"older class of functions.

\begin{definition}[H\"older space for functions]\label{def:holder}
Given any subinterval $\mathbf I\subseteq[0,1]$, $\alpha\in[0,1]$, and an integer $r\ge0$, we
define the H\"older space $C^{r,\alpha}(\mathbf I)$ as the class of functions
$f:[0,1]\to\mathbb{R}$ that are $r$ times continuously differentiable on $\mathbf I$ and whose
$r$th derivative is H\"older continuous with exponent $\alpha$, that is,
\begin{equation}\label{eq:Holder}
|f|_{\mathbf I;r,\alpha}
\stackrel{{\rm def.}}{=}
\sup_{x\neq y\in\mathbf I}
\frac{|f^{(r)}(x)-f^{(r)}(y)|}{|x-y|^\alpha}
<\infty.
\end{equation}
We call $|f|_{\mathbf I;r,\alpha}$ the $(r,\alpha)$-H\"older coefficient (or norm) of $f$ on
$\mathbf I$. If~\eqref{eq:Holder} holds for some $\alpha>1$, then necessarily
$|f|_{\mathbf I;r,\alpha}=0$, implying that $f^{(r)}$ is constant and hence that $f$ is a
polynomial of degree $r$ on $\mathbf I$. For notational continuity, we denote this case by
$C^{r,\infty}(\mathbf I)$ and set $|f|_{\mathbf I;r,\infty}=0$.
\end{definition}

We are now ready to state our local risk result. 

\begin{theorem}[Local Risk Bound]\label{thm:mainada}
Fix a degree $r\in\mathbb{N}$ and a function $f^*:[0,1]\to\mathbb{R}$. There exist constants
$c>1$ and $C_r$ (the same as in Theorem~\ref{thm:maingeneral}) such that the following holds with
probability at least $1-n^{-c}$, on the same event as in Theorem~\ref{thm:maingeneral}.
Simultaneously for all quadruples $(i_0,s_0,r_0,\alpha_0)$ with $i_0\in[n]$, $s_0>0$,
$r_0\in\{0,\dots,r\}$, and $\alpha_0\in[0,1]\cup\{\infty\}$ such that
$f^*\in C^{r_0,\alpha_0}([\tfrac{i_0\pm s_0}{n}])$, we have, with $\beta=r_0+\alpha_0$,
\[
\bigl|
\hat{f}^{(r,\lambda)}\!\left(\tfrac{i_0}{n}\right)
-
f^*\!\left(\tfrac{i_0}{n}\right)
\bigr|
\;\le\;
C_r\!\left(\frac{\tilde{\sigma}^2}{\lambda}+\frac{\lambda}{B_n}\right),
\]
where $\hat{f}^{(r,\lambda)}(\tfrac{i_0}{n})=\hat{\theta}^{(r,\lambda)}_{i_0}$ denotes the MTF/DSMTF
estimator defined in~\eqref{defn:minmaxfgen},
\[
B_n
=
\min\!\left\{
\tilde{\sigma}^{2/(2\beta+1)}
|f^*|_{[\frac{i_0\pm s_0}{n}];r_0,\alpha_0}^{-2/(2\beta+1)}
n^{2\beta/(2\beta+1)},
\;
|[i_0\pm s_0]|
\right\},
\qquad
\tilde{\sigma}
=
\sigma\sqrt{\log n}.
\]
\end{theorem}

We briefly sketch the proof of the above result below and highlight why the resulting bias variance
tradeoff differs fundamentally from that of standard linear smoothers. The full proof is given in Section~\ref{sec:localrateproofs}.
\begin{enumerate}
\item \textbf{Pointwise bias--variance reduction.}
The proof starts from the simultaneous pointwise bound in Theorem~5.1, which expresses the
estimation error at a location $i_0$ as the optimal tradeoff, over all intervals $J\ni i_0$,
between a local bias term and a variance term that depends explicitly on the
tuning parameter $\lambda$.

\item \textbf{Standard Bias behavior.}
When $f^*$ is locally H\"older smooth with exponent $\beta=r_0+\alpha_0$, Lemma~\ref{lem:bias} shows that
the local bias over an interval $J$ is of order $(|J|/n)^\beta$, matching the approximation
error of standard local polynomial regression. 

\item \textbf{Different structure of Variance Term.}
The (nonstandard) variance term for MTF consists of three components. The first coincides with the
variance term of standard local polynomial regression, while the dependence on $\lambda$
enters through the remaining two components. As a result, the overall bound initially
contains four terms: two corresponding to the classical bias--variance tradeoff of local
polynomial regression (which are independent of $\lambda$), and two additional
$\lambda$-dependent terms. Optimizing over $J\ni i_0$ using the classical bias variance
tradeoff yields the optimal local scale, after which the $\lambda$-dependent components dominate.
This reduction explains why the final bound involves only two terms and why the resulting
risk curve for MTF differs qualitatively from that of linear smoothers; see
Section~\ref{subsec:riskcurves} for more on this. 
\end{enumerate}

We now discuss several other aspects of the above theorem.

\begin{itemize}

\item
The theorem provides a simultaneous bound on the local estimation error of the TVD and MTF
estimators at points where the underlying signal is locally H\"older smooth. A notable feature
is that the bound holds for all $\lambda > 0$ and depends explicitly on $\lambda$, making
the behavior of the local risk transparent.

\item
If $\theta^*$ is locally H\"older smooth at $i_0$, with $i_0$ lying in the interior of an interval of
positive length, we may take $s_0=O(n)$. Ignoring constants and logarithmic factors, the bound
simplifies to
\begin{equation}\label{eq:bdsimpl}
\bigl|
\hat{f}^{(r,\lambda)}(\tfrac{i_0}{n}) - f^*(\tfrac{i_0}{n})
\bigr|
\;\lesssim\;
\frac{1}{\lambda}
+
\frac{\lambda}{n^{2\beta/(2\beta+1)}}.
\end{equation}

\item
The estimator degree $r$ is user-specified. Once fixed, MTF achieves near-optimal rates for any
local smoothness level $(r_0,\alpha_0)$ with $r_0\le r$. In particular, if
$f^*\in C^{r_0,\alpha_0}([0,1])$ globally, choosing
$\lambda=\tilde{O}(n^{\beta/(2\beta+1)})$ yields the minimax-optimal rate
$\tilde{O}(n^{-\beta/(2\beta+1)})$ up to logarithmic factors.

\item
The case $\alpha_0=\infty$ corresponds to $f^*$ being locally polynomial of degree at most
$r_0$. In this setting $\beta=\infty$, and choosing $\lambda=\tilde{O}(\sqrt{n})$ yields the
near-parametric rate $\tilde{O}(n^{-1/2})$, provided $|[i_0\pm s_0]|=O(n)$. For example, if $r = 2$ and 
$\theta^*$ is either locally constant, linear or quadratic, MTF achieves near-parametric rates.

\item
To the best of our knowledge, even for TVD ($r=0$), local rates under general local H\"older
smoothness assumptions have not previously been established.

\item
The theorem allows for spatially varying smoothness: the local smoothness parameters
$(r_0,\alpha_0)$ may vary with $i_0$.

\item
Optimal tuning of $\lambda$ for the location $i_0$ depends on the local smoothness exponent
$\beta=r_0+\alpha_0$; see~\eqref{eq:lambda_star} below. 

\item
The bound remains valid at the boundary points $i_0\in\{1,n\}$. Consequently, if the signal is
locally H\"older smooth at the boundary, both TVD and MTF are consistent there (with appropriate
tuning). To the best of our knowledge, such boundary consistency results are not available for
classical trend filtering.

\end{itemize}

\subsection{Risk curve comparison with linear smoothers}
\label{subsec:riskcurves}

We now use the simplified local bound in~\eqref{eq:bdsimpl} to explain why MTF can exhibit
stronger \emph{local adaptivity} than canonical linear smoothers when a single global tuning
parameter is used. The key point is not that MTF achieves better locally optimal rates—which coincide with those
of linear smoothers—but that its \emph{risk curve} deteriorates substantially more slowly under
oversmoothing.

Fix an integer degree $r\ge0$ and an exponent $\alpha\in[0,1]\cup\{\infty\}$, and set
$\beta=r+\alpha$. Consider estimating $f^*(x_0)$ at a point $x_0$ where $f^*$ is locally
$C^{r,\alpha}$ on an interval of positive length containing $x_0$. Interpreting the pointwise bound
in~\eqref{eq:bdsimpl} as a rate-level proxy for the local estimation error, we obtain the
following proxy pointwise risk curve for $r$th-order MTF:
\begin{equation}\label{eq:risk_mtf}
R^{\mathrm{MTF}}_n(\lambda)
\;:=\;
\frac{1}{\lambda}
\;+\;
\frac{\lambda}{n^{2\beta/(2\beta+1)}},
\qquad \lambda\in[1,\sqrt{n}].
\end{equation}
The proxy risk is minimized at
\begin{equation}\label{eq:lambda_star}
\lambda^\ast = n^{\beta/(2\beta+1)},
\end{equation}
with minimum value of order $n^{-\beta/(2\beta+1)}$.

For comparison, consider a canonical linear smoother such as local polynomial (LP) regression of
degree $r$ with a box kernel. It is well known that its standard pointwise bias--variance tradeoff has the form
\begin{equation}\label{eq:risk_lp_band}
\Bigl(\frac{|J|}{n}\Bigr)^\beta + \frac{1}{\sqrt{|J|}},
\end{equation}
where $|J|$ denotes the discrete bandwidth length. The optimal bandwidth satisfies
$|J|^\ast\asymp n^{2\beta/(2\beta+1)}$, yielding the same optimal rate
$n^{-\beta/(2\beta+1)}$.

To compare risk curves as functions of a single tuning parameter, we reparametrize the linear
smoother bound by identifying $\lambda$ with a square-root bandwidth,
\begin{equation}\label{eq:reparam}
|J|\asymp\lambda^2.
\end{equation}
Under this correspondence, the proxy risk curve for the linear smoother becomes
\begin{equation}\label{eq:risk_lp}
R^{\mathrm{LP}}_n(\lambda)
\;:=\;
\frac{1}{\lambda}
\;+\;
\Bigl(\frac{\lambda^2}{n}\Bigr)^\beta,
\qquad \lambda\in[1,\sqrt{n}].
\end{equation}
The proxy risks~\eqref{eq:risk_mtf} and~\eqref{eq:risk_lp} share the same minimizer and the same
minimum order, but their behavior away from $\lambda^\ast$ differs qualitatively.

\begin{itemize}

\item
\textbf{Undersmoothing.}
If $\lambda<\lambda^\ast$, the term $1/\lambda$ dominates both proxy risks. Consequently, in the
undersmoothing regime, MTF and linear smoothers behave similarly at the level of rates.

\item
\textbf{Oversmoothing.}
If $\lambda>\lambda^\ast$, the remaining term dominates. In this regime, the proxy risk for MTF
behaves like $\lambda\,n^{-2\beta/(2\beta+1)}$, whereas the proxy risk for a linear smoother
behaves like $(\lambda^2/n)^\beta$. A direct comparison of these two terms yields the following result which we record as a lemma. 

\end{itemize}

\begin{lemma}\label{lem:oversmooth_compare}
Let $\lambda^\ast$ be as defined in~\eqref{eq:lambda_star}. Suppose that $\lambda>\lambda^\ast$.
Then
\[
R^{\mathrm{LP}}_n(\lambda)
\begin{cases}
< R^{\mathrm{MTF}}_n(\lambda), & \text{if } \beta<\tfrac12,\\[4pt]
= R^{\mathrm{MTF}}_n(\lambda), & \text{if } \beta=\tfrac12,\\[4pt]
> R^{\mathrm{MTF}}_n(\lambda), & \text{if } \beta>\tfrac12.
\end{cases}
\]
\end{lemma}

The threshold $\beta=\tfrac12$ thus marks the transition at which the oversmoothing risk of
linear smoothers becomes more severe than that of MTF. Except for extremely rough local smoothness levels (i.e., $\beta<\tfrac12$), MTF is more
robust to oversmoothing in the sense that its risk curve deteriorates more slowly as
regularization increases. 

\begin{remark}[On proxy risks]\label{rem:proxy}
The quantities $R^{\mathrm{MTF}}_n(\lambda)$ and $R^{\mathrm{LP}}_n(\lambda)$ are derived from
upper bounds and should be interpreted as rate-level proxies for the true pointwise risk.
Moreover, for linear smoothers, the proxy curve is actually tight (in order) in a worst-case
sense over H\"older classes $C^{r,\alpha}.$ It is not hard to see that the variance term $\frac{1}{\sqrt{|J|}}$ matches the true stochastic fluctuation. The main issue is whether the bias term $\Bigl(\frac{|J|}{n}\Bigr)^\beta$ is tight. We prove that this is so; see Lemma~\ref{lem:biastight}. There exist functions for which the bias term attains order
$(|J|/n)^\beta.$ Let us fix $\alpha = 1$ for concreteness. Then this follows essentially because it is well known that for a local polynomial estimator of degree $r$ at an interior point $x_0$, assuming
$f \in C^{r,1}$ in a neighborhood of $x_0$, the bias of local polynomial regression (for bandwidth $h$) satisfies
\[
\operatorname{Bias}
=
\frac{f^{(r+1)}(x_0)}{(r+1)!}\, h^{r+1} C_{r,K}
\;+\; o(h^{r+1}),
\]
where the constant $C_{r,K}$ depends only on the kernel $K$ and the
polynomial degree $r$ and is typically non zero. In particular, local polynomial regression eliminates bias
up to order $r$, yielding the rate $\operatorname{Bias} = O(h^{r+1})$;~\citep{FanGijbels1996}. We give a self contained and rigorous proof of this fact adapted to our specific setting in Lemma~\ref{lem:biastight}, for $\alpha = 1$ and $0 \leq r \leq 10$. Combined with the fact that the proxy curve~\eqref{eq:risk_mtf} is an upper bound to the true risk for TVD/MTF, the above discussion actually implies that the
worst-case (over $C^{r,\alpha}$) true risk of MTF degrades more slowly (in rate, in the oversmoothing regime)
than the worst-case true risk of local polynomial regression.
\end{remark}

MTF can be interpreted as a form of local polynomial regression with data driven,
location adaptive bandwidth selection. The above discussion indicates that an advantage of MTF over vanilla local polynomial
regression lies not in improved locally optimal rates, but in greater robustness to
over-regularization. This distinction becomes important when a single global tuning parameter is
used across a signal with spatially varying smoothness.

To illustrate this mechanism, consider a heterogeneous smoothness setting with two local
smoothness indices. Fix $r\ge0$ and $\alpha_1,\alpha_2\in[0,1]\cup\{\infty\}$, define
$\beta_1=r+\alpha_1$ and $\beta_2=r+\alpha_2$, and suppose $\beta_1<\beta_2$ with
$\min\{\beta_1,\beta_2\}\ge\tfrac12$. The locally optimal tuning parameters then satisfy
\[
\lambda_1^\ast \asymp n^{\beta_1/(2\beta_1+1)}
\;<\;
\lambda_2^\ast \asymp n^{\beta_2/(2\beta_2+1)}.
\]
Any fixed global choice of $\lambda$ either undersmooths everywhere or must oversmooth in at
least one region. In the oversmoothing regime, Lemma~\ref{lem:oversmooth_compare} implies that
MTF can incur substantially less local error than a linear smoother, providing a mechanism for
local adaptivity under global tuning.

\subsubsection{Example and simulation evidence}\label{subsec:simulations}

To illustrate the preceding discussion, we specialize to the case $r=0$, so that MTF
coincides with total variation denoising (TVD) and the linear smoother reduces to local
averaging. Consider the function
\[
f_{\text{twohalves}}(x)
= 2(x - 1/2)^2 \mathbf{1}_{\{x > 1/2\}},
\quad x \in [0,1].
\]
This function is constant on $[0,1/2]$, corresponding to $\beta_1=\infty$, and Lipschitz on
$[1/2,1]$, corresponding to $\beta_2=1$. The locally optimal tuning parameters therefore scale
as
\[
\lambda_1^\ast \asymp n^{1/2},
\qquad
\lambda_2^\ast \asymp n^{1/3}.
\]

Choosing $\lambda$ to be optimal for the smoother region (the left half) leads to pronounced
oversmoothing in the rougher region (the right half), where the preceding analysis predicts
that the risk of TVD deteriorates far more gracefully than that of local averaging. Conversely,
choosing $\lambda$ to be optimal for the rougher region leads to undersmoothing in the smoother
region, where TVD is not expected to perform worse than local averaging.

We simulate noisy observations with sample size $n=900$, i.i.d.\ Gaussian noise $N(0,1)$,
signal-to-noise ratio $3$, and $50$ Monte Carlo replications. Figure~\ref{fig:curve_twohalves}
displays the estimated root mean squared error (RMSE) as a function of $\lambda$. Consistent
with the theory, local averaging exhibits a pronounced U-shaped risk curve, while the TVD risk
curve remains comparatively flat beyond its minimizer.

\begin{figure}[!th]
\centering
\includegraphics[scale=0.5]{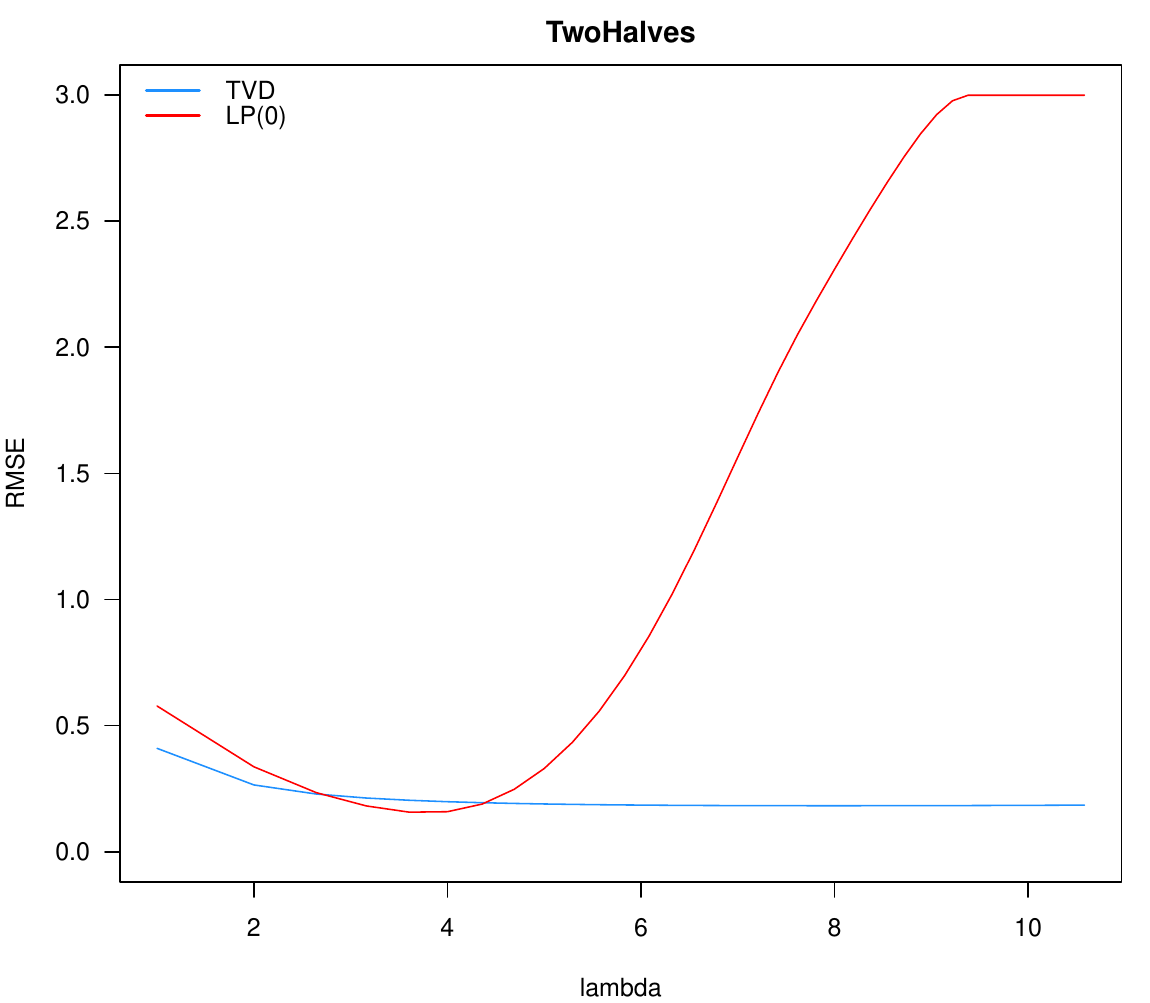}
\caption{Estimated RMSE as a function of $\lambda$ for the \texttt{twohalves} signal,
comparing TVD (blue) and local polynomial regression of order $0$ or local averaging (red).}
\label{fig:curve_twohalves}
\end{figure}

We repeat the same experiment for the standard test signals \texttt{Blocks}, \texttt{Bumps},
\texttt{HeaviSine}, and \texttt{Doppler}, used in Section~\ref{sec:simu}. The results, shown in
Figure~\ref{fig:curve_compare}, exhibit the same qualitative behavior: across signals with
spatially varying smoothness, the risk curve of TVD deteriorates far more slowly under
oversmoothing than that of local averaging.

\begin{figure}[!th]
\centering
\includegraphics[scale=0.6]{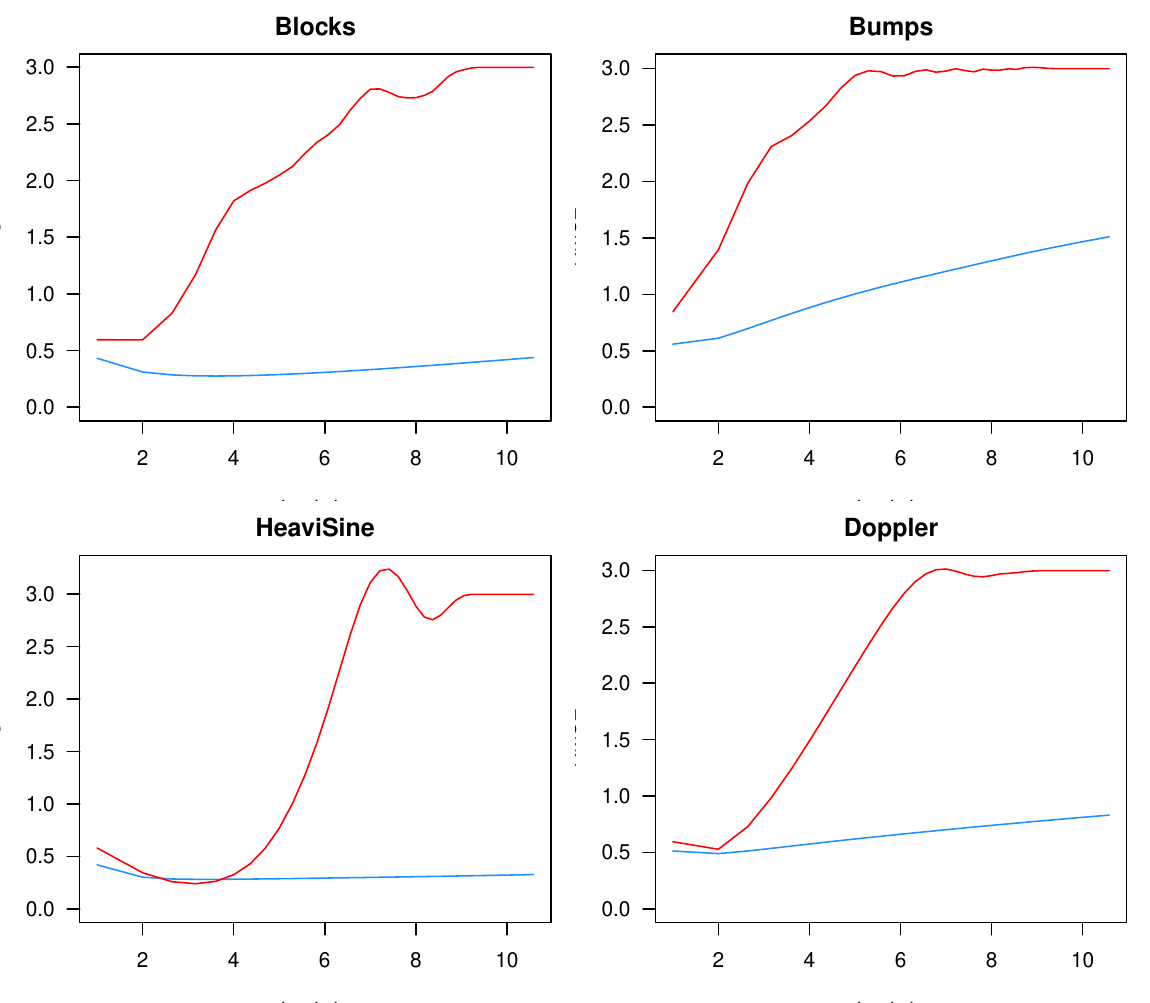}
\caption{Risk curves (RMSE) for TVD and local averaging as a function of $\lambda$ for signals with
spatially varying smoothness: \texttt{Blocks} (top left), \texttt{Bumps} (top right),
\texttt{HeaviSine} (bottom left), and \texttt{Doppler} (bottom right). In all cases, TVD
deteriorates far more gracefully under oversmoothing, consistent with the risk-curve analysis
of Section~\ref{subsec:riskcurves}. 
\label{fig:curve_compare}
}
\end{figure}


\subsubsection{Key takeaway}

To summarize, the local adaptivity of TVD and MTF does not seem to arise from strictly
improved locally optimal rates, which in fact coincide with those achievable by linear
smoothers. Instead, our findings suggest that their adaptivity may be more naturally
interpreted through differences in the geometry of the associated risk curves. In
particular, once oversmoothing occurs, the estimation error of TVD and MTF
increases more slowly than that of linear smoothers. This comparatively greater robustness
to over-regularization may help explain why a single global tuning parameter can still
perform reasonably well across regions with heterogeneous smoothness. This interpretation
is consistent with both our theoretical results and the empirical patterns observed in our
simulations.

This viewpoint differs from classical explanations based on minimax risk analyses under
global loss criteria such as mean squared error, which are often used to motivate why
standard linear smoothers are not locally adaptive whereas trend filtering is
(see, e.g.,~\cite{donoho1998minimax,sadhanala2016total,sadhanala2017higher}). 
\section{Global Rates}\label{sec:msebds}\label{sec:global}

Minmax Trend Filtering (MTF) is defined
pointwise via local minmax/maxmin constructions, in contrast to classical Trend Filtering, which arise as solutions to global convex optimization
problems. As demonstrated in the previous sections, this local definition
renders the analysis of pointwise estimation error particularly transparent.
A natural and nontrivial question is whether such locally defined estimators
retain optimal performance under global loss functions such as mean squared
error (MSE).

In this section, we answer this question in the affirmative. We show that the
simultaneous pointwise oracle inequality established in
Theorem~\ref{thm:maingeneral} is sufficiently sharp to recover near minimax
rate–optimal bounds in global MSE over two canonical function classes:
piecewise polynomial sequences and bounded variation sequences of arbitrary
order. These global guarantees coincide, up to logarithmic factors, with known minimax rates for classical trend filtering. Our contribution is to show
that the same guarantees for TVD/MTF follow directly from a purely pointwise analysis,
without appealing to a global optimization formulation or entropy-based
arguments.

All results in this section are proved for the full MTF estimator. The same
arguments extend directly to the dyadic symmetric variant (DSMTF), we omit these extensions for brevity.

\subsubsection{Fast rates for piecewise polynomial signals}

We begin by establishing fast rates under structural sparsity assumptions.
Throughout, $C_r$ denotes a positive constant depending only on the polynomial
degree $r\ge 0$, whose value may change from line to line. By an
\emph{interval partition} of $[n]$ we mean a partition into contiguous discrete
intervals.

\begin{theorem}[Fast rate for piecewise polynomial signals]\label{thm:fast}
Suppose there exists an interval partition $\pi^*$ of $[n]$ with intervals
$I_1,I_2,\dots,I_k$ such that $\theta^*_{I_j}$ is a discrete polynomial of degree
$r\ge 0$ (in the sense of $\mathcal P_n^{(r)}$ restricted to $I_j$) for each $j=1,\dots,k.$ Assume additionally a minimum length condition
\[
\min_{j\in[k]} |I_j| \;\ge\; c_1\,\frac{n}{k}
\]
for some absolute constant $c_1>0$.

If the tuning parameter is chosen as
\[
\lambda
\;=\;
C_r\Bigl(\frac{n\sigma^2\log n}{k}\Bigr)^{1/2},
\]
then, with probability at least $1-n^{-c}$ (for some $c > 1$) on the same event as in
Theorem~\ref{thm:maingeneral},
\[
\frac{1}{n}\,\|\hat{\theta}^{(r,\lambda)}-\theta^*\|^2
\;\le\;
C_r\,\sigma^2\,\frac{k}{n}\,\log n\,\log\frac{n}{k}.
\]

Moreover, for the $\ell_1$ loss the minimum length condition is not required:
with the same choice of $\lambda$,
\[
\frac{1}{n}\sum_{i=1}^n
\bigl|\hat{\theta}^{(r,\lambda)}_i-\theta^*_i\bigr|
\;\le\;
C_r\,\sigma\,\sqrt{\frac{k\log n}{n}}
\]
holds on the same high–probability event.
\end{theorem}

\begin{remark}
Theorem~\ref{thm:fast} is reminiscent of the fast rates attained by ideally
tuned trend filtering for discrete splines under a minimum length condition;
see, for example, \cite{ortelli2021prediction,guntuboyina2020adaptive}. Our proof relies exclusively on the pointwise
oracle inequality of Theorem~\ref{thm:maingeneral}, rather than on global convex
analysis or entropy bounds.
\end{remark}

\begin{remark}
Our argument applies uniformly to MTF of all polynomial degrees $r\ge 0$. By contrast, fast rates for penalized trend filtering have only been established
for $r\le 4$ to date \cite{ortelli2021prediction}.
\end{remark}

\begin{remark}
     To the best of our knowledge,
the fast-rate statement under $\ell_1$ loss without a minimum-length condition is new and has not
been stated explicitly even for TVD. It is not known whether such a result holds for Trend Filtering of general orders.  
\end{remark}

\begin{remark}
Classical trend filtering produces discrete splines, effectively imposing continuity or
higher--order smoothness constraints at the estimated knots. In contrast,
Theorem~\ref{thm:fast} does not rely on such regularity assumptions: the underlying signal is
allowed to be discontinuous and piecewise polynomial. In these settings, standard
consistency guarantees for trend filtering are not expected to hold, whereas MTF is still able to attain fast rates under the conditions of the
theorem. This distinction highlights a potential advantage of the minmax construction in
such regimes and is supported by the numerical results in Section~\ref{sec:simu}.
\end{remark}

\begin{remark}
The factors $\log(n/k)$ in the MSE bound and $\sqrt{\log n}$ in the $\ell_1$ bound
are likely artifacts of the proof technique. We do not pursue their optimality
here, as our primary goal is to establish rate–level optimality and to
demonstrate that global guarantees follow from local analysis.
\end{remark}

\subsubsection{Slow rates for bounded variation signals}

We now turn to global rates over bounded variation classes. For $\theta\in\mathbb
R^n$, recall the discrete difference operators $D^{(0)}(\theta)=\theta$,
$D^{(1)}(\theta)=(\theta_2-\theta_1,\dots,\theta_n-\theta_{n-1})$, and recursively
$D^{(r)}(\theta)=D^{(1)}(D^{(r-1)}(\theta))$ for $r\ge 2$. We define the $r$th–order
total variation by
\[
\TV^{(r)}(\theta)
\;=\;
n^{\,r-1}\,\|D^{(r)}(\theta)\|_1.
\]

\begin{remark}[Normalization]
The factor $n^{r-1}$ follows the convention of
\cite{guntuboyina2020adaptive}. When $\theta_i=f(i/n)$ for a sufficiently smooth
function $f$, $\TV^{(r)}(\theta)$ corresponds to a Riemann approximation of
$\int_0^1 |f^{(r)}(t)|\,dt$, and is therefore $O(1)$ for typical signals.
\end{remark}

\begin{theorem}[Slow rate for bounded variation signals]\label{thm:slow}
Fix an integer $r\ge 1$ and let $V=\TV^{(r)}(\theta^*)$. If
\[
\lambda
\;=\;
C_r\,
n^{r/(2r+1)}\,
V^{-1/(2r+1)}\,
\sigma^{1+1/(2r+1)}\,
(\log n)^{1/2+1/(2r+1)},
\]
then, with probability at least $1-n^{-c}$ on the same event as in
Theorem~\ref{thm:maingeneral},
\[
\frac{1}{n}\,
\|\hat{\theta}^{(r-1,\lambda)}-\theta^*\|^2
\;\le\;
C_r\,
n^{-2r/(2r+1)}\,
V^{2/(2r+1)}\,
\bigl(\sigma^2(\log n)^2\bigr)^{2r/(2r+1)}.
\]
\end{theorem}

\begin{remark}
Theorem~\ref{thm:slow} shows that Minmax Trend Filtering of order $r-1$ is near
minimax rate optimal over $r$th–order bounded variation classes, with the correct
dependence on $V$ and $n$ up to logarithmic factors. The proof proceeds by
approximating bounded variation sequences by piecewise polynomial sequences and
then invoking the pointwise oracle inequality; see
Proposition~\ref{prop:piecewise}.
\end{remark}

\begin{remark}
Theorems~\ref{thm:fast} and~\ref{thm:slow} together show that Minmax Trend
Filtering achieves near minimax–optimal global rates over the two canonical
signal classes used to justify the adaptivity of trend filtering. Importantly,
these results are obtained as direct consequences of the pointwise bound in
Theorem~\ref{thm:maingeneral}. This illustrates that the local bias–variance
tradeoff developed in this paper is a unified framework from which both local and global
optimality (up to log factors) follow. 
\end{remark}

\section{Computation and Simulations}\label{sec:simu}

The dyadic symmetric min--max trend filtering (DSMTF) estimator can be computed
efficiently in near-linear $O(n (\log n)^2)$ time for fixed polynomial degree $r$ and a fixed grid
of tuning parameters.
Crucially, DSMTF evaluates projections only over the dyadic family
$\mathcal I_i$ at each location $i \in [n]$, where $|\mathcal I_i| = O(\log n)$.
A naive implementation that precomputes projection quantities for all intervals
$I \subseteq [n]$ would incur $O(n^2)$ time and memory costs.
This is unnecessary for DSMTF, since only $O(n \log n)$ dyadic intervals are ever
queried.
We now describe an implementation that exploits this structure.

The computation proceeds in two stages: a preprocessing step independent of the
tuning parameter $\lambda$, followed by pointwise aggregation.

\begin{itemize}

\item \textbf{Precomputation.}

We precompute prefix sums that allow us to evaluate local polynomial projections
for any interval $I = [a:b]$ using constant-time access (up to
$r$-dependent matrix algebra), without enumerating all $O(n^2)$ intervals.
Let $B^{(r,I)}$ denote the discrete polynomial basis matrix of degree $r$
and size $|I| \times (r+1)$ associated with an interval $I$.
By standard least squares theory,
\begin{equation*}
(P^{(|I|,r)} y_I)_i
=
\left(1, \frac{i}{n}, \ldots, \left(\frac{i}{n}\right)^r \right)^{\top}
\bigl((B^{(r,I)})^{\top} B^{(r,I)}\bigr)^{-1}
(B^{(r,I)})^{\top} y_I .
\end{equation*}

\begin{enumerate}
\item
Let $u_t \in \mathbb R^{r+1}$ denote the polynomial basis vector evaluated at the
design point $x_t = t/n$, that is,
\[
u_t := (1, x_t, x_t^2, \ldots, x_t^r)^{\top}.
\]
For any interval $I = [a:b] \subseteq [n]$, the associated Gram matrix is
\[
G_I := (B^{(r,I)})^{\top} B^{(r,I)} = \sum_{t=a}^b u_t u_t^{\top}.
\]
We form matrix-valued prefix sums
\[
S_b := \sum_{t=1}^b u_t u_t^{\top}, \qquad b = 1, \ldots, n,
\]
with the convention $S_0 = 0$.
Then, for any interval $I = [a:b]$,
\[
G_I = S_b - S_{a-1}.
\]
Constructing the prefix sums $\{S_b\}_{b=1}^n$ requires $O(n r^2)$ arithmetic
operations.

\item
Similarly, we form vector-valued prefix sums
\[
S'_b := \sum_{t=1}^b u_t y_t, \qquad b = 1, \ldots, n,
\]
with $S'_0 = 0$, so that for any interval $I = [a:b]$,
\[
H_I := (B^{(r,I)})^{\top} y_I = S'_b - S'_{a-1}.
\]
Constructing $\{S'_b\}_{b=1}^n$ requires $O(n r)$ operations.

\item
For DSMTF, we only require inverse Gram matrices $G_I^{-1}$ for dyadic intervals
$I \in \mathcal I_i$ across all $i \in [n]$.
The total number of distinct dyadic intervals queried is at most $O(n \log n)$.
Each inverse is computed by directly solving a $(r+1)\times(r+1)$ linear system
(with a small ridge regularization for numerical stability), which costs
$O(r^3)$ operations.
Consequently, the total cost of computing all required inverse Gram matrices is
$O(r^3 n \log n)$.
\end{enumerate}

\item \textbf{Pointwise aggregation.}

We now compute the DSMTF estimate at each location $i \in [n]$.

\begin{enumerate}
\item
For each interval $I \in \mathcal I_i$ containing $i$, we evaluate the local
polynomial projection at $x_i = i/n$,
\[
(P^{(|I|,r)} y_I)_i
=
\left(1, \frac{i}{n}, \ldots, \left(\frac{i}{n}\right)^r \right)^{\top}
G_I^{-1} H_I .
\]
Since $|\mathcal I_i| = O(\log n)$ and each evaluation involves matrix--vector
products of dimension $(r+1)$, this step requires $O(r^2 \log n)$ operations per
location.

\item
Using these values, we form an array of at most
$O(\log n) \times O(\log n)$ quantities of the form
\[
(P^{(|I|,r)} y_I)_i \pm \frac{\lambda C_{I,J}}{|I|},
\]
and compute the associated min--max and max--min values.
This step requires at most $O((\log n)^2)$ operations per location for each value
of $\lambda$.
\end{enumerate}

\end{itemize}

\medskip

Combining the above steps, the total computational complexity of DSMTF is
\[
O\!\left(
n r^2
\;+\;
r^3 n \log n
\;+\;
n r^2 \log n
\;+\;
n (\log n)^2 |\Lambda|
\right),
\]
where $|\Lambda|$ denotes the number of candidate values of the tuning parameter
$\lambda$.
For fixed $r$ and fixed $|\Lambda|$, this yields an overall runtime of order
$O\!\bigl(n (\log n)^2\bigr)$, up to constant factors.

\subsection{Empirical Comparisons with Trend Filtering}

We compare Minmax Trend Filtering (MTF) with Trend Filtering (TF) on the four
benchmark test functions introduced in \cite{donoho1994ideal}, namely
\texttt{Blocks}, \texttt{Bumps}, \texttt{HeaviSine}, and \texttt{Doppler}. These
functions exhibit highly heterogeneous local smoothness and have become
standard test beds for evaluating locally adaptive nonparametric regression
methods; see also \cite{tibshirani2014adaptive,donoho1995adapting,mammen1997locally}.
Trend Filtering (with cross-validation) is computed using the \texttt{genlasso}
R package, while MTF is implemented using our own R code. Specifically, we
provide an R package implementing dyadic symmetric minmax trend filtering
(DSMTF), supporting estimation over a prescribed grid of tuning parameters as
well as cross-validated tuning, which can be installed via
\texttt{remotes::install\_github("sabyasachic/dsmtf")} and is publicly available
at \url{https://github.com/sabyasachic/dsmtf}.

Our code takes as input the polynomial degree $r$ and a finite grid of candidate
tuning parameter values $\Lambda$ for cross-validation. For each of the four
benchmark test functions, the grid $\Lambda$ was chosen manually, based on
preliminary simulations in which we examined the typical range of cross-validated
optimal tuning parameter values. The final grid was then fixed in advance and used
uniformly across all simulation replicates for that function. In contrast, the R
package for trend filtering automatically searches over all tuning parameters, leveraging the piecewise linear structure of the solution path;
see~\cite{arnold2014glmgen}.

\begin{figure}[!th]
\centering
\includegraphics[scale = 0.4]{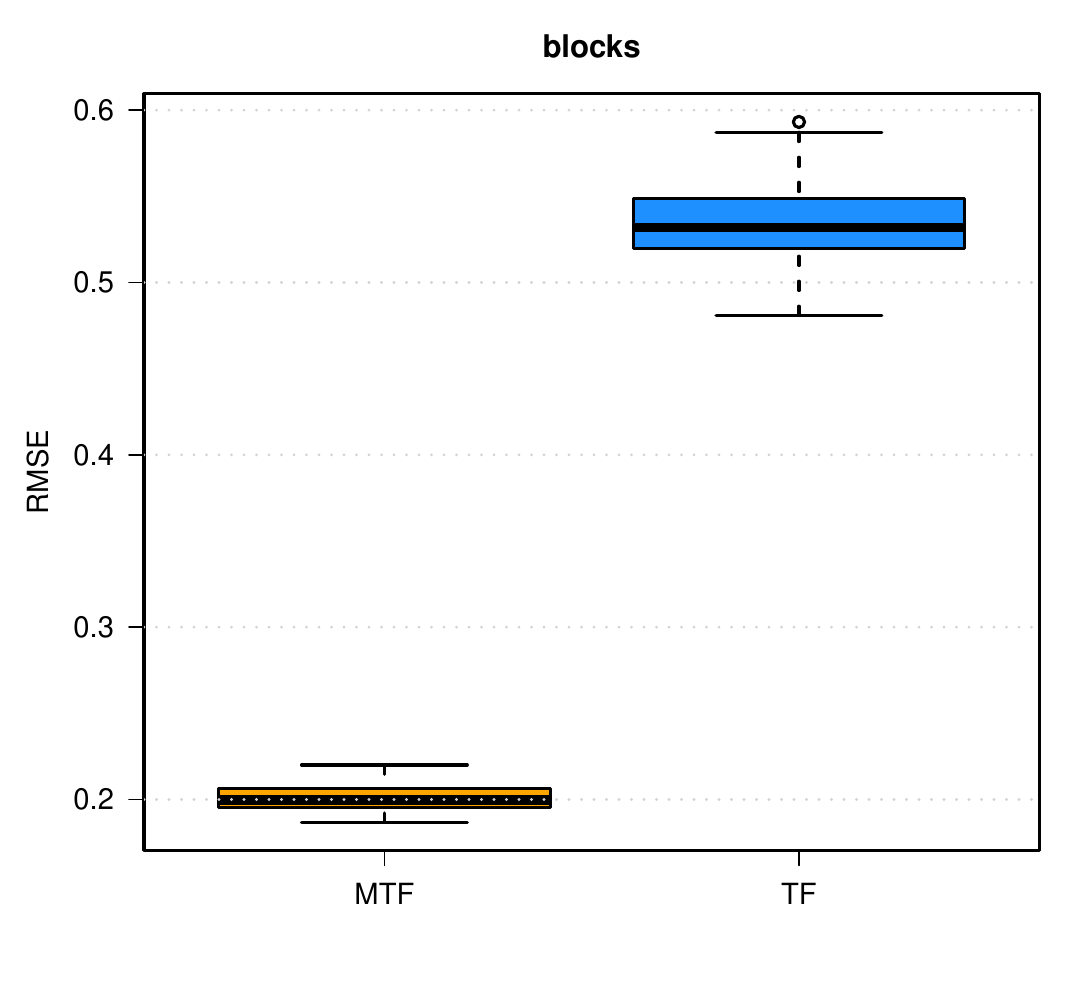}
\includegraphics[scale = 0.4]{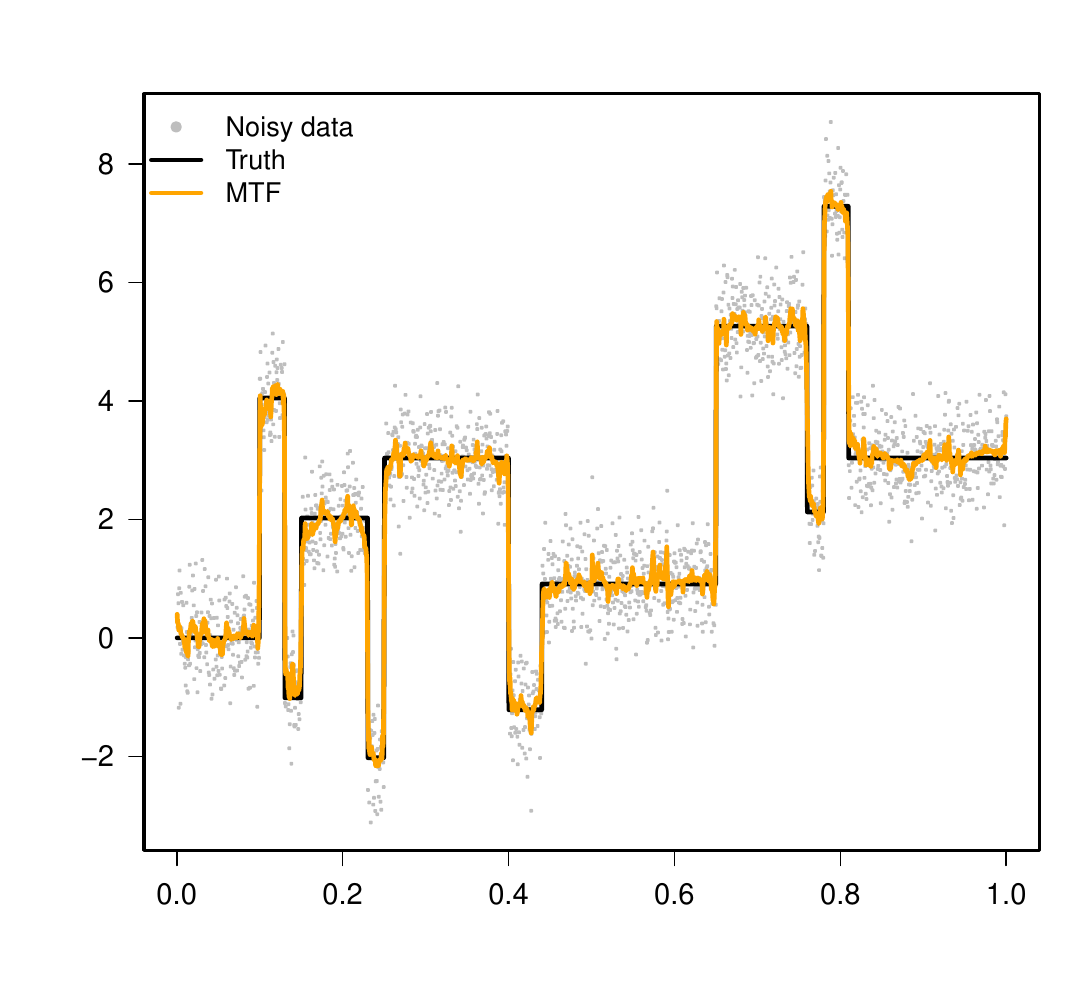}
\includegraphics[scale = 0.4]{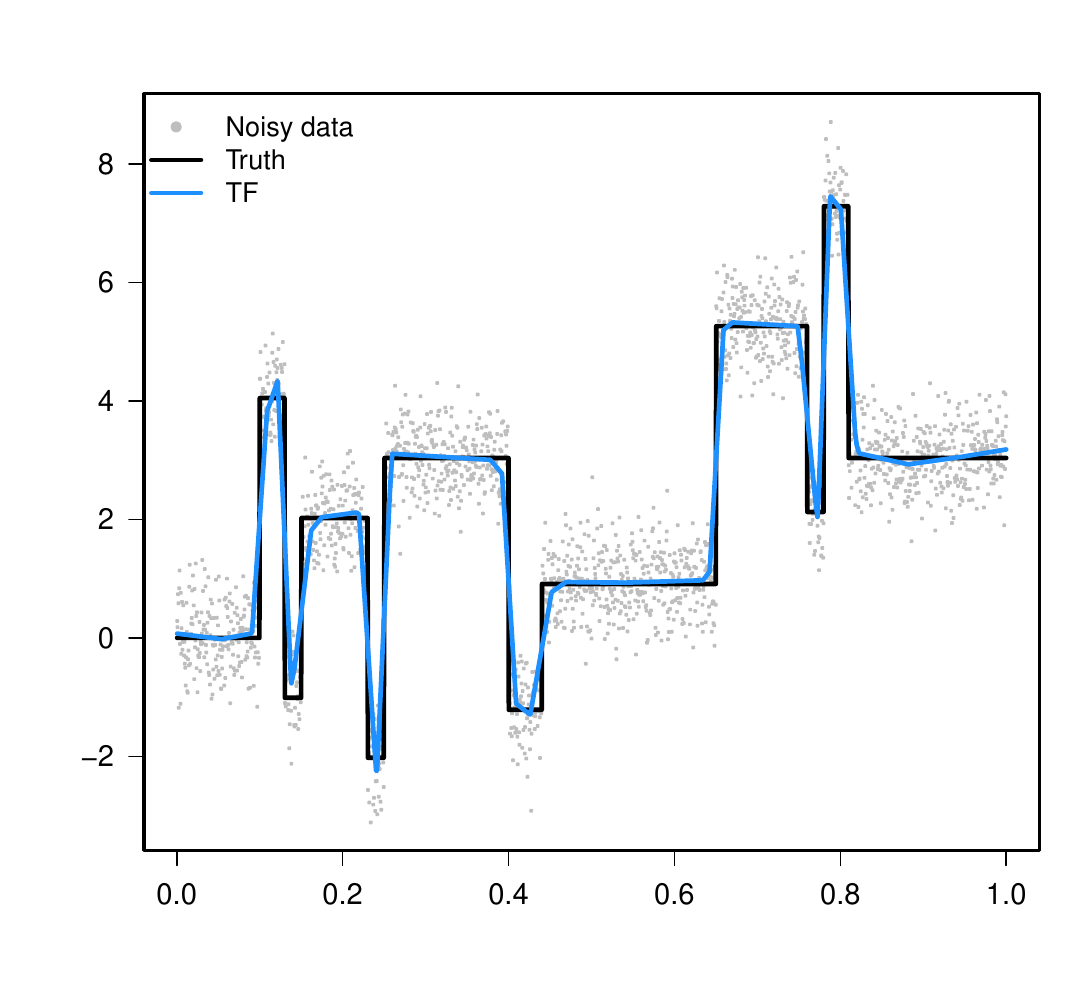}
\caption{The \texttt{Blocks} function. We compare DSMTF with $r=1$ and first-order
Trend Filtering.}
\label{fig:comp_blocks}
\end{figure}

\begin{figure}[!th]
\centering
\includegraphics[scale = 0.4]{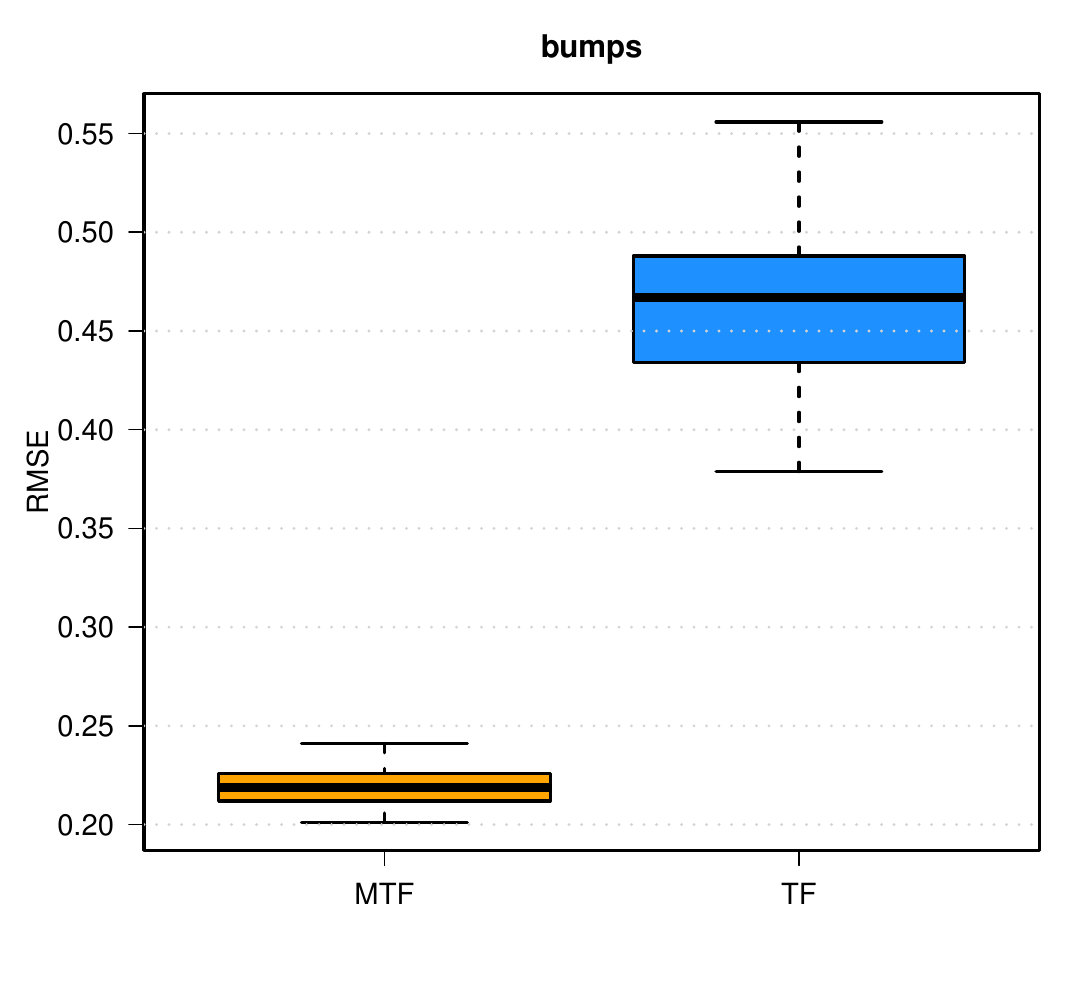}
\includegraphics[scale = 0.4]{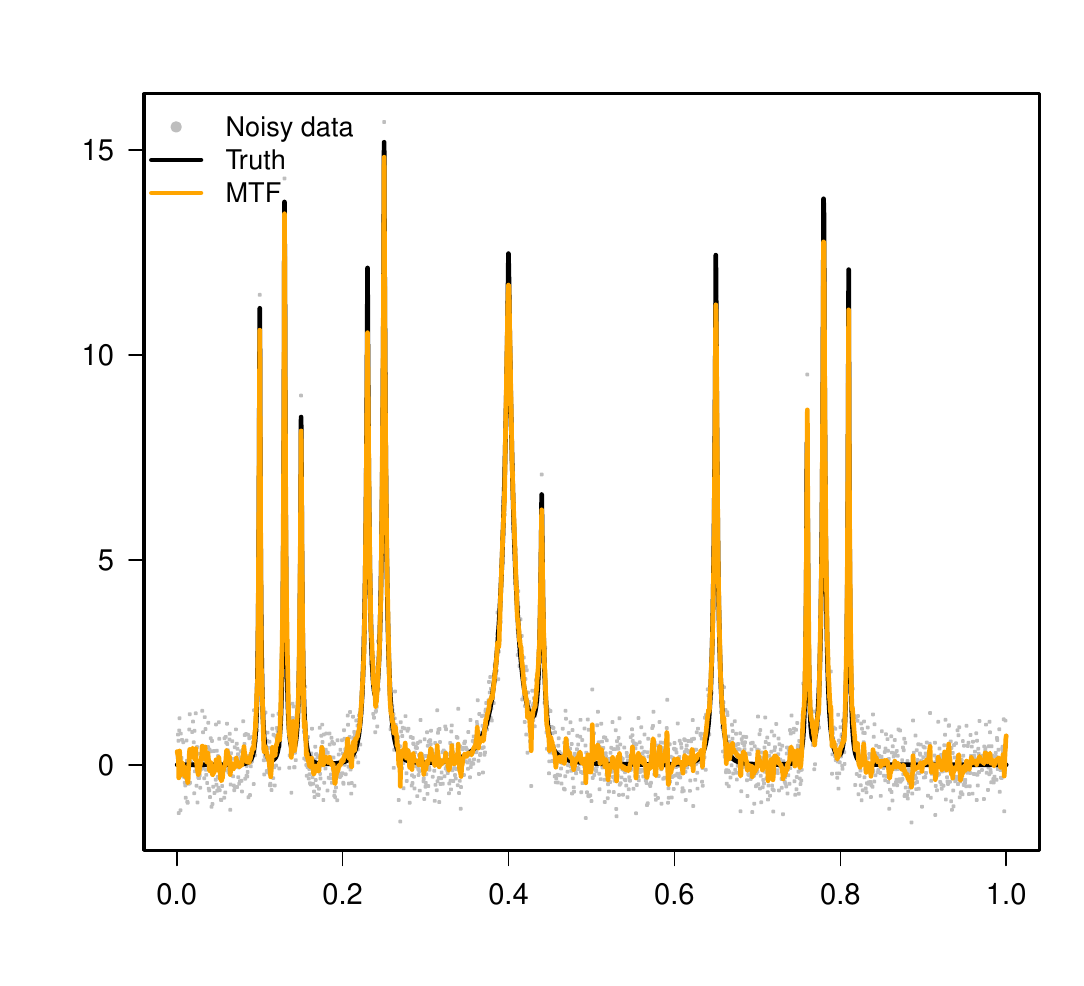}
\includegraphics[scale = 0.4]{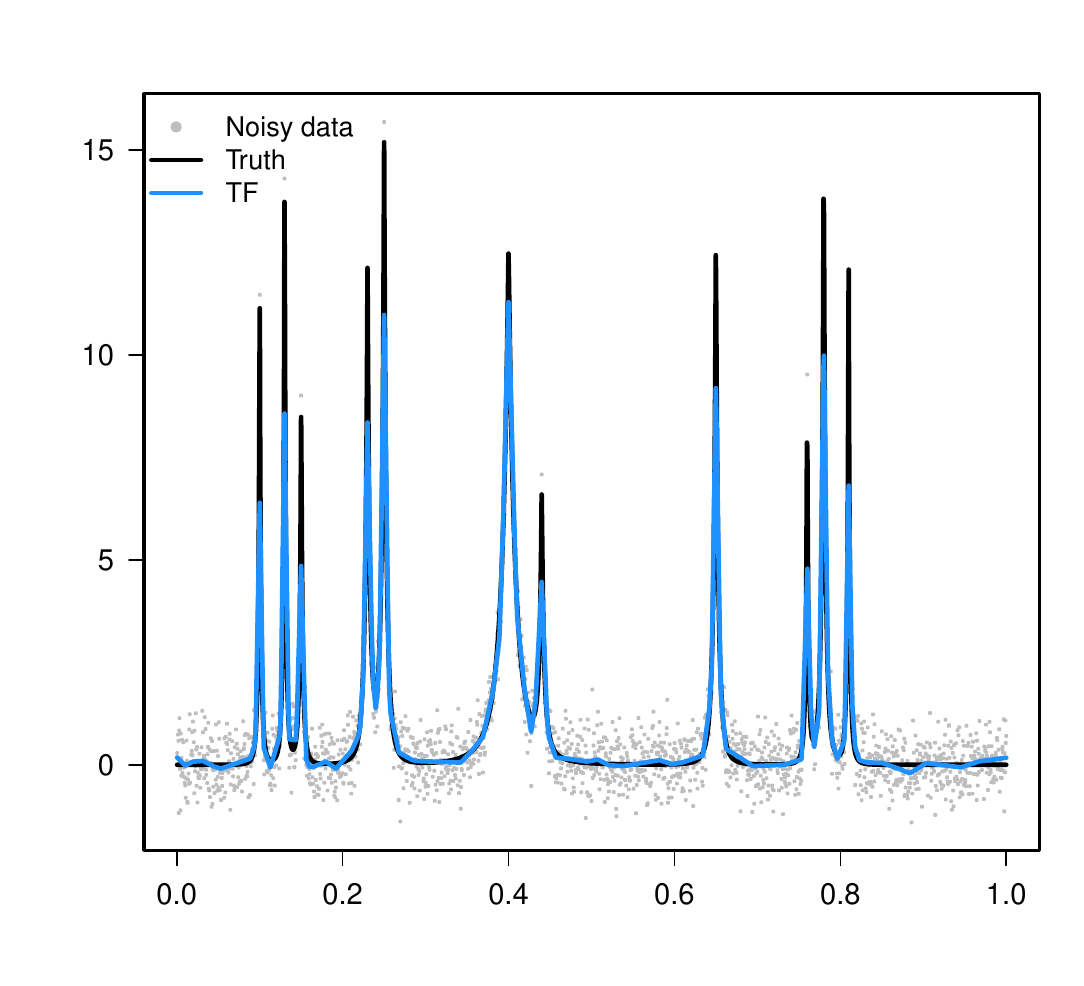}
\caption{The \texttt{Bumps} function. We compare DSMTF with $r=1$ and first-order
Trend Filtering.}
\label{fig:comp_bumps}
\end{figure}

\begin{figure}[!th]
\centering
\includegraphics[scale = 0.4]{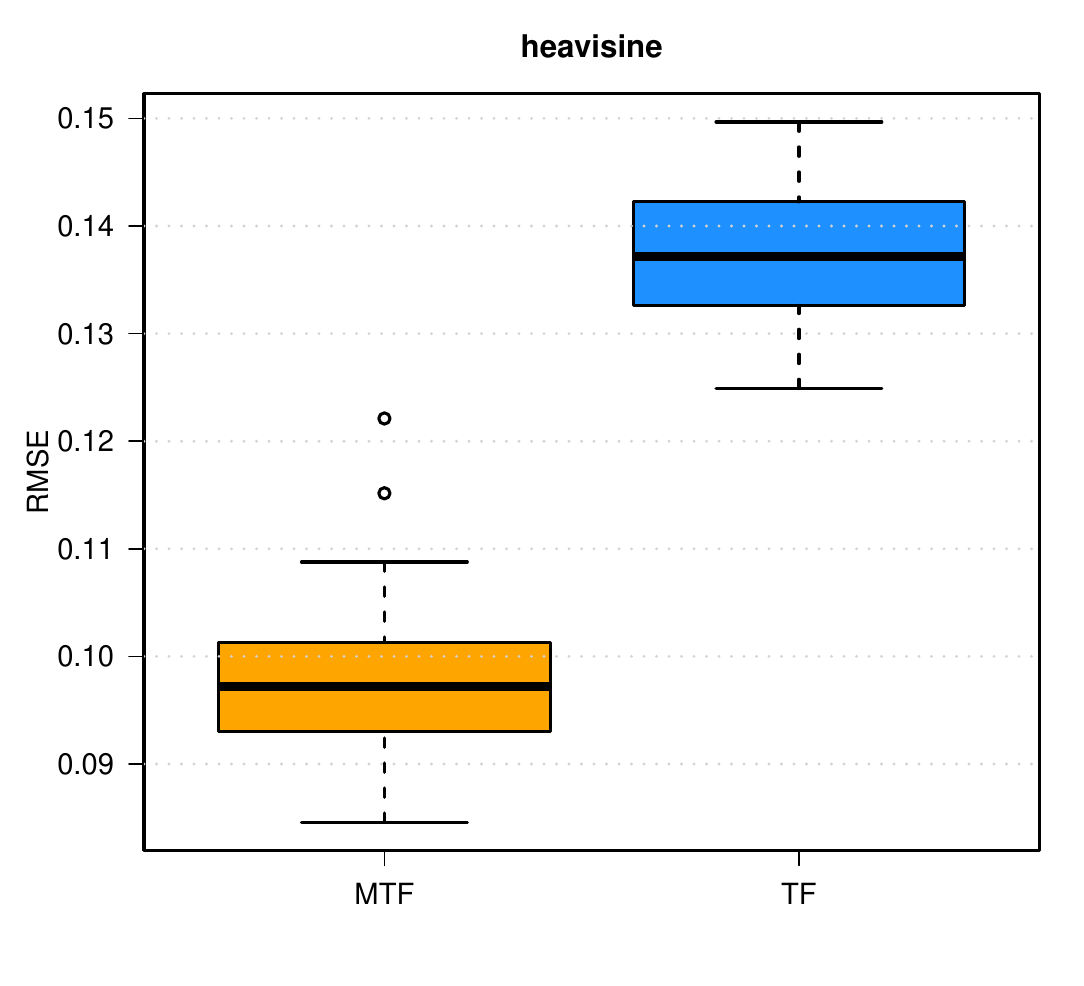}
\includegraphics[scale = 0.4]{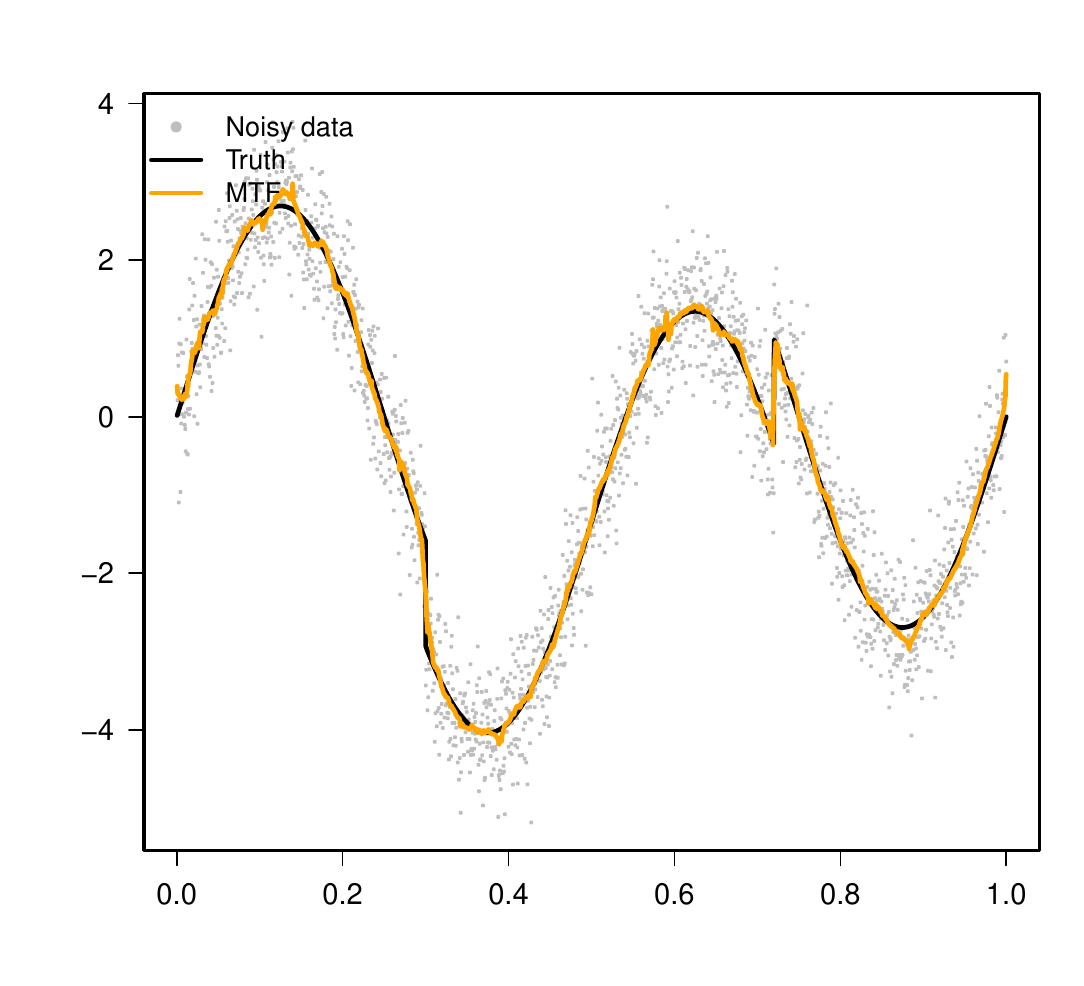}
\includegraphics[scale = 0.4]{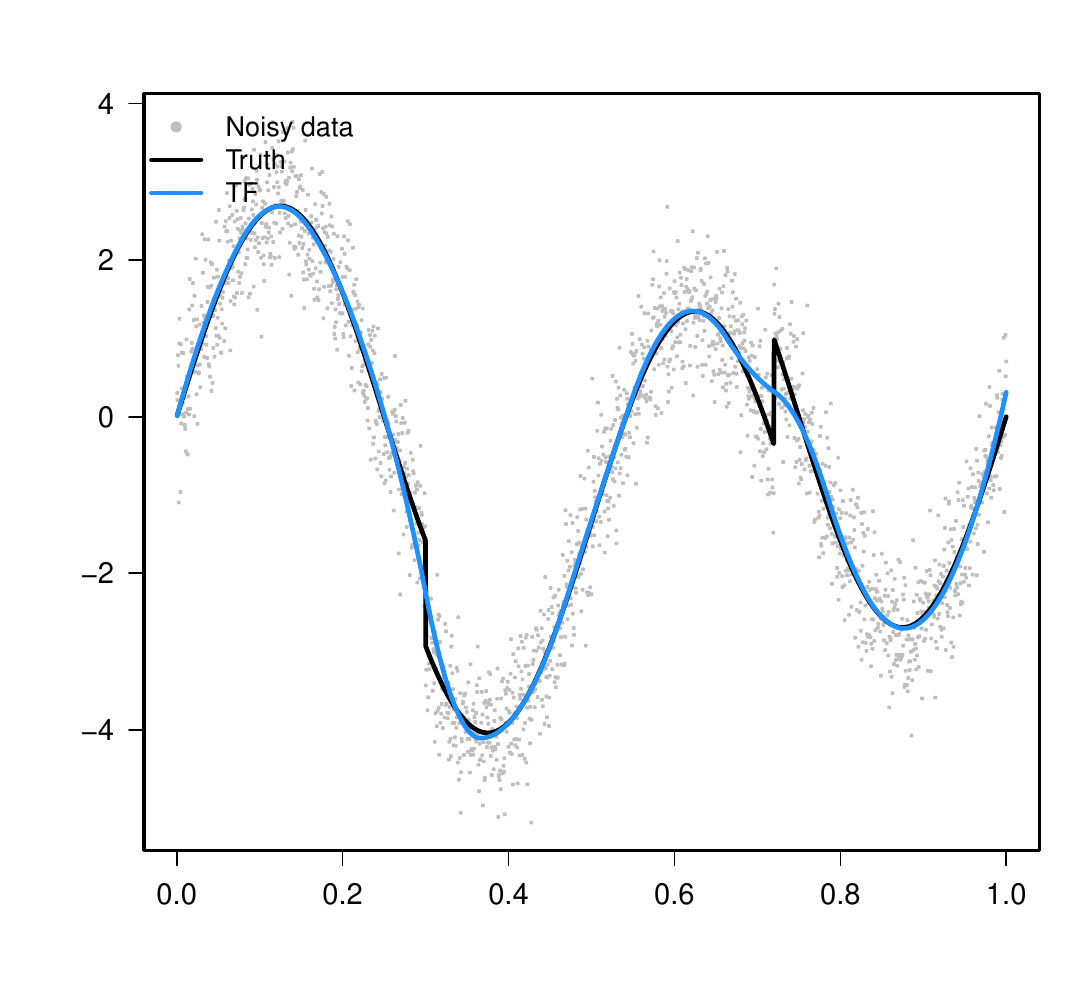}
\caption{The \texttt{Heavisine} function. We compare DSMTF with $r=2$ and second-order
Trend Filtering.}
\label{fig:comp_heavisine}
\end{figure}

\begin{figure}[!th]
\centering
\includegraphics[scale = 0.4]{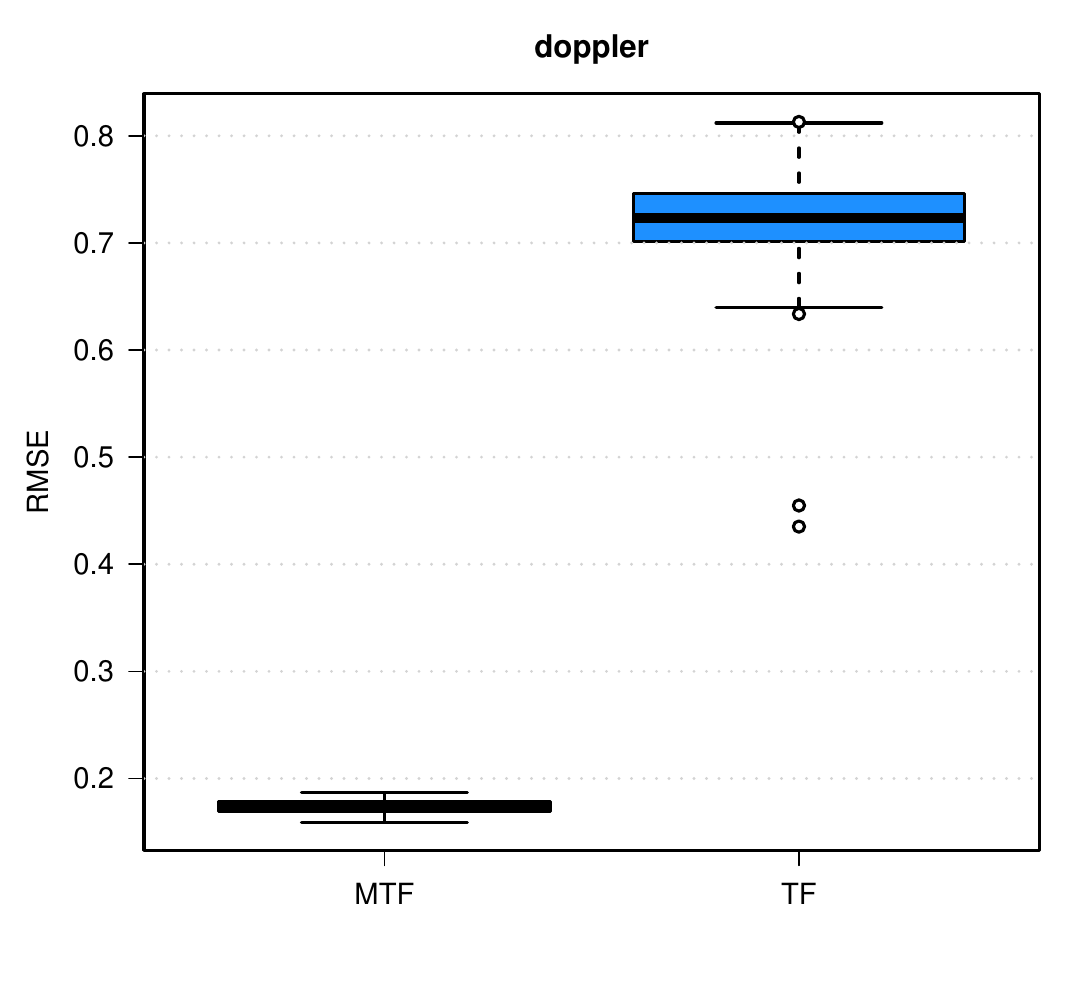}
\includegraphics[scale = 0.4]{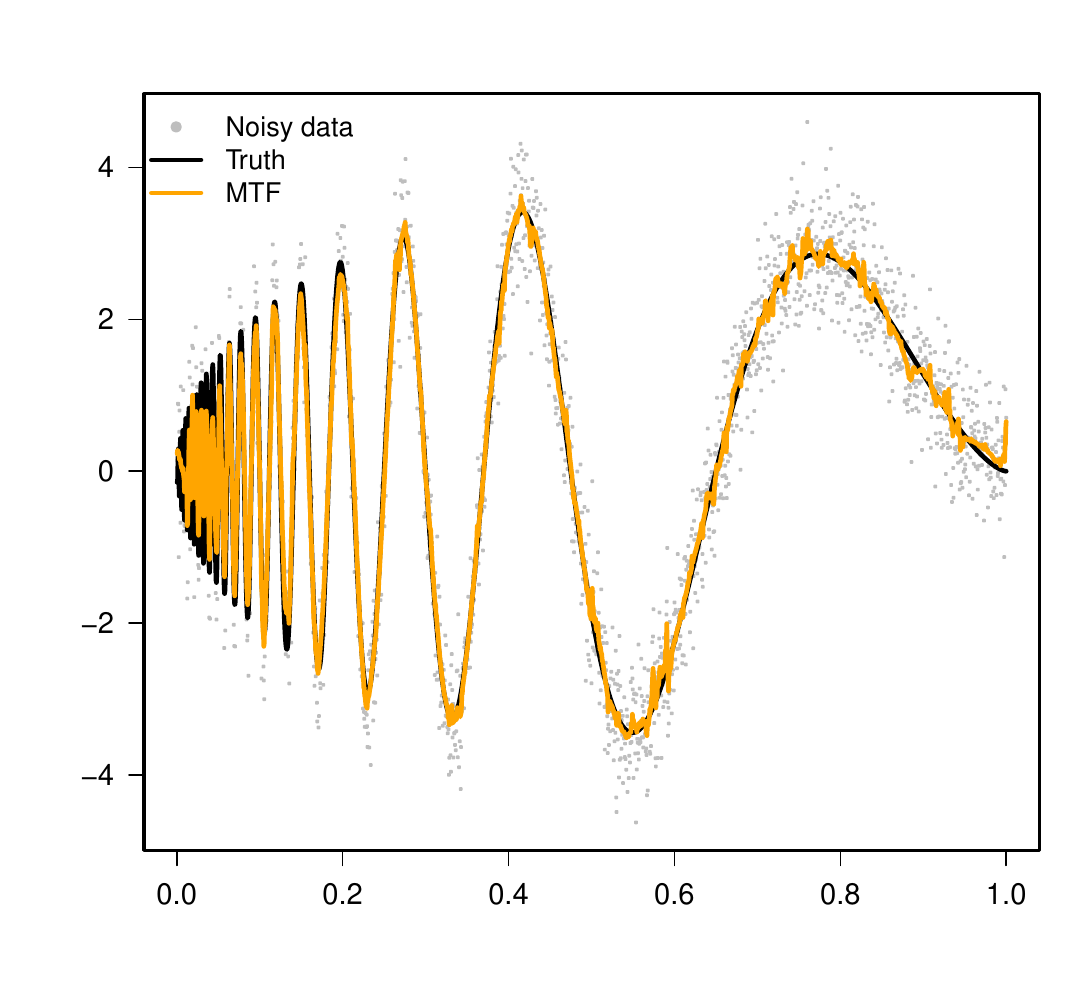}
\includegraphics[scale = 0.4]{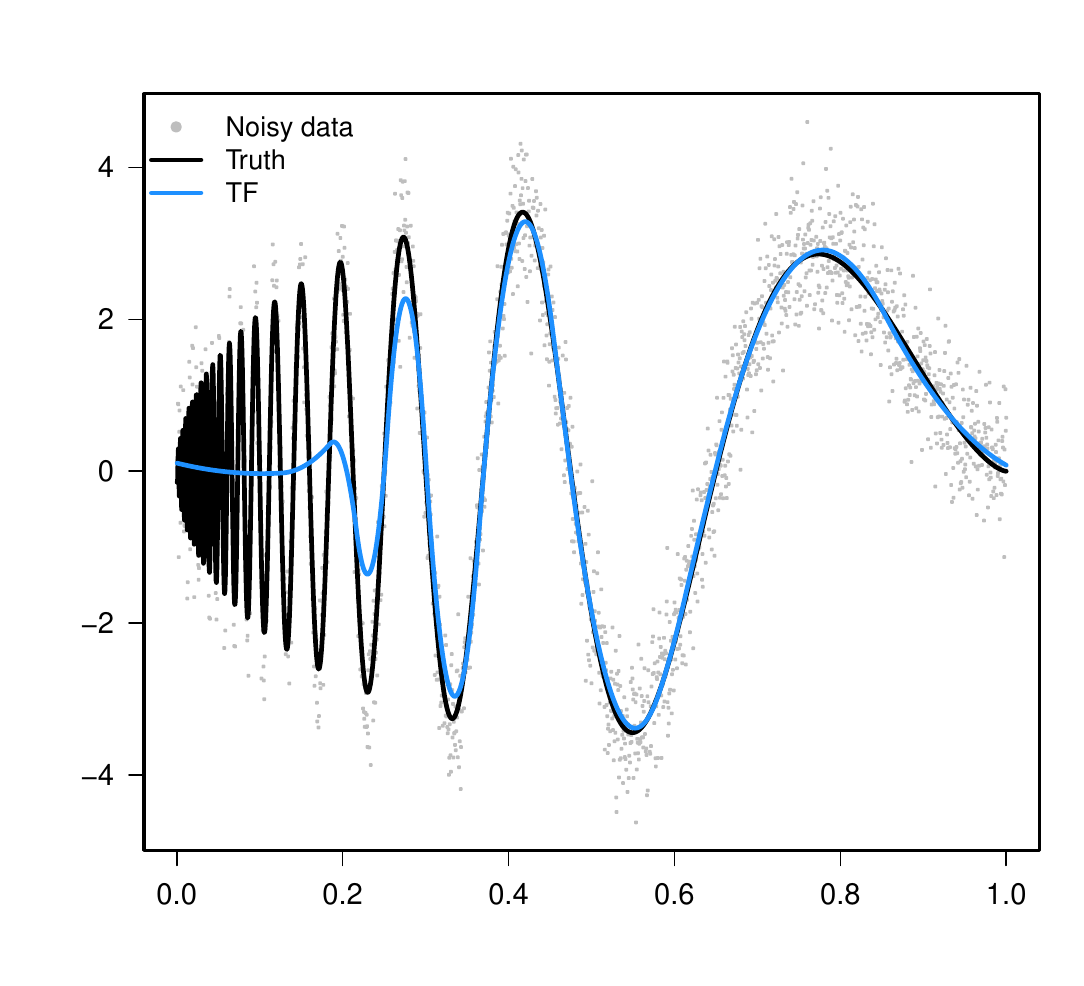}
\caption{The \texttt{Doppler} function. We compare DSMTF with $r=2$ and second-order
Trend Filtering.}
\label{fig:comp_doppler}
\end{figure}

Figures~\ref{fig:comp_blocks}--\ref{fig:comp_doppler} summarize the results of four
simulation studies, one for each of the benchmark functions. In each experiment,
we generate data of the form
\[
y_i = \theta_i + \sigma \varepsilon_i, \qquad i = 1,\ldots,n,
\]
where $\varepsilon \sim N_n(0,\mathrm{Id})$ and
\[
\theta_i = f\!\left(\frac{i}{n}\right), \qquad f \in
\{\texttt{Blocks},\texttt{Bumps},\texttt{HeaviSine},\texttt{Doppler}\}.
\]
The signal is rescaled according to
\[
\theta \leftarrow \mathrm{SNR}\cdot\sigma\cdot
\frac{\theta}{\mathrm{sd}(\theta)},
\]
where $\mathrm{sd}(\theta)$ denotes the empirical standard deviation, so that
$\mathrm{SNR} = \mathrm{sd}(\theta)/\sigma$ represents the signal-to-noise ratio.

Since Trend Filtering coincides with TVD when $r=0$, we focus on $r>0$. Specifically, we
use $r=1$ for the \texttt{Blocks} and \texttt{Bumps} functions, and $r=2$ for the
\texttt{HeaviSine} and \texttt{Doppler} functions, matching the order of Trend
Filtering accordingly.

Throughout these experiments, we fix $n=2048$, $\sigma=0.5$, and
$\mathrm{SNR}=4$. Each boxplot is based on $50$ Monte Carlo replications. The
tuning parameter $\lambda$ is selected using $5$-fold cross-validation for both
DSMTF and Trend Filtering. In each of
Figures~\ref{fig:comp_blocks}--\ref{fig:comp_doppler}, the top panel displays
boxplots of the RMSE across replications, while the bottom two panels show fitted
curves for a representative realization. To clarify, the RMSE is computed in-sample; that is, we calculate $$RMSE = \sqrt{\frac{1}{n} \sum_{i = 1}^{n} \big(\hat{\theta}^{(r,\hat{\lambda}_{cv})}_{i} - \theta^*_{i}\big)^2}.$$

Across all four test functions, DSMTF exhibits substantially improved
performance relative to Trend Filtering. For example, in the \texttt{Doppler}
experiment (Figure~\ref{fig:comp_doppler}), DSMTF accurately captures more than
seven oscillatory cycles in the displayed realization, whereas Trend Filtering
captures only a few. In the \texttt{Bumps} example
(Figure~\ref{fig:comp_bumps}), first-order Trend Filtering fails to recover
several prominent peaks, while DSMTF captures most of these local features. For
the \texttt{HeaviSine} function (Figure~\ref{fig:comp_heavisine}), DSMTF with
$r=2$ successfully recovers the kink near $x\approx0.7$, a feature missed by
second-order Trend Filtering. Finally, for the piecewise constant
\texttt{Blocks} function, we can see in Figure~\ref{fig:comp_blocks} that DSMTF localizes the change points more accurately. In all cases,  the RMSE of
DSMTF is stochastically smaller than that of Trend Filtering by a large and
statistically significant margin. 

Trend Filtering fits discrete splines of order $r$, that is, piecewise polynomials with enforced smoothness constraints at the knots; see \cite{tibshirani2020divided}. In contrast, MTF is not constrained
to produce spline fits and can therefore adapt to signals that are
discontinuous, have discontinuous derivatives, or exhibit nondifferentiable
behavior. Moreover, the performance of Trend Filtering can be sensitive to the
choice of order $r$. For instance, when the true signal is nearly piecewise
constant with heterogeneous segment lengths, choosing $r=1$ or $r=2$ instead of
$r=0$ can substantially degrade performance. This phenomenon does not arise for
MTF: since piecewise constant signals are also piecewise linear or quadratic,
MTF with $r=1$ or $r=2$ continues to perform well in such settings.

Finally, we investigate robustness with respect to the signal-to-noise ratio (SNR).
Holding all other parameters fixed, we increase the noise level $\sigma$ from $0.5$
to $1$, $2$, and $4$, corresponding to $\mathrm{SNR}=2, 1,$ and $0.5$, respectively,
and repeat the experiment on the \texttt{Doppler} test function. The results are
reported in Figure~\ref{fig:snr}. We observe that DSMTF keeps outperforming trend filtering
as the SNR is reduced from $4$ to $2$ and $1$. However, when the SNR decreases further
to $0.5$, trend filtering achieves a lower RMSE than DSMTF, suggesting that TF may
be preferable in extremely noisy regimes. That said, in this low-SNR setting both
estimators perform poorly on the Doppler function, indicating that neither method
is reliable under such severe noise conditions.

\begin{figure}[!th]
\centering
\includegraphics[scale = 0.5]{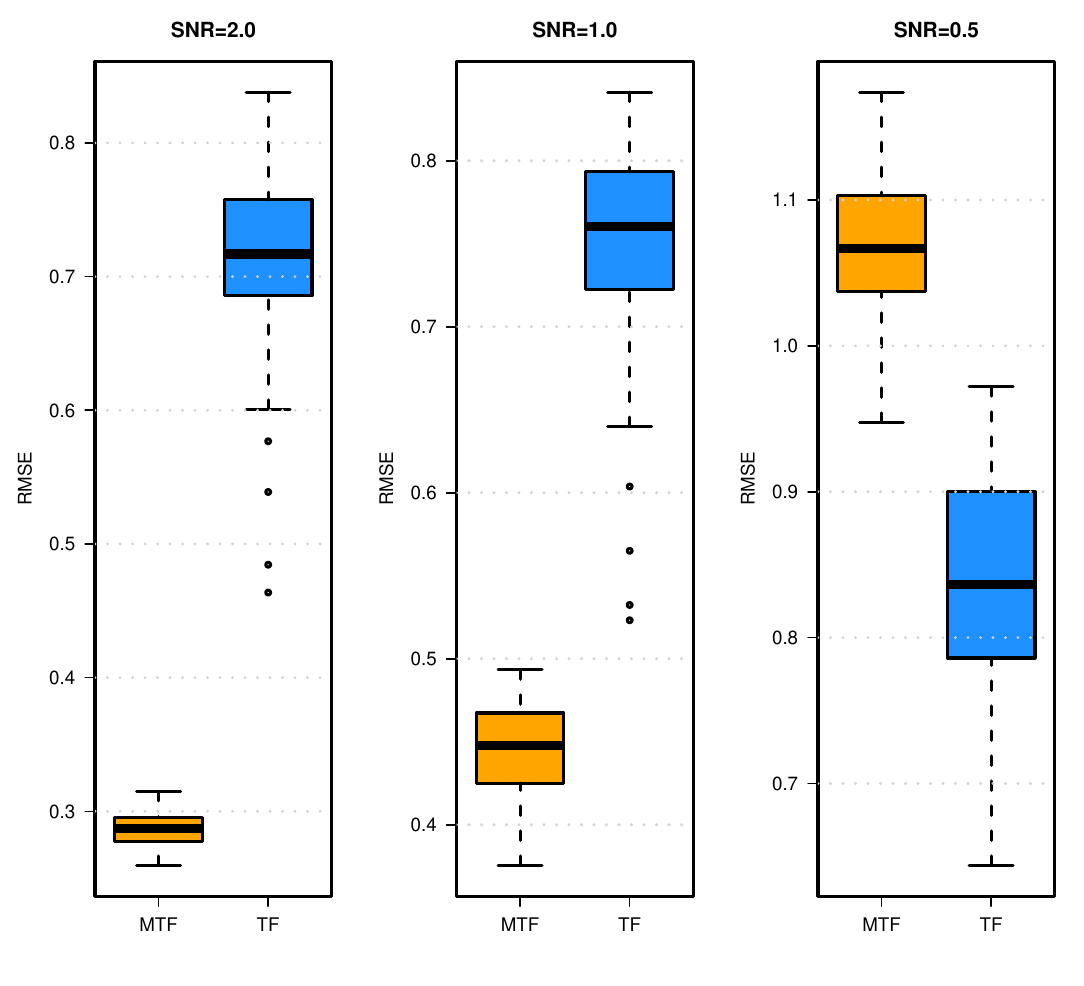}
\caption{Comparison of DSMTF and Trend Filtering on the \texttt{Doppler} function
for decreasing signal-to-noise ratios.}
\label{fig:snr}
\end{figure}

Overall, these experiments suggest that DSMTF can offer substantial gains over
Trend Filtering when the underlying signal exhibits strongly heterogeneous local
smoothness. We emphasize that we do not claim MTF to be universally superior to
trend filtering. Indeed, for sufficiently regular signals, trend filtering is
expected to perform better. For example, on a simple sinusoidal signal, we observe
that trend filtering achieves an RMSE approximately half that of MTF under the same
simulation settings considered in this section. This behavior is unsurprising, as
TF fits are constrained to be discrete splines and are therefore inherently smoother
than MTF fits. For such regular signals, estimating by smooth piecewise polynomial
functions is preferable, and trend filtering remains more suitable.

Nevertheless, our simulations indicate that (DS)MTF can be a practically useful and
competitive alternative to trend filtering in settings where strong adaptivity to
local irregularity is essential. The goal of these preliminary simulations is not to
assert uniform dominance of MTF, but rather to highlight regimes in which it serves
as an effective practical alternative and to motivate further empirical evaluation
of MTF across a broader range of applied problems.


\section{Discussion}
\label{sec:discuss}

This section discusses some additional aspects of the present work and outlines
related directions for future research. We first compare our pointwise results for total
variation denoising (TVD) with recent work in the literature. We then highlight broader
methodological connections and open problems suggested by the minmax/maxmin perspective
developed in this article.

\subsection{Relation to Previous Work on Pointwise Bounds for TVD}
\label{sec:previous}

The paper~\cite{zhang2023element} initiated the study of pointwise estimation error bounds
for univariate total variation denoising (TVD). We summarize below the main ways in which
the present work differs from and extends~\cite{zhang2023element}.

\begin{itemize}

\item Beyond piecewise constant signals.

The main result for TVD in~\cite{zhang2023element} is Theorem~4, which (informally) yields,
with high probability,
\[
|\hat{\theta}_i - \theta^*_i|
\;\lesssim\;
\frac{1}{\sqrt{d_i}} + \frac{1}{\lambda} + \frac{\lambda}{\ell_i},
\]
where $d_i$ is the distance from index $i$ to the nearest change point of $\theta^*$ and
$\ell_i$ is the length of the constant segment of $\theta^*$ containing $i$ (up to
logarithmic factors and dependence on the noise level).

This bound is meaningful primarily for \emph{exactly} piecewise constant signals with long
constant stretches. If $\theta^*$ has all distinct entries, then $d_i=\ell_i=1$ and the
bound becomes $O(1)$, even when $\theta^*$ is well approximated by a piecewise constant
signal with few jumps. As a result, the bound in~\cite{zhang2023element} does not directly
yield the fast rate for nearly piecewise constant signals
(Theorem~\ref{thm:fast}), the slow rate over bounded variation classes
(Theorem~\ref{thm:slow}), or the local rates for $C^{r,\alpha}$ functions with
$\alpha\in(0,1]$ (Theorem~\ref{thm:mainada}).

In contrast, the pointwise error bound developed here
(Theorem~\ref{thm:maingeneral}) holds for \emph{arbitrary} signals $\theta^*$ and recovers
the piecewise constant result of~\cite{zhang2023element} as a special case. To the best of our knowledge, even in the case $r=0$, Theorem~\ref{thm:maingeneral} is new and provides a natural and far reaching extension of
Theorem~4 in~\cite{zhang2023element}. Its generality enables a unified local rates analysis
(Theorem~\ref{thm:mainada}) for all degrees $r\ge0$ and all
$\alpha\in(0,1]\cup\{\infty\}$, with the setting $r=0$, $\alpha=\infty$ corresponding to
\cite{zhang2023element}. Because these bounds apply pointwise to every signal, they enable a direct study of how risk depends on local smoothness and to compare risk curves of TVD
and local averaging, yielding a complementary explanation of local adaptivity for TVD and
MTF.

\item Minmax/maxmin formulation of the TVD fit.

A central observation underlying our analysis is that the TVD estimator itself admits an
\emph{exact pointwise minmax/maxmin representation} in terms of modified local averages.
This identity holds for all data vectors $y$ and all tuning parameters $\lambda$, without
assumptions on the true signal. Such a pointwise
representation of the fitted values $\hat{\theta}_i$ was neither formulated nor exploited
in~\cite{zhang2023element}. We view the identification and use of this minmax/maxmin
structure as a key conceptual contribution of the present work.

\item Proof technique and bias--variance structure.

Expressing pointwise quantities in minmax/maxmin form is crucial for tractable analysis;
see, for example, \cite{Zhang02,deng2018isotonic} in the context of isotonic regression. In
our setting, the minmax/maxmin representation yields a pointwise bound that takes the form
of an optimized local bias--variance type tradeoff.

In~\cite{zhang2023element}, the analysis focuses on a single interval: the constant segment
of $\theta^*$ containing index $i$, on which the bias is zero by construction. Since the relevant interval
depends on the unknown $\theta^*$, such an approach cannot yield a pointwise formula for the TVD fit and moreover does not extend beyond locally constant signals. By contrast, our analysis considers all intervals $J$
containing $i$, leading to bounds involving both a local bias term and a local standard
deviation term $SD^{(r)}(i,J,\lambda)$. Optimizing this tradeoff over $J$ permits pointwise
control even when $\theta^*$ is not locally constant.

\item Extension to higher degrees.

The minmax representation provides a new perspective on TVD and naturally suggests
entry-wise polynomial generalizations; see Section~\ref{sec:minmaxtf}. Such formulas would
be difficult to anticipate without first identifying the minmax structure of the TVD fit.
Accordingly, higher-degree generalizations lie outside the scope of
\cite{zhang2023element}, whose analysis is intrinsically tied to piecewise constant
signals. In contrast, the present work introduces estimators with explicit pointwise
formulas for all $r\ge0$ and develops a unified pointwise analysis that is meaningful for all polynomial degrees and a broad class of signals.

\end{itemize}

\subsection{Some Further Aspects and Open Directions}

We conclude by highlighting several additional observations and open directions suggested
by the results of this article.

\begin{itemize}

\item Extensions to other loss functions:
It might be of interest to investigate whether the minmax/maxmin representation developed
here for squared-error loss extends to other losses, such as those arising in quantile or
logistic regression.

\item Tuning parameter adaptivity:
As discussed in Section~\ref{sec:local}, the MTF estimators require different choices of
the tuning parameter $\lambda$ depending on the local smoothness of the signal. A relevant question is whether the definition of MTF can be modified so that a
\emph{single} global choice of $\lambda$ yields adaptivity to multiple local H\"older
smoothness levels. Nonparametric estimators with this stronger form of adaptivity, while
remaining computationally efficient, are relatively rare; see
\cite{chatterjee2024new} and references therein.

\item Theorems~\ref{thm:ptwiseformula} and~\ref{thm:maingeneral} provide an exact pointwise
identity for the fitted values and a transparent bias--variance interpretation of the
pointwise estimation error for univariate TVD (trend filtering of order $0$), and
consequently for Minmax Trend Filtering of all orders $r \ge 0$. By contrast, deriving
analogous pointwise representations or pointwise risk bounds for classical trend
filtering of higher orders remains an open problem. The proof of the pointwise min--max/max--min identity for TVD relies on structural
features that are closely tied to first-order total variation regularization and,
to our knowledge, do not carry over in a direct way to higher-order trend filtering.
In the TVD case, the KKT conditions introduce a single
real-valued sequence whose increments correspond to first
differences of the fitted signal and which is subject to simple box constraints
and boundary pinning. This structure yields an exact interval-averaging identity; averaging the fitted values over any interval causes the contribution of the regularization term to collapse to boundary terms only. Together with the
piecewise-constant geometry of the TVD solution, this identity underlies the
plateau/valley argument that leads to an exact pointwise min--max representation.

For trend filtering of order $r \ge 1$, the corresponding dual variables encode
higher-order discrete derivatives.
As a result, we are not aware of an analogue of the telescoping interval identity
used in the TVD proof, and interval averages of the optimality conditions generally
depend on interior behavior rather than only on endpoints. Moreover, higher-order
trend filtering produces piecewise polynomial fits, so the geometric notion of
maximal plateaus and valleys that plays a central role in the TVD argument does not
have an obvious counterpart. These differences suggest that extending the TVD proof
strategy to classical higher-order trend filtering would likely require additional new ideas. Nevertheless, we hope that the pointwise
perspective and proof techniques developed here will help stimulate progress toward
analogous results for classical higher-order trend filtering.

\item Beyond TVD, kernel and nearest-neighbor variants:
The minmax/maxmin principle underlying our construction is not specific to TVD. In principle, one may start from any of the classical linear smoothers—
such as kernel smoothing or $k$-nearest-neighbor estimators—and apply the same framework to obtain
locally adaptive, pointwise-defined variants. A systematic study of such estimators is an
interesting direction for future work.

\item Simplicity of the proof technique:
Our analysis relies on relatively elementary probabilistic tools, primarily square-root
logarithmic bounds on maxima of sub-Gaussian random variables. This simplicity contributes
to the transparency of the resulting error bounds and clarifies how $\lambda$ should be
chosen in relation to local bias and variance considerations.

\item Connections to isotonic regression:
Univariate isotonic regression admits a classical pointwise minmax representation
\cite{RW75,RWD88}, which has been instrumental in deriving sharp pointwise risk bounds;
see, e.g., \cite{Zhang02,chatterjee2015risk}. Such representations have also been
extended to certain multivariate settings
\cite{fokianos2020integrated,deng2020isotonic}. In contrast, the estimators studied here
are not shape-constrained: the minmax/maxmin optimization is taken over intervals and
their subintervals containing a fixed point. To the best of our knowledge, this work is
the first to define a non–shape-constrained nonparametric regression estimator directly
via a pointwise minmax/maxmin formula.

\end{itemize}

To summarize, the central contribution of this work is the identification of an exact pointwise
minmax/maxmin representation for total variation denoising and its higher degree polynomial
generalizations. This representation yields transparent pointwise error bounds in the
form of local bias--variance type tradeoffs, enables a unified analysis across smoothness
classes and polynomial degrees, and offers a new conceptual explanation of local
adaptivity beyond classical shape-constrained settings. More broadly, the minmax/maxmin
perspective provides a general and flexible framework for constructing locally adaptive
nonparametric estimators.

\smallskip

\noindent \textbf{Acknowledgement.} SC's research was partly supported by the NSF Grant DMS-1916375. We thank Deep Ghoshal for many helpful discussions. We also thank three anonymous
reviewers for their helpful comments on an earlier version of the paper.

	\begin{supplement}
		\stitle{Supplement A : Supplementary File to "Minmax Trend Filtering: Generalizations of Total Variation Denoising via a Local Minmax/Maxmin Formula"}
		\sdescription{This supplementary contains the proofs of all the main results presented in this paper.}
	\end{supplement}

    \begin{appendix}
	\section{Proof of Theorem~\ref{thm:ptwiseformula}}\label{sec:ptwiseformulaproof}

\begin{proof}[Proof of Theorem~\ref{thm:ptwiseformula}]

Throughout this proof, we fix $\lambda \ge 0$ and write $\hat{\theta}$ in place of $\hat{\theta}^{(\lambda)}$ for notational simplicity.

It suffices to prove the min--max identity in \eqref{eq:fldefn}. Since the TVD objective is invariant under the transformation $(\theta,y)\mapsto(-\theta,-y)$,
it follows that $\hat\theta(\lambda;-y)=-\hat\theta(\lambda;y)$, and therefore the max--min
identity follows from the min--max identity applied to $-y$.

\subsection{Interval Identity}
We will begin by stating and proving an interval identity for the TVD solution.

\begin{lemma}[Interval identity for TVD]\label{lem:interval-identity}
There exists a vector $z=(z_0,\dots,z_n)\in\mathbb{R}^{n+1}$ such that
\[
z_0=z_n=0,
\qquad
z_k\in
\begin{cases}
\{\lambda\}, & \hat\theta_k>\hat\theta_{k+1},\\
[-\lambda,\lambda], & \hat\theta_k=\hat\theta_{k+1},\\
\{-\lambda\}, & \hat\theta_k<\hat\theta_{k+1},
\end{cases}
\quad k=1,\dots,n-1,
\]
and for every interval $I=[a:b]\subseteq[n]$,
\begin{equation}\label{eq:interval-identity}
\bar{\hat\theta}_I
=
\bar y_I - \frac{z_b-z_{a-1}}{|I|}.
\end{equation}
\end{lemma}

\begin{proof}
Let $D:\mathbb{R}^n\to\mathbb{R}^{n-1}$ be the first-difference operator
$(D\theta)_k=\theta_{k+1}-\theta_k$.
The TVD objective can be written as
\[
f(\theta)=\frac12\|\theta-y\|_2^2+\lambda\|D\theta\|_1.
\]
Since $f$ is a nondifferentiable convex function, optimality of $\hat\theta$ is equivalent to
$0\in\partial f(\hat\theta)$ where $\partial$ is the usual notation referring to the set of subgradients at a given point.

Hence there exists $s\in\partial\|D\hat\theta\|_1$ such that
\begin{equation}\label{eq:optcharac}
    0=(\hat\theta-y)+\lambda D^\top s.
\end{equation}

Componentwise,
\[
s_k\in
\begin{cases}
\{+1\}, & \hat\theta_{k+1}>\hat\theta_k,\\
[-1,1], & \hat\theta_{k+1}=\hat\theta_k,\\
\{-1\}, & \hat\theta_{k+1}<\hat\theta_k.
\end{cases}
\]

A direct computation gives

\begin{equation*}
    D^\top s = (-s_1,s_1 - s_2,s_2 - s_3,\dots,s_{n - 2} - s_{n - 1},s_{n - 1}).
\end{equation*}

Now define
\[
z_0=0,\qquad z_k=-\lambda s_k\ (k=1,\dots,n-1),\qquad z_n=0.
\]
Then the stated bounds on $z_k$ hold, and moreover we can now write

\begin{equation*}
    \lambda D^\top s = (z_1,z_2 - z_1,z_3 - z_2,\dots,z_{n - 1} - z_{n - 2},-z_{n - 1}) = (z_1 - z_0,\dots,z_{n} - z_{n - 1})
\end{equation*}

Plugging in the above expression for $\lambda D^\top s$ in~\eqref{eq:optcharac} yields
\begin{equation*}
\hat\theta_j - y_j = z_{j-1} - z_j,
\qquad j=1,\dots,n.
\end{equation*}

Summing the above display over
$j=a,\dots,b$, the right hand side telescopes to give
\[
\sum_{j=a}^b(\hat\theta_j-y_j)=z_{a-1} - z_b,
\]
which implies \eqref{eq:interval-identity}.
\end{proof}

\subsection{Upper bound (inequality)}

We now prove that the minmax identity is an upper bound, i.e., given any fixed interval $J=[a:b]\ni i$, the following holds: 
\begin{equation}\label{eq:upident}
    \hat\theta_i
\leq \max_{\substack{I\subseteq J\\ i\in I}}
\Bigl(
\bar y_I - \frac{2\lambda}{|I|} C_{I,J}
\Bigr).
\end{equation}

\textit{Let $I=[c:d]\subseteq J$ be the \emph{largest} subinterval of $J$ containing $i$ such that
\[
\hat\theta_u\ge\hat\theta_i\qquad\forall u\in I.
\]
}

We claim that 
\begin{equation*}
    \hat\theta_i
\leq \Bigl(
\bar y_I - \frac{2\lambda}{|I|} C_{I,J}
\Bigr)
\end{equation*}

which will then prove~\eqref{eq:upident}.

Note that we can write
\begin{equation*}
    \hat\theta_i \leq \bar{\hat\theta}_I = \bar{y}_I-\frac{z_d-z_{c-1}}{|I|}
\end{equation*}
where the first inequality is by definition of the interval $I$ and the equality is by Lemma~\ref{lem:interval-identity}.

In view of the last two displays it suffices to show
\begin{equation}\label{eq:upper-key}
z_{c-1}-z_d\le -2\lambda\,C_{I,J}.
\end{equation}

We verify \eqref{eq:upper-key} case by case, exactly following Definition~\ref{def:cij}. Recall that the $z$ vector satisfies 
\[
|z_k|\le \lambda \quad \text{for all } k=1,\dots,n-1,
\]
and moreover, if $\hat{\theta}_{k + 1} \neq \hat{\theta}_{k}$,
\[
z_k=\pm \lambda
\]
with the sign determined by $\operatorname{sign}(\hat\theta_{k}-\hat\theta_{k + 1})$. Also, $z_0 = z_n = 0.$

\medskip
\noindent\underline{\textbf{Case 1: $1<a\le b<n$ (interior $J$).}}
\[
C_{I,J}=
\begin{cases}
1, & I\subset J,\\
-1, & I=J,\\
0, & \text{otherwise}.
\end{cases}
\]
\begin{itemize}
\item If $I\subset J$, maximality of $I$ forces $z_{c-1}=-\lambda$ and $z_d=\lambda$,
so $z_{c-1}-z_d=-2\lambda$.

\item If $I=[c:b]$ with $c > a$ then maximality of $I$ forces $z_{c-1}=-\lambda$, and thus
$z_{c-1}-z_d=-\lambda-z_d\le 0$ since $|z_d|\le \lambda.$

\item If $I=[a:d]$ with $d < b$, then maximality of $I$ forces $z_d=\lambda$, and thus
$z_{c-1}-z_d=z_{c-1}-\lambda\le 0$ since $|z_{c-1}|\le \lambda.$

\item If $I=J$, $z_{c-1}-z_d = z_{a-1}-z_b \leq 2\lambda$ since $\max\{|z_{a - 1}|,|z_{b}|\} \leq \lambda.$
\end{itemize}

\medskip
\noindent\underline{\textbf{Case 2: $J=[1:b]$, $b<n$.}}
\[
C_{I,J}=
\begin{cases}
1, & I\subset J,\\
-\tfrac12, & I=J,\\
\tfrac12, & I=[1:d],\ d<b,\\
0, & I=[c:b],\ c>1.
\end{cases}
\]
Using $z_0=0$:
\begin{itemize}
\item $I\subset J$: same as Case 1, $z_{c-1}-z_d=-2\lambda$.
\item $I=[1:d]$, $d<b$: maximality of $I$ forces $z_d=\lambda$, hence $z_{c-1}-z_d=-z_d=-\lambda$.
\item $I=[c:b]$, $c>1$: maximality of $I$ forces $z_{c-1}=-\lambda$, hence $z_{c-1}-z_b=-\lambda-z_b \le 0$ since $|z_b|\le \lambda$.
\item $I=J$: $z_{c-1}-z_d=-z_b\le\lambda$.
\end{itemize}

\medskip
\noindent\underline{\textbf{Case 3: $J=[a:n]$, $a>1$.}}
This can be shown similar to Case $2$ using $z_n=0.$ 

\medskip
\noindent\underline{\textbf{Case 4: $J=[1:n]$.}}
\[
C_{I,J}=
\begin{cases}
1, & I\subset J,\\
0, & I=J,\\
\tfrac12, & \text{otherwise}.
\end{cases}
\]
Using $z_0=z_n=0$:

\begin{itemize}
\item $I\subset J$: same as Case 1, $z_{c-1}-z_d=-2\lambda$.

\item If $I=[c:d]$ with $c > 1, d = n$ then maximality of $I$ forces $z_{c-1}=-\lambda$, thus giving $z_{c-1}-z_d = z_{c-1} = -\lambda.$

\item If $I=[c:d]$ with $c = 1,d < n$, then maximality of $I$ forces $z_d=\lambda$, thus giving $z_{c-1}- z_d = -z_d = -\lambda.$

\item If $I=J$, then $z_{c-1}-z_d  = z_0 - z_n = 0.$

\end{itemize}

This exhausts all the possible cases and thus we have shown~\eqref{eq:upper-key} and hence proved~\eqref{eq:upident}. 

\subsection{Equality}

We now prove that the minmax identity is actually an equality, i.e., there exists an interval $J=[a:b] \ni i$, such that the following holds: 
\begin{equation}\label{eq:loident}
    \hat\theta_i
\geq \max_{\substack{I\subseteq J\\ i\in I}}
\Bigl(
\bar y_I - \frac{2\lambda}{|I|} C_{I,J}
\Bigr).
\end{equation}

\textit{Let $J=[a:b]$ be the \emph{largest} interval containing $i$ such that
\[
\hat\theta_u\le\hat\theta_i\qquad\forall u\in J.
\]
}

We will now show that~\eqref{eq:loident} holds for this particular choice of $J.$

For any subinterval $I=[c:d]\subseteq J$ containing $i$, we can write 
$$\hat\theta_i \ge \bar{\hat\theta}_I =
\bar y_I-\frac{z_d-z_{c-1}}{|I|}$$
where the first inequality is by definition of $J$ and the fact $I \subseteq J$ and the equality is by Lemma~\ref{lem:interval-identity}.

Thus it suffices to show
\begin{equation}\label{eq:lower-key}
z_{c-1}-z_d\ge -2\lambda\,C_{I,J}.
\end{equation}

We now verify \eqref{eq:lower-key} again case by case, exactly following Definition~\ref{def:cij}.

\medskip
\noindent\underline{\textbf{Case 1: $1<a\le b<n$ (interior $J$).}}
\[
C_{I,J}=
\begin{cases}
1, & I\subset J,\\
-1, & I=J,\\
0, & \text{otherwise}.
\end{cases}
\]
\begin{itemize}
\item If $I\subset J$, then $z_{c-1}-z_d \geq -2\lambda$ since $\max\{|z_{c - 1}|,|z_{d}|\} \leq \lambda.$ 

\item If $I=[c:d]$ with $a = c, d < b$, then maximality of $J$ forces $z_{c - 1} = z_{a-1}=\lambda$, giving $z_{c-1}- z_d = \lambda - z_d \ge 0$ since $|z_{d}| \leq \lambda.$

\item If $I=[c:d]$ with $c > a, d = b$ then maximality of $J$ forces $z_{d} = z_b=-\lambda$, giving $z_{c-1}-z_d = z_{c-1}+\lambda \ge 0$ since $|z_{c - 1}| \leq \lambda.$

\item If $I=J$, then by maximality of $J$ we have $\hat\theta_{a-1}>\hat\theta_a$ and $\hat\theta_b<\hat\theta_{b+1}$, hence $z_{a-1}=\lambda$ and $z_b=-\lambda$, so $z_{c-1}-z_d = z_{a-1}-z_b = 2\lambda.$
\end{itemize}

\medskip
\noindent\underline{\textbf{Case 2: $J=[1:b]$, $b<n$.}}
\[
C_{I,J}=
\begin{cases}
1, & I\subset J,\\
-\tfrac12, & I=J,\\
\tfrac12, & I=[1:d],\ d<b,\\
0, & I=[c:b],\ c>1.
\end{cases}
\]
Using $z_0=0$:
\begin{itemize}
\item $I\subset J$: same as Case 1, $z_{c-1}-z_d \geq -2\lambda$.
\item $I=[c:d]$, $c = 1, d<b$: $z_{c-1}-z_d=-z_d \geq -\lambda$.
\item $I=[c:d]$, $c>1, d = b$: maximality of $J$ forces $z_b=-\lambda$, hence $z_{c-1}-z_d = z_{c-1}+\lambda \ge 0$.
\item $I=J$: $z_{c-1}-z_d=-z_b = \lambda$.
\end{itemize}

\medskip
\noindent\underline{\textbf{Case 3: $J=[a:n]$, $a>1$.}}
This can be shown similar to Case $2$ using $z_n=0.$ 

\medskip
\noindent\underline{\textbf{Case 4: $J=[1:n]$.}}
\[
C_{I,J}=
\begin{cases}
1, & I\subset J,\\
0, & I=J,\\
\tfrac12, & \text{otherwise}.
\end{cases}
\]
Using $z_0=z_n=0$:

\begin{itemize}
\item $I\subset J$: same as Case 1, $z_{c-1}-z_d \geq -2\lambda$.

\item If $I=[c:d]$ with $c > 1, d = n$ then $z_{d}= 0$ thus giving $z_{c-1}-z_d = z_{c-1} \geq -\lambda.$

\item If $I=[c:d]$ with $c = 1,d < n$, then $z_{c - 1} = 0$ giving $z_{c-1}- z_d = -z_d \geq -\lambda.$

\item If $I=J$, then $z_{c-1}-z_d  = z_0 - z_n = 0.$

\end{itemize}

This exhausts all the possible cases and thus we have shown~\eqref{eq:lower-key} and hence proved~\eqref{eq:loident}. 

\subsection{Conclusion}
Combining~\eqref{eq:upident} and~\eqref{eq:loident} yields the min--max identity and applying the same argument to $-y$ yields the corresponding max--min identity.

\end{proof}

	\section{Proof of Theorem~\ref{thm:maingeneral}}\label{sec:mainthmproof}

To prove Theorem~\ref{thm:maingeneral}, it suffices to prove Proposition~\ref{prop:determ} and Proposition~\ref{prop:noise} which we state and prove below.

	We first define the $r$th order effective noise variable  
	$$M_i^{(r)} = \max_{I \in \mathcal{I}_i} \big[|(P^{(|I|,r)} \epsilon_{I})_i| \sqrt{|I|}\big].$$

	We now define a notion of $r$th order local standard error for any location $i \in [n]$ and any interval $J \in \mathcal{I}_i.$

	\begin{equation*}
		SE^{(r)}(i,J,\lambda) = \frac{M_i^{(r)}}{\sqrt{Dist(i,\partial J)}} \mathrm{1}(i \notin \{1,n\}) + \frac{M_i^{(r)}}{\sqrt{|J|}} + \frac{\left(M_i^{(r)}\right)^2}{4 \lambda} + \frac{2\lambda}{|J|}.
	\end{equation*}

	
	We now state our main deterministic pointwise error bound in the form of the next proposition.
	
	\begin{proposition}[Deterministic Pointwise Estimation Error as Local Bias Variance Tradeoff]\label{prop:determ}
		
		Fix a nonnegative integer $r \geq 0$ and $\lambda > 0.$ The estimation error of the $r$th order Minmax Filtering estimator defined in~\ref{defn:minmaxfgen}, at any location $i$, is deterministically bounded by a local bias variance tradeoff:



	
	\begin{equation}
		\max_{J \in \mathcal{I}_i} \left(Bias^{(r)}_{-}(i,J,\theta^*) - SE^{(r)}(i,J,\lambda)\right)	\leq \hat{\theta}^{(r,\lambda)}_i - \theta^*_i  \leq \min_{J \in \mathcal{I}_i} \left(Bias^{(r)}_{+}(i,J,\theta^*) + SE^{(r)}(i,J,\lambda)\right).
	\end{equation}
\end{proposition}


\begin{proposition}[A Probabilistic Bound on the Effective Noise]\label{prop:noise}
	
	Recall the effective noise variable
	$$M_i^{(r)} = \max_{I \in \mathcal{I}_i} \big[|(P^{(|I|,r)} \epsilon_{I})_i| \sqrt{|I|}\big].$$
	Suppose $(\epsilon_1,\dots,\epsilon_n)$ are i.i.d with a Subgaussian($\sigma$) distribution.

	For full MTF, when $\mathcal{I}_i$ is the set of all intervals of $[n]$ containing $i$, with (polynomially high) probability not less than $1 - n^{-c}$, 
	$$M_i^{(r)} \leq C_r \sigma \sqrt{\log n}$$
	where $C_r > 0$ is an absolute constant which only depends on $r$ and $c$ which can be fixed to be any positive number, say $5$. 

    For DSMTF, when $\mathcal{I}_i$ is the set of all dyadic intervals of $[n]$ centred at $i$, with (polynomially high) probability not less than $1 - (\log n)^{-c}$, 
	$$M_i^{(r)} \leq C_r \sigma \sqrt{\log \log n}.$$

    \end{proposition}

\begin{remark}
	The above proposition is proved by showing that for any interval $I \in \mathcal{I}_i$ the random variable $\big[|(P^{(|I|,r)} \epsilon_{I})_i| \sqrt{|I|}\big]$ is subgaussian with subgaussian norm of the order $\sigma$ and then using the standard maxima bound for subgaussians. Technical facts about projection matrices on the subspace of polynomials are used to show the subgaussianity property. 
\end{remark}

\begin{remark}\label{rem:mtfvsdsmtf}
The above proposition implies that, for the DSMTF estimator, the $\log n$ factor
in the standard deviation term can be improved to $\log\log n$ at the cost of a
weaker probability guarantee, namely $1-(\log n)^{-c}$ instead of $1-n^{-c}$.
Accordingly, one could state a version of Theorem~\ref{thm:maingeneral} for
DSMTF at a fixed location $i$ with this sharper standard deviation bound on a
lower-probability event.

However, throughout the paper we emphasize bounds that hold simultaneously over
all locations $i\in[n]$. For such uniform guarantees, even in the DSMTF setting,
a union bound over locations necessarily reintroduces a $\log n$ factor. For
this reason, we do not state separate theorems for DSMTF, and instead present
results that apply uniformly to both the full MTF and DSMTF estimators.
\end{remark}

We now prove the above two propositions.

\subsubsection{Proof of Proposition~\ref{prop:determ}}
\begin{proof}[Proof of Proposition~\ref{prop:determ}]
This proof relies on a few intermediate lemmas. 
The first lemma is the following.

\begin{lemma}\label{lem:determ}
Fix a nonnegative integer $r \geq 0.$ Fix any location $i \in [n]$ and any interval $J \in \mathcal{I}_i.$ Recall the ($r$th order) positive and negative bias terms
\begin{equation*}
Bias^{(r)}_{+}(i,J,\theta^*) = \max_{I \in \mathcal{I}_i: I \subseteq J} \big[(P^{(|I|,r)} \theta^*_{I})_i - \theta^*_i\big]
\end{equation*}

\begin{equation*}
Bias^{(r)}_{-}(i,J,\theta^*) = \min_{I \in \mathcal{I}_i: I \subseteq J} [(P^{(|I|,r)} \theta^*_{I})_i - \theta^*_i]
\end{equation*}

Also recall the $r$th order effective noise term  
$$M_i^{(r)} = \max_{I \in \mathcal{I}_i} \big[|(P^{(|I|,r)} \epsilon_{I})_i| \sqrt{|I|}\big].$$

Now define the following intermediate standard error quantity

\begin{equation*}
\tilde{SE}(i,J,\lambda) = \max_{I \in \mathcal{I}_i: I \subseteq J} \big[\frac{M_i^{(r)}}{\sqrt{|I|}} - \frac{2\lambda C_{I,J}}{|I|}\big].
\end{equation*}

Then the following deterministic inequality is  true:

\begin{equation*}
\max_{J \in \mathcal{I}_i} \left(Bias^{(r)}_{-}(i,J,\theta^*) - \tilde{SE}(i,J,\lambda)\right)
\leq \hat{\theta}^{(r,\lambda)}_i - \theta^*_i
\leq
\min_{J \in \mathcal{I}_i} \left(Bias^{(r)}_{+}(i,J,\theta^*) + \tilde{SE}(i,J,\lambda)\right).
\end{equation*}

\end{lemma}

\begin{proof}[Proof of Lemma~\ref{lem:determ}]

For any $i \in [n]$ and any $J \in \mathcal{I}_i$ we have
\begin{align*}
&\hat{\theta}^{(r,\lambda)}_i
\leq
\max_{I \in \mathcal{I}_i : I \subseteq J}
\big[(P^{(|I|,r)} y_{I})_i - \tfrac{2\lambda C_{I,J}}{|I|}\big] \\
&=
\max_{I \in \mathcal{I}_i : I \subseteq J}
\big[(P^{(|I|,r)} \theta^*_{I})_i + (P^{(|I|,r)} \epsilon_{I})_i - \tfrac{2\lambda C_{I,J}}{|I|}\big] \\
&\leq
\max_{I \in \mathcal{I}_i : I \subseteq J}
(P^{(|I|,r)} \theta^*_{I})_i
+
\max_{I \in \mathcal{I}_i : I \subseteq J}
\big[(P^{(|I|,r)} \epsilon_{I})_i - \tfrac{2\lambda C_{I,J}}{|I|}\big].
\end{align*}

Therefore,
\begin{align*}\label{eq:1}
\hat{\theta}^{(r,\lambda)}_i - \theta^*_i
&\leq
\max_{I \in \mathcal{I}_i : I \subseteq J}
\big[(P^{(|I|,r)} \theta^*_{I})_i - \theta^*_i\big] \\
&\quad+
\max_{I \in \mathcal{I}_i : I \subseteq J}
\big[(P^{(|I|,r)} \epsilon_{I})_i - \tfrac{2\lambda C_{I,J}}{|I|}\big] \\
&\leq
\max_{I \in \mathcal{I}_i : I \subseteq J}
\big[(P^{(|I|,r)} \theta^*_{I})_i - \theta^*_i\big]
+
\max_{I \in \mathcal{I}_i : I \subseteq J}
\big[\tfrac{M_i^{(r)}}{\sqrt{|I|}} - \tfrac{2\lambda C_{I,J}}{|I|}\big].
\end{align*}

Similarly,
\begin{align*}
&\hat{\theta}^{(r,\lambda)}_i
\geq
\min_{I \in \mathcal{I}_i : I \subseteq J}
\big[(P^{(|I|,r)} y_{I})_i + \tfrac{2\lambda C_{I,J}}{|I|}\big] \\
&=
\min_{I \in \mathcal{I}_i : I \subseteq J}
\big[(P^{(|I|,r)} \theta^*_{I})_i + (P^{(|I|,r)} \epsilon_{I})_i + \tfrac{2\lambda C_{I,J}}{|I|}\big] \\
&\geq
\min_{I \in \mathcal{I}_i : I \subseteq J}
(P^{(|I|,r)} \theta^*_{I})_i
+
\min_{I \in \mathcal{I}_i : I \subseteq J}
\big[(P^{(|I|,r)} \epsilon_{I})_i + \tfrac{2\lambda C_{I,J}}{|I|}\big].
\end{align*}

Hence,
\begin{align*}
\hat{\theta}^{(r,\lambda)}_i - \theta^*_i
&\geq
\min_{I \in \mathcal{I}_i : I \subseteq J}
\big[(P^{(|I|,r)} \theta^*_{I})_i - \theta^*_i\big] \\
&\quad-
\max_{I \in \mathcal{I}_i : I \subseteq J}
\big[\tfrac{M_i^{(r)}}{\sqrt{|I|}} - \tfrac{2\lambda C_{I,J}}{|I|}\big].
\end{align*}

\end{proof}

Given Lemma~\ref{lem:determ}, to prove Proposition~\ref{prop:determ} it now suffices to show that for any interval $J \in \mathcal{I}_i$, we have $\tilde{SE}(i,J,\lambda) \leq SE^{(r)}(i,J,\lambda).$ This is the content of the next lemma.

\begin{lemma}\label{lem:sd}
Fix any $i \in [n]$ and any interval $J \in \mathcal{I}_i.$ Then we have for all $\lambda > 0$,
\begin{equation*}
\underbrace{\max_{I \in \mathcal{I}_i : I \subseteq J} \big[\frac{M_i^{(r)}}{\sqrt{|I|}} - \frac{2\lambda C_{I,J}}{|I|}\big]}_{\tilde{SE}(i,J,\lambda)} \leq \underbrace{\frac{M_i^{(r)}}{\sqrt{Dist(i,\partial J)}} \mathrm{1}(i \notin \{1,n\}) + \frac{M_i^{(r)}}{\sqrt{|J|}} + \frac{\left(M_i^{(r)}\right)^2}{4 \lambda} + \frac{2\lambda}{|J|}}_{SE^{(r)}(i,J,\lambda)}.
\end{equation*}
\end{lemma}

\begin{proof}[Proof of Lemma~\ref{lem:sd}]
It will be helpful to first solve the optimization problem suggested by the left hand side above. We do this in the following lemma. 
\begin{lemma}(An Optimization Problem)\label{lem:opt}
For a positive integer $N \geq 1$, and $M > 0, \lambda > 0$, consider the optimization problem
\begin{equation*}
OPT(M,\lambda,N) = \max_{1 \leq x \leq N} \big(\frac{M}{\sqrt{x}} - \frac{\lambda}{x}\big).
\end{equation*}

Then, we have
\begin{equation*}
OPT(M,\lambda,N) = 
\begin{cases}
M - \lambda & \text{if} \:\:0 < \lambda < \frac{M}{2} \\
\frac{M^2}{4 \lambda} & \text{if} \:\:\frac{M}{2} \leq \lambda < \frac{M}{2} \sqrt{N} \\
\frac{M}{\sqrt{N}} - \frac{\lambda}{N} \leq \frac{M}{2 \sqrt{N}}  & \text{if} \:\:\frac{M}{2} \sqrt{N} \leq \lambda.
\end{cases}
\end{equation*}

Also, for any fixed $M,\lambda,N$ we have 

\begin{equation*}
OPT(M,\lambda,N) \leq \frac{M^2}{4 \lambda}.
\end{equation*}
\end{lemma}

\begin{proof}[Proof of Lemma~\ref{lem:opt}]
We can write
\begin{equation*}
OPT(M,\lambda,N) = \max_{1 \leq x \leq N} \big(\frac{M}{\sqrt{x}} - \frac{\lambda}{x}\big) = \max_{\frac{1}{\sqrt{N}} \leq a \leq 1} \big(Ma - \lambda a^2\big).
\end{equation*}

So we are simply maximizing a concave quadratic in an interval. The roots of the quadratic are $0$ and $\frac{M}{\lambda}$ and the global maximizer of the quadratic is at $\frac{M}{2 \lambda}.$ This means there are three cases to consider.

\begin{enumerate}
\item $\frac{M}{2 \lambda} > 1$:  This is the case when the global max is larger than $1.$ In this case the maximizer is at $1$ and the value is $M - \lambda.$

\item $\frac{1}{\sqrt{N}} \leq \frac{M}{2 \lambda} \leq 1$: This is the case when the global max is inside the feasible interval. The maximizer is the global max and the value is $\frac{M^2}{4 \lambda}.$

\item $\frac{1}{\sqrt{N}} > \frac{M}{2 \lambda}$: This is the case when the global max is smaller than the smallest feasible value. In this case, the maximizer is at the smallest feasible value which is $\frac{1}{\sqrt{N}}$ and the value is $\frac{M}{\sqrt{N}} - \frac{\lambda}{N}.$ 
\end{enumerate} 

The second display simply follows from the fact that 
$$ OPT(M,\lambda,N) \leq \max_{0 \leq a} \big(Ma - \lambda a^2\big).$$

The proof is finished. 
\end{proof}

We are now ready to finish the proof of Lemma~\ref{lem:sd}. We will need to consider different cases.

\begin{enumerate}
    \item $J \subseteq [2:(n - 1)]$ is an interior interval. In this case there are three values $C_{IJ}$ can take depending on $I$. We note that 
    when $I \subset J$ we have $C_{IJ} = 1$ and we can use Lemma~\ref{lem:opt} (with $\lambda$ replaced by $2 \lambda$); for the case $C_{IJ} = 0$ we have $|I| \geq Dist(i,\partial J)$. Thus we get the bound 
    \begin{equation*}
        \tilde{SE}(i,J,\lambda) \leq \max\{\frac{M_i^{(r)}}{\sqrt{Dist(i,\partial J)}},\frac{(M_i^{(r)})^2}{8 \lambda}, \frac{M_i^{(r)}}{\sqrt{|J|}} + \frac{2\lambda}{|J|}\}.
    \end{equation*}

    \item $J = [1:n].$ In this case, similarly by considering different subcases we can obtain
    \begin{equation*}
        \tilde{SE}(i,J,\lambda) \leq \max \{\frac{(M_i^{(r)})^2}{4 \lambda}, \frac{M_i^{(r)}}{\sqrt{|J|}}\}.
    \end{equation*}

\item $J = [j_1:j_2]$ where $1 = j_1 \leq j_2 < n.$
In this case, one can check by going through the four different subcases, 
    \begin{equation*}
        \tilde{SE}(i,J,\lambda) \leq \max \{\frac{(M_i^{(r)})^2}{4 \lambda}, \frac{M_i^{(r)}}{\sqrt{|J|}} + \frac{\lambda}{|J|}, M_i^{(r)} \mathrm{1}\{i = 1\}\}.
    \end{equation*}

\item $J = [j_1:j_2]$ where $1 < j_1 \leq j_2 = n.$
This case is similar to the above case. 
\end{enumerate}

We can combine the above three displays and finish the proof of Lemma~\ref{lem:sd} and hence of Proposition~\ref{prop:determ}. 
\end{proof}


\subsubsection{Proof of Proposition~\ref{prop:noise}}

\begin{proof}[Proof of Proposition~\ref{prop:noise}]

Fix any interval $I \in \mathcal{I}_i.$ Note that for any fixed $i \in I,$ we can write $(P^{(|I|,r)} \epsilon_{I})_i = \sum_{j \in I} P^{(|I|,r)}_{ij} \epsilon_j$ as a linear combination of $\{\epsilon_j; j \in I\}$, therefore it will be subgaussian. The subgaussian norm squared will be at most the sum of squares of the coefficients $\sum_{j \in I} \left(P^{(|I|,r)}_{ij}\right)^2.$ Now note that 
\begin{equation*}
\sum_{j \in I} \left(P^{(|I|,r)}_{ij}\right)^2 = \sum_{j \in I} P^{(|I|,r)}_{ij} P^{(|I|,r)}_{ji} = \left(P^{(|I|,r)}\right)^2_{ii} = P^{(|I|,r)}_{ii}.
\end{equation*}
In the first equality we used the symmetry of the orthogonal projection matrix $P^{(|I|,r)}$ and in the last equality we used the fact that $P^{(|I|,r)}$ is idempotent.

Now, we claim that there exists a constant $c_r > 0$ only depending on $r$ such that 
$$P^{(|I|,r)}_{ii} \leq \frac{c_r}{|I|}.$$

This claim is a property about the subspace of discrete polynomials and is stated and proved in a stand alone Proposition~\ref{prop:poly}.

The above claim implies that for any $I$ containing $i$, the random variable $\sqrt{|I|} (P^{(|I|,r)} \epsilon_{I})_i$ is Subgaussian with subgaussian norm bounded by $\sigma$ times a constant $c_r$  only depending on $r.$ Applying Lemma~\ref{lem:subgmax} (see below) to the collection
$\Big\{\sqrt{|I|}(P^{(|I|,r)}\epsilon_I)_i : I\in \mathcal{I}_i\Big\}$
with $m=|\mathcal{I}_i|$, $w = c \log n$ for full MTF and $w = c \log \log n$ for DSMTF, yields the stated bound on $M_i^{(r)}$. 
\end{proof}

\begin{lemma}[Max of finitely many subgaussians]\label{lem:subgmax}
Let $X_1,\dots,X_m$ be random variables such that each $X_k$ is mean-zero
subgaussian with variance proxy at most $v^2$, i.e.
\[
\mathbb{E}\exp(tX_k)\le \exp\!\Big(\frac{v^2 t^2}{2}\Big)
\qquad \forall\, t\in\mathbb{R},\ \forall\,k\in[m].
\]
Then for any $u>0$,
\[
\mathbb{P}\!\left(\max_{1\le k\le m}|X_k|\ge u\right)
\le 2m\exp\!\Big(-\frac{u^2}{2v^2}\Big).
\]

Consequently, for any $w > 0$, with probability at least $1- \exp(-w)$ we have 
\[
\max_{1\le k\le m}|X_k|
\le v\sqrt{2\log(2m)+2w}.
\]
\end{lemma}



\begin{proof}
Fix $k$. By the subgaussian mgf bound, a standard Chernoff argument gives
\[
\mathbb{P}(X_k\ge u)\le \exp\!\Big(-\frac{u^2}{2v^2}\Big),
\qquad
\mathbb{P}(X_k\le -u)\le \exp\!\Big(-\frac{u^2}{2v^2}\Big).
\]
Hence $\mathbb{P}(|X_k|\ge u)\le 2\exp(-u^2/(2v^2))$. A union bound over
$k\in[m]$ yields
\[
\mathbb{P}\!\left(\max_{1\le k\le m}|X_k|\ge u\right)
\le \sum_{k=1}^m \mathbb{P}(|X_k|\ge u)
\le 2m\exp\!\Big(-\frac{u^2}{2v^2}\Big).
\]
\end{proof}

\end{proof}

	\section{A Fact about Discrete Polynomials}\label{sec:fact1}

\begin{proposition}\label{prop:poly}
Fix an integer $r \geq 0$. For any positive integer $m$, define $I = [m].$ Define the (Vandermonde) matrix $B \in \R^{m \times (r + 1)}$ obtained by stacking together columns 
$$B = \left(b_0:b_1:\dots:b_r\right)$$
where for each $j \in [0:r]$ we define 
$$b_j = (1^j,2^j,\dots,m^{j})^{T}.$$
We call $b_j$ as the (discrete) polynomial vector of degree $j$ on $I.$  Define $P^{(r)}$ to be the orthogonal projection matrix on to the subspace $S^{(r)}$ of $r$th degree polynomials or more precisely, $$S^{(r)} = Span(b_0,\dots,b_r).$$ Then there exists a constant $C_r > 0$ only depending on $r$ such that 
\begin{equation*}
\|Diag(P^{(r)})\|_{\infty} \leq \frac{C_r}{m}.
\end{equation*}

\end{proposition}

\begin{proof}
It suffices to prove the statement for $m > r.$
Let the vectors $\tilde{b}_0,\dots,\tilde{b}_{r}$ be an orthogonal basis of $S^{(r)}$ obtained by performing Gram Schmidt orthogonalization on the ordered set $\{b_0,\dots,b_r\}.$ We can think of $\tilde{b}_0,\dots,\tilde{b}_{r}$ as a set of (discrete) orthogonal polynomials, infact these can be thought of as (discrete) Legendre polynomials. We can now write the orthogonal projection matrix $P^{(r)}$ as follows:
$$P^{(r)} = \sum_{j = 0}^{r} \frac{\tilde{b}_j  \tilde{b}_j^{T}}{\|\tilde{b_j}\|^2}.$$

Fix an $i \in [m]$ and we can write the $i$ th diagonal element of $P^{(r)}$ as
\begin{equation*}
P^{(r)}_{ii} = e_i^{T} P^{(r)} e_i = \sum_{j = 0}^{r} \frac{(e_i^{T} \tilde{b}_j)^2}{\|\tilde{b_j}\|^2}.
\end{equation*}

In the above, $e_i$ is the $i$th canonical basis vector in $\R^m.$

The following two lemmas will now finish the proof. 

\begin{lemma}\label{lem:poly1}
Fix non negative integers $r$ and $m > r.$ There exists a positive constant $c_r$ only depending on $r$ such that for each $0\le j\le r$,
\begin{equation}\label{eq:poly1}
\|\tilde{b_j}\|^2 \geq c_{r} m^{2j + 1}.
\end{equation}
\end{lemma}

\begin{lemma}\label{lem:poly2}
Fix non negative integers $r$ and $m.$ There exists a positive constant $c_r$ only depending on $r$ such that such that for each $0\le j\le r$, 
\begin{equation}\label{eq:poly2}
\|\tilde{b}_j\|_{\infty} \leq c_{r} m^{j}.
\end{equation}
\end{lemma}

\end{proof}

Now we give proofs of both these lemmas. Within these proofs $c_r$ will denote a generic positive constant only depending on $r$ and whose exact value might change from line to line.

\begin{proof}[Proof of Lemma~\ref{lem:poly1}]
If $j = 0$, then $\tilde{b}_j = b_j$ and there is nothing to prove since $b_0$ is the all $1$  vector. So fix any $j \in [r].$ Note that since we are performing Gram Schmidt orthogonalization, we can write $\tilde{b}_j$ as a linear combination of $b_0,b_1,\dots,b_j$ where the coefficient of $b_j$ is $1$, i.e,
\begin{equation*}
\tilde{b}_j = a_0 b_0 + a_1 b_1 + \dots + a_{j - 1} b_{j - 1} + a_j b_j
\end{equation*}  
where $a_j = 1.$ Therefore, we can write 
\begin{align*}
&\|\tilde{b_j}\|^2 = \sum_{i = 1}^{m} \left(a_0 + a_1 i + a_2 i^2 + \dots + a_j i^{j}\right)^2 = \sum_{i = 1}^{m} \sum_{u = 0}^{j} \sum_{v = 0}^{j} a_u i^{u} a_v i^{v} = \\& \sum_{u = 0}^{j} \sum_{v = 0}^{j} \underbrace{a_u m^{u + 1/2}}_{x_u} \underbrace{a_v m^{v + 1/2}}_{x_v} \underbrace{\left(\frac{1}{m} \sum_{i = 1}^{m} (\frac{i}{m})^{u + v}\right)}_{Q_{uv}} = \\&
x^{T} Q x. 
\end{align*}

In the above step, we wrote $\|\tilde{b_j}\|^2$ as a quadratic form in a vector $x = (x_0,\dots,x_j) \in \R^{j + 1}.$

It will help to think of $Q$ in the block matrix form as follows. 

\[
Q=
\left[
\begin{array}{c|c}
Q_{11} & Q_{12} \\
\hline
Q_{21} & Q_{22}
\end{array}
\right]
\]

where $Q_{11} = Q_{[0:(j - 1),0:(j - 1)]} \in \R^{j \times j}$ and $Q_{22} = Q_{jj} \in \R.$
We can now write 
\begin{equation*}
x^{T} Q x = y^{T} Q_{11} y + 2 y^{T} Q_{12} x_j + x_j^2 Q_{jj}
\end{equation*}
where $y = x[0:(j - 1)].$

We now claim that $Q$ is strictly positive definite, we will prove this at the end. This will imply that its leading principal minor $Q_{11}$ is also strictly positive definite. Thus, viewing $x^{T} Q x$ as a function of $y$ as above (keeping $x_j$ fixed), we see that it is a strongly convex function of $y$ (since $Q_{11}$ is positive definite) and hence has a unique minima. By differentiating and solving for $y$, it can be checked that $y^* = -Q_{11}^{-1} Q_{12} x_j$ is the minima and the minimum value is $x_j^2 \left(Q_{jj} - Q_{21} Q_{11}^{-1} Q_{12}\right).$ This gives us the lower bound 
\begin{equation*}
x^{T} Q x \geq x_j^2 \left(Q_{jj} - Q_{21} Q_{11}^{-1} Q_{12}\right).
\end{equation*}

Note that $x_j^2 = a_j^{2} m^{2j + 1} = m^{2j + 1}$ since $a_j = 1.$ Therefore, to show~\eqref{eq:poly1} it suffices to show that 

\begin{equation}\label{eq:polyreq1}
\left(Q_{jj} - Q_{21} Q_{11}^{-1} Q_{12}\right) \geq c_r > 0.
\end{equation}

Now, using linear algebra terminology, $\left(Q_{jj} - Q_{21} Q_{11}^{-1} Q_{12}\right)$ is the Schur complement of $Q_{11}$ and using the well known block matrix inversion formula we obtain 

\begin{equation*}
(Q^{-1})_{jj} = \frac{1}{Q_{jj} - Q_{21} Q_{11}^{-1} Q_{12}}
\end{equation*}

Moreover, we also have 

\begin{equation*}
(Q^{-1})_{jj} \leq \lambda_{max} (Q^{-1}) = \frac{1}{\lambda_{min} (Q) }.
\end{equation*}

where $\lambda_{max},\lambda_{min}$ denote the maximum and minimum eigenvalue respectively. Therefore, to show~\eqref{eq:polyreq1}, it suffices to show that for all $m \geq 1$,

\begin{equation}\label{eq:polyreq2}
\lambda_{min} (Q) \geq c_r > 0.
\end{equation}

Let $U_m$ be a discrete random variable uniform on the set $\{\frac{1}{m},\dots,\frac{m}{m}\}$ and $U$ denote a $U(0,1)$ random variable. Then, we have $U_m$ converging to $U$ weakly; i.e, 
\begin{equation*}
U_m \xrightarrow[m \to \infty]{{\rm law}} U.
\end{equation*}

Note that $Q$ is product moment matrix of the random vector $U_m^{(vec)}= (U_m^{1},\dots,U_m^{j}).$ That is, 
$$Q_{uv} = \E\: U_m^{u} U_m^{v}.$$

Define $Q^{pop}$ to be the population version of $Q$; more precisely, define
$$Q^{POP}_{uv} = \E\: U^{u} U^{v}.$$

By the continuous mapping theorem, we can conclude that $$Q \xrightarrow[m \to \infty]  \:Q^{pop}.$$

Since $\lambda_{min}$ is a continuous function on the space of positive definite matrices, we further can write

\begin{equation*}
\lambda_{min} (Q) \xrightarrow[m \to \infty] \:\lambda_{min}(Q^{pop}).
\end{equation*}

Now we claim that $Q^{pop}$ is positive definite and hence	there exists a constant $c_r > 0$ such that $\lambda_{min}(Q^{pop}) > c_r.$ Therefore, there exists a positive integer $M \geq 1$ such that $\lambda_{min} (Q) \geq \frac{c_r}{2}$ for all $m \geq M.$ Combined with the fact that $\lambda_{min}(Q) > 0$ for all $m > r$, this proves~\eqref{eq:polyreq2} and in turn proves~\eqref{eq:polyreq1} which in turn proves~\eqref{eq:poly1}.

All that remains is to show that $Q^{pop}$ is positive definite and so is $Q$ for all $m > r.$

Take any vector $v \in \R^{j + 1}$ and consider the quadratic form $v^{T} Q^{pop} v.$ Suppose

\begin{equation*}
v^{T} Q^{pop} v = E \sum_{u = 0}^{j} \sum_{l = 0}^{j} v_u v_l U^{u + l} =  \E \left(\sum_{u = 0}^{j} v_u U^u \right)^2 = 0
\end{equation*}

This implies that the random variable $\sum_{u = 0}^{j} v_u U^u = 0$ almost surely. If any of the $v_u$'s are non zero then the above is a polynomial of degree at most $j$ and hence cannot be $0$ almost surely in $U.$ Therefore, it has to be the case that the vector $v$ is zero. This shows that $Q^{pop}$ is positive definite.

Similarly, suppose 
\begin{equation*}
v^{T} Q v =  \E \left(\sum_{u = 0}^{j} v_u U_m^u \right)^2 = 0
\end{equation*}

The above means that the polynomial $p(x) = \sum_{u = 0}^{j} v_u x^{u}$ has atleast $m$ roots $\{\frac{1}{m},\dots,\frac{m}{m}\}.$ However, $p(x)$ is a polynomial of degree $j \leq r.$ Therefore, if $m > r$ then this is a contradiction unless $v$ is the zero vector. This shows that if $m > r$, then $Q$ is also positive definite.

\end{proof}

\begin{proof}[Proof of Lemma~\ref{lem:poly2}]

If $j=0$, then $\tilde b_0=b_0=\mathbf 1$, so $\|\tilde b_0\|_\infty=1\le c_r$ and
there is nothing to prove. Fix any $j\in[r]$. Let $P^{(j-1)}$ denote the orthogonal
projection onto $\mathrm{span}(b_0,\dots,b_{j-1})$. By Gram--Schmidt,
$\tilde b_j = b_j - P^{(j-1)} b_j$. Therefore, by the triangle inequality,
\[
\|\tilde b_j\|_{\infty} \le \|b_j\|_{\infty} + \|P^{(j-1)} b_j\|_{\infty}.
\]
It is immediate that $\|b_j\|_\infty = m^j$. Hence it suffices to show that
$\|P^{(j-1)} b_j\|_\infty \le c_r m^j$.

Next, note that $\|b_j\|_2^2 \le c_r m^{2j+1}$, and therefore
$\|P^{(j-1)} b_j\|_2 \le \|b_j\|_2 \le c_r m^{j+1/2}$. Define
\[
v := \frac{P^{(j-1)} b_j}{\|P^{(j-1)} b_j\|_2}.
\]
It suffices to show that
\begin{equation}\label{eq:polyreq3}
  \|v\|_\infty \le \frac{c_r}{\sqrt m}.
 \end{equation}

Let $(L_0,\dots,L_r)$ be the orthonormal Legendre polynomials on $[0,1]$, so that
$\int_0^1 L_a(x)L_b(x)\,dx = \mathbf 1(a=b)$ for $0\le a,b\le r$. Moreover, for fixed
$r$ there is a constant $c_r<\infty$ such that
\[
\max_{0\le a\le r}\bigl\{\|L_a\|_\infty,\ \|L_a'\|_\infty\bigr\}\le c_r.
\]
Since $\{L_0,\dots,L_{j-1}\}$ is a basis of the space of polynomials of degree at most $j-1$,
there exists an invertible matrix $A\in\mathbb R^{j\times j}$ such that
$(L_0(x),\dots,L_{j-1}(x))=(1,x,\dots,x^{j-1})A$. Evaluating at the grid points
$x_i=i/m$ shows that the column space of the monomial design matrix $B$ equals the column
space of the $m \times j$ matrix $(L_0(x_i),\dots,L_{j-1}(x_i))$.

Therefore, we can write for each $i \in [m]$,

\begin{equation*}
v_i = \sum_{u = 0}^{j - 1} a_u L_{u}(x_i).
\end{equation*}

Note that
\begin{equation*}
\|v\|_{\infty} \leq \left(\max_{0 \leq u \leq (j - 1)} |a_u|\right) \sum_{u = 0}^{j - 1} \|L_u\|_{\infty}
\end{equation*}

and hence 
\begin{equation*}
\|v\|_{\infty} \leq c_r \left(\max_{0 \leq u \leq (j - 1)} |a_u|\right) \leq c_r \sqrt {\sum_{u = 0}^{j - 1} a_u^{2}}.
\end{equation*}

Therefore, to show~\eqref{eq:polyreq3} it suffices to show 
\begin{equation}\label{eq:polyreq4}
\sum_{u = 0}^{j - 1} a_u^{2} \leq \frac{c_r}{m}.
\end{equation}

Define the function
\[
f(x) := \sum_{u=0}^{j-1} a_u L_u(x).
\]
Since $\{L_u\}_{u=0}^{j-1}$ are orthonormal in $L^2([0,1])$, we have
\[
\int_0^1 f(x)^2\,dx = \sum_{u=0}^{j-1} a_u^2.
\]

By construction, $v_i = f(x_i)$ for $x_i=i/m$, and $v$ is a unit vector in
$\mathbb R^m$. Hence
\[
\frac{1}{m}\sum_{i=1}^m f(x_i)^2
= \frac{1}{m}\sum_{i=1}^m v_i^2
= \frac{1}{m}.
\]

Therefore,
\[
\Bigg|\sum_{u=0}^{j-1} a_u^2 - \frac{1}{m}\Bigg|
=
\Bigg|\int_0^1 f(x)^2\,dx - \frac{1}{m}\sum_{i=1}^m f(x_i)^2\Bigg|.
\]

Let $x_0=0$. Writing the difference as a Riemann sum error,
\[
\Bigg|\int_0^1 f(x)^2\,dx - \frac{1}{m}\sum_{i=1}^m f(x_i)^2\Bigg|
=
\Bigg|\sum_{i=1}^m \int_{x_{i-1}}^{x_i}
\bigl(f(x)^2 - f(x_i)^2\bigr)\,dx\Bigg|.
\]
By the mean value theorem,
\[
\le \sum_{i=1}^m \int_{x_{i-1}}^{x_i}
\|(f^2)'\|_\infty |x-x_i|\,dx
\le \|(f^2)'\|_\infty \sum_{i=1}^m \frac{1}{m^2}
= \frac{1}{m}\|(f^2)'\|_\infty.
\]

Moreover,
\begin{align*}
\|(f^2)'\|_{\infty}
&= \Big\|\sum_{u=0}^{j-1}\sum_{v=0}^{j-1} a_u a_v (L_uL_v)'\Big\|_\infty \\
&\le \sum_{u=0}^{j-1}\sum_{v=0}^{j-1} |a_u||a_v|\,\|(L_uL_v)'\|_\infty.
\end{align*}
Since $(L_uL_v)' = L_u' L_v + L_u L_v'$ and
$\max_{0\le u\le r}\{\|L_u\|_\infty,\|L_u'\|_\infty\}\le c_r$, we have
$\|(L_uL_v)'\|_\infty \le c_r$ for all $0\le u,v\le j-1$. Hence,
\[
\|(f^2)'\|_\infty \le c_r\Big(\sum_{u=0}^{j-1} |a_u|\Big)^2
\le c_r \sum_{u=0}^{j-1} a_u^2,
\]
where the last inequality uses Cauchy--Schwarz and that $j\le r$ is fixed (so constants
depending on $j$ can be absorbed into $c_r$).

Therefore, combining with the previous display yields
\[
\Big|\sum_{u=0}^{j-1} a_u^2 - \frac{1}{m}\Big|
\le \frac{c_r}{m}\sum_{u=0}^{j-1} a_u^2.
\]
Let $S:=\sum_{u=0}^{j-1} a_u^2$. Then $S-\frac1m \le \frac{c_r}{m}S$, i.e.
\[
S\Big(1-\frac{c_r}{m}\Big)\le \frac1m.
\]
Thus for all $m>c_r$ we have $S \le \frac{1}{m(1-c_r/m)} \le \frac{2}{m}$, and hence this proves \eqref{eq:polyreq4}.
\end{proof}

    \section{An Approximation Result for Bounded Variation Sequences}\label{sec:fact2}

We prove the following proposition about approximation of a bounded variation vector by a piecewise polynomial vector. 
\begin{proposition}\label{prop:piecewise}

Fix an integer $r \geq 1$ and $\theta \in \R^n$, and let $\TV^{(r)}(\theta) \coloneqq V.$ For any $\delta > 0$, there exists an interval partition $\pi$ of $[n]$ such that 
\newline
a) $\TV^{(r)}(\theta_{I}) \leq V \delta \:\:\:\: \forall I \in \pi$, \newline
b)\:For any $i \in [n]$, we have $$
\max\{|Bias_{+}^{(r - 1)}(i,J_i,\theta)|,|Bias_{-}^{(r - 1)}(i,J_i,\theta)|\} \leq C_r V \delta$$ where $J_i$ is the interval within the partition $\pi$
which contains $i$,\newline
c) $|\pi| \leq C_r \delta^{-1/r}.$
\newline
d) There exist absolute constants $0<c_1\le c_2$ such that for any integer $\ell\ge 0$,
\begin{equation*}
    \bigl|\{I\in\pi: c_1\tfrac{n}{2^\ell}\le |I|\le c_2\tfrac{n}{2^\ell}\}\bigr|
\le C_r \min\Big\{\tfrac{2^{-\ell(r-1)}}{\delta},\,2^\ell\Big\}.
\end{equation*}

\end{proposition}



\begin{remark}
The proof uses a recursive partitioning scheme proposed in~\cite{chatterjee2021adaptive}; see Proposition $8.9$ therein, which further can be thought of as a discrete version of a classical analogous result for functions defined on the continuum in~\cite{BirmanSolomjak67}.	 
\end{remark}

\begin{proof}[Proof of Proposition~\ref{prop:piecewise}]
We first need a lemma quantifying the error when approximating an arbitrary vector $\theta$ by its polynomial projection.
\begin{lemma}{\label{lem:approxpoly}}
Fix any integer $r \geq 0.$ For any $n \geq 1$ and for any $\theta \in \R^n$ we have
\begin{equation}
|\theta - P^{(n,r)} \theta|_{\infty} \leq C_r \TV^{(r + 1)}(\theta).
\end{equation}
\end{lemma}

\begin{proof}
Let us denote $P^{(n,r)}$ by $P^{(r)}$ within this proof and let us denote the subspace of discrete $r$th order polynomials on $[n]$ by $S^{(r)}.$

Write the projection matrix onto the orthogonal complement of $S^{(r)}$ (denote by $S^{(r,\perp)}$) by $P^{\perp}.$ We want to bound $|\theta - P^{(r)} \theta|_{\infty} = |P^{\perp} \theta|_{\infty}.$

Note that $S^{(r)}$ is precisely the null space of the matrix $D^{(r + 1)}.$ Therefore, $S^{(r,\perp)}$ becomes the row space of the matrix $D^{r + 1}.$ In case, $D^{(r + 1)}$ was full row rank (which it is not), then by standard least squares theory we could have written $$P^{\perp} \theta = (D^{(r + 1)})^{t} \big(D^{(r + 1)} (D^{(r + 1)})^{t}\big)^{-1} D^{(r + 1)} \theta.$$ 

Since $D^{(r + 1)}$ is not of full row rank we have to modify the above slightly. Using the concept of generalized inverse, the above display still holds with the inverse replaced by a generalized inverse. The main point in all of this is that entries of $P^{\perp} \theta$ can be written as linear combinations of the entries of $D^{(r + 1)} \theta.$ In fact, the above display can be simplified as
$$P^{\perp} \theta = (D^{(r + 1)})^{+} D^{(r + 1)} \theta$$ where $(D^{(r + 1)})^{+}$ is the appropriate matrix from above; also known as the Moore Penrose Inverse of $D^{(r + 1)}.$

We now claim that $|(D^{(r + 1)})^{+}|_{\infty} \leq C_r n^{r}.$ This will finish the proof by using 
$$|P^{\perp} \theta|_{\infty} \leq |(D^{(r + 1)})^{+}|_{\infty} |D^{(r + 1)} \theta|_{1} \leq C_r \TV^{(r + 1)}(\theta).$$

It remains to prove the claim. We will use certain existing representations of $(D^{(r + 1)})^{+}$ for this.

By  Lemma  13 in~\cite{wang2016trend}, we have that $(D^{(r + 1)})^{+} =  \frac{n^{r}}{r!} P^{\perp} H $  where  $H$  consists of  the last $n-r - 1$ columns of the so-called $r$th order falling factorial basis matrix. Further, expressions for the falling factorial basis are given in~\cite{wang2014falling}. We have that for  $i \in \{1,\ldots,n\}$ and  $j \in  \{ 1,\ldots,n-r-1\}$,
\[
H_{i,j}   =  h_j(i/n),
\]
where 
\[
h_j(x) =   \prod_{l=1}^{r-1}   \left(   x -  \frac{j+l}{n} \right)1_{  \left\{ x\geq  \frac{j+r-1}{n} \right\}   }.
\]

Take $e_i$, the $i$th element of the canonical basis in  $R^{n-r - 1}$. Using the expression for $(D^{(r + 1)})^{+}$ we can write
\[
\begin{array}{lll}
\frac{1}{n^{r}} \|e_i^{\top} (D^{(r + 1)})^{+}\|_{\infty}  & \leq&     \| P^{\perp} e_i\|_{1} \| H\|_{\infty}/r!\\ 
& \leq&  \left(  \|e_i \|_{1} +     \| P^{(r)} e_i \|_{1}  \right)    \| H\|_{\infty}/(r-1)!  \\
& \leq& \left[  1  +     \| P^{(r)} e_i \|_{1}  \right]  /(r-1)!  \\
\end{array}
\]
where the  first  inequality follows from H\"{o}lder's inequality, the second from the triangle inequality and the last by the definition of $H$. 

Next let  $v_1,\ldots, v_{r + 1}$ be an  orthonormal basis of  $S^{(r)}$. Then
\[
\displaystyle 	 \| P^{(r)} e_i \|_{1}    \,=\,   \left\|  \sum_{j=1}^{r + 1}    (e_i^{\top} v_j) v_j \right\|_1 \,\leq\, \sum_{j=1}^{r + 1}    \vert   (e_i^{\top} v_j)   \vert \|v_j\|_1  \,\leq \,  \sum_{j=1}^{r + 1}    \|  v_j  \|_{\infty} \|v_j\|_1 \,\leq\,  \sum_{j=1}^{r + 1}    \|  v_j  \|_{\infty} n^{1/2}.
\]
Now, Lemmas~\ref{lem:poly1},~\ref{lem:poly2}  tell us that $\|  v_j  \|_{\infty} \leq \frac{C_r}{\sqrt{n}}$ for all $j \in [r + 1].$

All in all, the above arguments finally imply our claim
\begin{equation}
\label{eqn:pseudo_inv}
\|(D^{(r + 1)})^{+}\|_{\infty} \leq C_r n^r.
\end{equation}
\end{proof}

We are now ready to proceed with the proof of Proposition~\ref{prop:piecewise}.
For the sake of clean exposition, we assume $n$ is a power of $2.$ The reader can check that the proof holds for arbitrary $n$ as well (by adopting a convention for splitting an interval by half). For an  interval $I \subseteq [n]$, let us define $$\mathcal{M}(I) = \TV^{(r)}(\theta_{I}) = |I|^{r - 1} |D^{(r)} \theta_{I}|_1$$ where $|I|$ is the cardinality of $I$ and $\theta_{I}$ is the vector $\theta$ restricted to the indices in $I.$ Let us now perform recursive dyadic partitioning of $[n]$ according to the following rule. Starting with the root vertex $I = [n]$ we check whether $\mathcal{M}(I) \leq V \delta.$ If so, we stop and the root becomes a leaf. If not, divide the root $I$ into two equal nodes or intervals $I_1 = [n/2]$ and $I_2 = [n/2 + 1 : n].$ For $i = 1,2$ we now check whether $\mathcal{M}(I_j) \leq V \delta$ for $j = 1,2.$ If so, then this node becomes a leaf otherwise we keep partitioning. When this scheme halts, we would be left with a Recursive Dyadic Partition $\pi$ of $[n]$ which are constituted by disjoint intervals. Let's say there are $k$ of these intervals denoted by $B_1,\dots,B_{k}.$ 
By construction, we have $\mathcal{M}(B_i) \leq V \delta$ which proves part (a). \newline

One of the $B_1,\dots,B_{k}$ would contain $i.$ We denote this interval by $J_i.$ Let $I$ be any subset of $J_i$ containing $i.$ Since $\TV^{(r)}(\theta_{J_i}) \leq V \delta$ we must have $$\TV^{(r)}(\theta_{I}) \leq V \delta.$$ We can now apply Lemma~\ref{lem:approxpoly} to $\theta_{I}$ to obtain  
$$|\theta_I - P^{(|I|,r - 1)} \theta_{I}|_{\infty} \leq C_r \TV^{(r)}(\theta_{I}) \leq C_r V \delta.$$
Since this bound holds uniformly for all such $I$, we prove part (b).\newline

Let us rewrite $\mathcal{M}(I) = (\frac{|I|}{n})^{r - 1} n^{r - 1} |D^{(r)} \theta_{I}|_1.$ Note that for arbitrary disjoint intervals $B_1,B_2,\dots,B_{k}$ we have by sub-additivity of the $\TV^{(r)}$ functional,
\begin{equation}\label{eq:subadd}
\sum_{j \in [k]} n^{r -1} |D^{(r)} \theta_{B_j}|_1 \leq \TV^{(r)}(\theta) = V.
\end{equation}
The entire process of obtaining our recursive partition of $[n]$ actually happened in several rounds. In the first round, we possibly partitioned the interval $I = [n]$ which has size proportion $|I|/n = 1 = 2^{-0}.$ In the second round, we possibly partitioned intervals having size proportion $2^{-1}$. 
In general, in the $\ell$ th round, we possibly partitioned 
intervals having size proportion $2^{-\ell}$. Let $n_\ell$ be the number of intervals with size proportion $2^{-\ell}$ that 
we divided in round $\ell$. Let us count and give an upper bound on $n_\ell.$ If we indeed partitioned $I$ with size proportion $2^{-\ell}$ then by construction this means 
\begin{equation}
n^{r - 1} |D^{(r)} \theta_{I}|_1 > \frac{V \delta}{2^{-\ell(r - 1)}}.
\end{equation}
Therefore, by sub-additivity as in~\eqref{eq:subadd} we can conclude that the number of such divisions is at most $\frac{2^{-\ell(r - 1)}}{\delta}.$ On the other hand, note that clearly the number of such divisions is bounded above by $2^{\ell}.$ Thus we conclude
\begin{equation*}
n_\ell \leq \min\{\frac{2^{-\ell(r - 1)}}{\delta},2^\ell\}.
\end{equation*}
This proves part (d).

Therefore, we can assert that 
\begin{equation}
k = 1 + \sum_{l = 0}^{\infty} n_\ell \leq  \sum_{\ell = 0}^{\infty} \min\{\frac{2^{-\ell(r - 1)}}{\delta},2^\ell\} \leq C_r \delta^{-1/r}.
\end{equation}
In the above, we set $n_\ell = 0$ for $\ell$ exceeding the maximum number of rounds of division possible. The last summation can be easily performed as there exists a nonnegative integer $2^{\ell^*} = O( \delta^{-1/r})$ such that 
\begin{equation*}
\min\{\frac{2^{-\ell(r - 1)}}{\delta},2^\ell\} = 
\begin{cases}
2^\ell, & \text{for} \:\:\ell < \ell^* \\
\frac{2^{-\ell(r - 1)}}{\delta} & \text{for} \:\:\ell \geq \ell^*
\end{cases}
\end{equation*}
This proves part (c) and finishes the proof. 
\end{proof}

	\section{Proof of Theorem~\ref{thm:mainada}}\label{sec:localrateproofs}

We first bound the bias term for Hölder smooth functions. 

\begin{lemma}[Local Bias Control]\label{lem:bias}
Suppose $\theta^* \in C^{r_0, \alpha_0}(J)$ with Hölder constant bounded by $L$, for an interval $J \subseteq [n]$ containing $i.$ Then we have the following bound on the bias:
\begin{equation*}
    \max\{|Bias^{(r)}_{+}(i,J,\theta^*)|,|Bias^{(r)}_{-}(i,J,\theta^*)|\} \leq C_r L \big(\frac{|J|}{n}\big)^{\beta}
\end{equation*}
where $\beta = r_0 + \alpha_0.$
\end{lemma}


\begin{proof}
We write the proof for $r_0 = r$; the entire argument goes through verbatim for any $r_0 < r$ as well. Throughout this proof, we will go back and forth between discrete intervals and real intervals (denoted in bold). For any discrete interval $I = [l_1:l_2] \subseteq [n]$, the corresponding real interval is $\mathbf{I} = [\frac{l_1}{n},\frac{l_2}{n}]$ and vice versa.

We first need some preparatory results. Let $J$ be the discrete interval $[i:j] \subseteq [n].$ For any discrete sub interval $I = [u:v] \subseteq J$ we can define the sequence $Tayl(\theta^*,I,r) \in \R^{|I|}$ which is basically the $r$th order Taylor expansion of $\theta^*$ about the initial point in $I.$ To be precise, recall that  $\theta^*_i = f^*(\frac{i}{n})$ are evaluations of some underlying function $f: [0,1] \rightarrow \mathbb{R}$ such that $f \in C^{r, \alpha_0}(\mathbf J)$ for the (real) interval $\mathbf J = [\frac{i}{n},\frac{j}{n}] \subseteq [0,1].$ For the (real) interval $\mathbf{I} = [\frac{u}{n},\frac{v}{n}] \coloneqq [a,b]$ we denote its Taylor Series approximation $f_{Tayl,\mathbf{I}} : \mathbf{I} \rightarrow \mathbb{R}$ as follows:
$$f_{Tayl,\mathbf{I}}(x) = \sum_{l = 0}^{r} \frac{f^{(l)}(a)}{l!} (x - a)^{l}.$$

We can now define $Tayl(\theta^*,I,r) \in \R^{|I|}$ to be the evaluations of $f_{Tayl,\mathbf{I}}$ on the discrete grid within $\mathbf{I}.$

We observe that since $f \in C^{r,\alpha_0}(\mathbf{I})$, by Taylor's theorem, $f$ can be written as 

$$f(x) = \sum_{l = 0}^{r - 1} \frac{f^{(l)}(a)}{l!} (x - a)^{l} + \frac{f^{(r)}(\xi)}{r!} (x - a)^{r}$$ for some $\xi \in [a,x].$

Therefore, for any $x \in \mathbf{I}$, we have

$$|f(x) - f_{Tayl,\mathbf{I}}(x)| \leq C_r |f^{(r)}(\xi) - f^{(r)}(a)| |b - a|^{r} \leq C_r |b - a|^{r + \alpha_0} = C_r |b - a|^{\beta}.$$

When we apply this argument to $\theta^*$ inside the discrete interval $I$, we obtain
\begin{equation}\label{eq:taylor}
    [\theta^*_{I} - Tayl(\theta^*,I,r)]_{\infty} \leq C_r L \big(\frac{|I|}{n}\big)^{\beta} \leq C_r L \big(\frac{|J|}{n}\big)^{\beta}. 
\end{equation}

Now for the discrete interval $I$, consider the matrix $[I_{|I|} - P^{(|I|,r)}]$ where $I_{|I|}$ is the $|I| \times |I|$ identity matrix. We denote its $\ell_{\infty,1}$ matrix norm

$$[I_{|I|} - P^{(|I|,r)}]_{row,\ell_1}  = \max_{1 \leq i \leq |I|} \sum_{1 \leq j \leq |I|} |[I_{|I|} - P^{(|I|,r)}]_{ij}|.$$

We now claim that there exists a constant $C_r$ only depending on $r$ such that 

\begin{equation}\label{eq:matnormbd}
    [I_{|I|} - P^{(|I|,r)}]_{row,\ell_1} \leq C_r.
\end{equation}

We can show this by arguing as follows:
\begin{equation*}
    \sum_{1 \leq j \leq |I|} |[P^{(|I|,r)}]_{ij}| \leq \big(\sum_{1 \leq j \leq |I|} [P^{(|I|,r)}]^2_{ij}\big)^{1/2} |I|^{1/2} = \sqrt{P^{(|I|,r)}_{ii}}\,|I|^{1/2} \leq C_r
\end{equation*}
where in the first inequality we used Cauchy--Schwarz, in the equality we used the fact that $P^{(|I|,r)}$ is symmetric and idempotent and in the last inequality we used Proposition~\ref{prop:poly}.

Now note that by triangle inequality for norms,
$$[I_{|I|} - P^{(|I|,r)}]_{row,\ell_1} \leq 1 + [P^{(|I|,r)}]_{row,\ell_1}$$ which proves~\eqref{eq:matnormbd}.

We are now ready to give the proof. 

Take any subinterval $I \subseteq J$ such that $i \in I.$ 
We can write
\begin{align*}
    \big[(P^{(|I|,r)} \theta^*_{I})_i - \theta^*_i\big] &= - \big([I_{|I|} - P^{(|I|,r)}] \theta^*_{I}\big)_i = - \big([I_{|I|} - P^{(|I|,r)}] [\theta^*_{I} - Tayl(\theta^*,I,r)]\big)_i \\& \leq \big([I_{|I|} - P^{(|I|,r)}] [\theta^*_{I} - Tayl(\theta^*,I,r)]\big)_{\infty} \\& \leq [I_{|I|} - P^{(|I|,r)}]_{row,\ell_1} [\theta^*_{I} - Tayl(\theta^*,I,r)]_{\infty} \\& \leq C_r L \big(\frac{|J|}{n}\big)^{r + \alpha_0}.
\end{align*}

In the above, in the second equality we used the fact that $Tayl(\theta^*,I,r)$ is a discrete $r$th degree polynomial, in the second inequality we used Hölder's inequality and in the last inequality we used both~\eqref{eq:taylor} and~\eqref{eq:matnormbd}. This finishes the proof. 
\end{proof}

We are now ready to give the proof. 

\begin{proof}[Proof of Theorem~\ref{thm:mainada}]
We consider the DSMTF estimator here. The same proof works for the full MTF as well. Hence $\mathcal{I}_{i_0}$ consists of symmetric intervals of all scales centred at $i = i_0.$
Combining~\eqref{eq:thmbound} and Lemma~\ref{lem:bias} we can write
\begin{align}\label{eq:biasvar}
    \hat{\theta}^{(r,\lambda)}_{i_0} - \theta^*_{i_0} \leq \min_{J \in \mathcal{I}_{i_0}: J \subseteq [i_0 \pm s_0]} \big(C_r L \big(\frac{|J|}{n}\big)^{\beta} +  \frac{C_r \tilde{\sigma}}
		{\sqrt{|J|}} + \frac{C_r \tilde{\sigma}^2}{\lambda} + \frac{2\lambda}{|J|}\big).
\end{align}

Now we will choose $J$ so that the sum of the first two terms inside the min in~\eqref{eq:biasvar} are minimized. For this, we can choose among $\{J \in \mathcal{I}_{i_0}: J \subseteq [i_0 \pm s_0]\}$ such that $$|J| = B_n = \lfloor \min\{\tilde{\sigma}^{2/(2\beta + 1)} L^{-2/(2\beta + 1)} n^{2\beta/(2\beta + 1)},l_0\} \rfloor.$$
In the above $l_0 = |[i_0 \pm s_0]|.$

With this choice the sum of the first two terms inside the min in~\eqref{eq:biasvar} simply becomes 
$$R_n = \max\{\tilde{\sigma}^{2\beta/(2\beta + 1)} L^{1/(2\beta + 1)} n^{-\beta/(2\beta + 1)}, \tilde{\sigma} l_0^{-1/2}\}.$$ 

Now note that with this choice of $J$, the sum of the last two terms (up to a constant factor) inside the min in~\eqref{eq:biasvar} equals 
$$g(\lambda) = \frac{\tilde{\sigma}^2}{\lambda} + \frac{\lambda}{B_n}\big.$$ It is easy to see that 
this is minimized when $\lambda^* = \tilde{\sigma} \sqrt{B_n}.$ Notably, $g(\lambda^*)$ is of the same order as the optimized first-two-terms contribution, and hence is of the same order as $R_n$ (up to constants). This finishes the proof. 
\end{proof}

   \begin{lemma}[Local Bias Lower Bound]\label{lem:biastight}
Fix an integer $0 \leq r \leq 10$ and let $\beta = r+1$.
There exists a constant $C_r>0$ depending only on $r$ such that the following holds.

For any $n$, any location $i\in[n]$, and any interval
$J=[i\pm m]\subseteq[n]$ with $m\to\infty$ and $m/n\to 0$,
there exists a function $f\in C^{r,1}$ with
$\theta^*_l = f(l/n)$ for $l\in[n]$ such that the bias of degree-$r$
local polynomial regression satisfies
\[
\max\bigl\{
|Bias^{(r)}_{+}(i,J,\theta^*)|,
|Bias^{(r)}_{-}(i,J,\theta^*)|
\bigr\}
\;\ge\;
C_r \Bigl(\frac{|J|}{n}\Bigr)^{\beta}.
\]
\end{lemma}

\begin{proof}

Define the effective bandwidth
\[
h = \frac{|J|}{n} = \frac{2m+1}{n}.
\]

Let $x_0 = i/n$ and define the design points
$x_j = (i+j)/n$ for $j\in[m]$ and
$x_{m+j}=(i-j)/n$ for $j\in[1:m]$.
Define the rescaled covariates
\[
u_j = \frac{x_j-x_0}{h}, \qquad j\in[0:2m],
\]
so that $u_0=0$ and $u_j\in[-1/2,1/2]$.

Define the design matrix $\mathbf X\in\R^{(2m+1)\times(r+1)}$ by
\[
\mathbf X
=
\begin{pmatrix}
1 & u_0 & \cdots & u_0^r \\
1 & u_1 & \cdots & u_1^r \\
\vdots & \vdots & & \vdots \\
1 & u_{2m} & \cdots & u_{2m}^r
\end{pmatrix}.
\]

Let $\tilde y_j$ denote the entry of $y$ corresponding to the design point $x_j.$
The local polynomial estimator of degree $r$ (at $x_0$) solves
\[
\hat{\boldsymbol\beta}
=
\arg\min_{\boldsymbol\beta\in\R^{r+1}}
\sum_{j=0}^{2m}
\left(
\tilde y_j - \beta_0 - \beta_1 u_j - \cdots - \beta_r u_j^r
\right)^2,
\]
and the fitted value at $x_0$ is
\[
\hat f^{(r)}(x_0)
=
\mathbf e_1^\top \hat{\boldsymbol\beta},
\qquad
\mathbf e_1=(1,0,\dots,0)^\top.
\]

By standard least squares theory, 
\[
\hat{\boldsymbol\beta}
=
(\mathbf X^\top \mathbf X)^{-1}\mathbf X^\top \tilde{\mathbf y},
\qquad
\E \hat{\boldsymbol\beta}
=
(\mathbf X^\top \mathbf X)^{-1}\mathbf X^\top \mathbf v,
\]
where $\mathbf v=(f(x_0),\dots,f(x_{2m}))^\top$.
Therefore,
\[
\E \hat f^{(r)}(x_0)
=
\mathbf e_1^\top (\mathbf X^\top \mathbf X)^{-1}\mathbf X^\top \mathbf v.
\]

\medskip
\noindent

Now we choose a \textit{worst case function}.
Define
\[
f(x) = |x-x_0|^{r+1}.
\]
Then $f\in C^{r,1}$ and $f(x_0)=0$.
Moreover, for each $j\in[0:2m]$,
\[
f(x_j) = h^{r+1}|u_j|^{r+1}.
\]
Thus we may write
\[
\mathbf v = h^{r+1}\mathbf b,
\qquad
\mathbf b_j = |u_j|^{r+1}, \;\; j\in[0:2m].
\]

Substituting this yields
\[
\E \hat f^{(r)}(x_0)
=
h^{r+1}
\mathbf e_1^\top (\mathbf X^\top \mathbf X)^{-1}\mathbf X^\top \mathbf b.
\]
Since $f(x_0)=0$, the bias satisfies
\[
Bias
=
\E \hat f^{(r)}(x_0) - f(x_0)
=
h^{r+1}
\mathbf e_1^\top (\mathbf X^\top \mathbf X)^{-1}\mathbf X^\top \mathbf b.
\]

\medskip
\noindent

Now we will show that the term $e_1^\top (\mathbf X^\top \mathbf X)^{-1}\mathbf X^\top \mathbf b$ converges to some non zero constant.  
Consider an element of the normalized matrix for $k,\ell\in\{0,\dots,r\}$,
\[
\frac{1}{m}(\mathbf X^\top \mathbf X)_{k\ell}
=
\frac{1}{m}\sum_{j=0}^{2m} u_j^{k+\ell} = \frac{1}{m}\sum_{j=-m}^{m} (\frac{j}{2m + 1})^{k+\ell}
\;\longrightarrow\;
\E U^{k+\ell},
\]
and also an element of the normalized vector
\[
\frac{1}{m}(\mathbf X^\top \mathbf b)_k
=
\frac{1}{m}\sum_{j=0}^{2m} u_j^k |u_j|^{r+1}
\;\longrightarrow\;
\E U^k |U|^{r+1},
\]
where $U\sim\mathrm{Unif}[-1/2,1/2]$.
Let $M$ denote the $(r+1)\times(r+1)$ moment matrix
with $M[i,j]=\E U^{i+j}$.

By continuous mapping theorem,
\[
\mathbf e_1^\top (\mathbf X^\top \mathbf X)^{-1}\mathbf X^\top \mathbf b
\;\longrightarrow\;
\mathbf e_1^\top M^{-1} d,
\]
where $d_k=\E U^k|U|^{r+1}$.

Thus the limiting bias constant appearing is
$\mathbf e_1^\top M^{-1} d$. For $r=0,\dots,10$, this limiting constant can be computed explicitly by symbolic
calculation. The resulting values are listed in the table below and are all nonzero.

\[
\begin{array}{c|c}
r & \mathbf e_1^\top M^{-1} d \\ \hline
0  & \dfrac{1}{4} \\[6pt]
1  & \dfrac{1}{12} \\[6pt]
2  & -\dfrac{1}{128} \\[6pt]
3  & -\dfrac{3}{560} \\[6pt]
4  & \dfrac{3}{8192} \\[6pt]
5  & \dfrac{5}{14784} \\[6pt]
6  & -\dfrac{5}{262144} \\[6pt]
7  & -\dfrac{7}{329472} \\[6pt]
8  & \dfrac{35}{33554432} \\[6pt]
9  & \dfrac{63}{47297536} \\[6pt]
10 & -\dfrac{63}{1073741824}
\end{array}
\]

\medskip
\noindent

Consequently, for all sufficiently large $m$,
$$|Bias| \geq C_r h^{r + 1} \geq C_r \Bigl(\frac{|J|}{n}\Bigr)^{\beta}$$
for some constant $C_r>0$ depending only on $r$.

\medskip
\noindent
\textbf{Conjecture.}
The explicit computations suggest that $\mathbf e_1^\top M^{-1}d\neq 0$ for all
$r\ge 0$. If true, the above argument would extend verbatim to all degrees.
Since this question is tangential to the main focus of the paper, we leave its
full resolution for future work.

\end{proof}

	\section{Proof of Theorem~\ref{thm:slow} (Slow Rate)}\label{sec:slow}

\begin{proof}
For a $\delta > 0$ to be chosen later, we invoke Proposition~\ref{prop:piecewise} to obtain an interval partition $\pi_{\delta} \coloneqq \pi$ such that 
\newline
\newline
a) $TV^{(r)}(\theta^*_{I}) \leq V \delta \:\:\:\: \forall I \in \pi$, \newline
b)\:For any $i \in [n]$, we have $$
\max\{|Bias^{(r - 1)}_{+}(i,J_i,\theta^*)|,|Bias^{(r - 1)}_{-}(i,J_i,\theta^*)|\} \leq C_r V \delta$$ where $J_i$ is the interval within the partition $\pi$
which contains $i$,\newline
c) $|\pi| \leq C_r \delta^{-1/r}$
\newline
d) For any integer $u \geq 0$,

\begin{equation*}
\bigl|\{I \in \pi: c_1 \tfrac{n}{2^u} \le |I| \le c_2 \tfrac{n}{2^u}\}\bigr|
\le C_r \min\{\tfrac{2^{-u(r-1)}}{\delta}, 2^u\}
\end{equation*}
where $c_1,c_2$ are absolute constants.

Now, let us bound the positive part of $\hat{\theta_i} - \theta^*_i.$ The negative part can be bounded similarly. The bound as given by Theorem~\ref{thm:maingeneral} is that with high probability,
\begin{align*}
\hat{\theta}^{(r - 1,\lambda)}_i - \theta^*_i &\leq \min_{J \in \mathcal{I}: i \in J} \left(Bias_{+}^{(r - 1)}(i,J,\theta^*) + SD^{(r - 1)}(i,J,\lambda)\right) \\&\leq Bias_{+}^{(r - 1)}(i,J_i,\theta^*) + SD^{(r - 1)}(i,J_i,\lambda) \\& \leq C_r V \delta + \frac{C_r \sigma \sqrt{\log n}}
{\sqrt{Dist(i,\partial J_i)}} + \frac{C_r \sigma^2 \log n}{\lambda} + \frac{2\lambda}{|J_i|}.
\end{align*}


Squaring and adding over all indices in $i$, we get 

\begin{equation}\label{eq:bdslow1}
\sum_{i = 1}^{n} (\hat{\theta}^{(r - 1,\lambda)}_i - \theta^*_i)^2_{+} \lesssim n V^2 \delta^2 + \sigma^2 \log n \underbrace{\sum_{i = 1}^{n} \frac{1}{Dist(i,\partial J_i)}}_{T_1} + \frac{n \sigma^4 \log^2 n}{\lambda^2} + \lambda^2 \underbrace{\sum_{i = 1}^{n} \frac{1}{|J_i|^2}}_{T_2}
\end{equation}

where $\lesssim$ notation means up to a constant factor $C_r$ which only depends on $r.$ We will use this notation throughout this proof.

We will now bound $T_1$ and $T_2$ separately. Let $\pi$ consist of intervals $(B_1,\dots,B_k)$ where $k = |\pi| \lesssim \delta^{-1/r}.$ Let us also denote the cardinalities of these intervals by $n_1,\dots,n_k.$

We can write
\begin{align*}
T_1 = \sum_{i = 1}^{n} \frac{1}{Dist(i,\partial J_i)} &= \sum_{l = 1}^{k} \sum_{i \in B_l} \frac{1}{Dist(i,\partial B_l)} = \sum_{l = 1}^{k} 2 \big(1 + \frac{1}{2} + \dots + \frac{1}{n_l/2}\big) \\& \lesssim \sum_{l = 1}^{k} \log n_l = k \big(\frac{1}{k} \sum_{l = 1}^{k} \log n_l\big) \leq k \log \frac{n}{k} \leq k \log n \lesssim \delta^{-1/r} \log n
\end{align*}
where in the third last inequality we used Jensen's inequality.

We can also write
\begin{align*}
T_2 = \sum_{i = 1}^{n} \frac{1}{|J_i|^2} = \sum_{l = 1}^{k} \sum_{i \in B_l} \frac{1}{|B_l|^2} = \sum_{l = 1}^{k}\frac{1}{n_l}. 
\end{align*}

At this point, for the sake of simpler exposition, we assume $n$ is a power of $2$ although the argument works for any $n.$ Then, by the nature of our recursive dyadic partioning scheme, the cardinalities $n_l$ are of the form $\frac{n}{2^{u}}$ for some integer $u \geq 0.$ Continuing from the last display, we can write
\begin{align*}
\sum_{l = 1}^{k}\frac{1}{n_l} = \sum_{l = 1}^{k} \sum_{u = 0}^{\infty} \frac{1}{n_l} \mathrm{1}(n_l = \frac{n}{2^u}) &= \sum_{u = 0}^{\infty} \frac{2^{u}}{n} \sum_{l = 1}^{k} \mathrm{1}(n_l = \frac{n}{2^u}) \lesssim \sum_{u = 0}^{\infty} \frac{2^{u}}{n} \min\{\frac{2^{-u(r - 1)}}{\delta},2^{u}\}
\\& = \frac{1}{n} \sum_{u = 0}^{\infty} \min\{\frac{2^{-u(r - 2)}}{\delta},2^{2u}\} \lesssim \frac{\delta^{-2/r}}{n}.
\end{align*}


The last step above follows from the fact that there exists a nonnegative integer $u^* = O(\log(1/\delta))$ such that 
\begin{equation*}
\min\{\frac{2^{-u(r - 2)}}{\delta},2^{2u}\} = 
\begin{cases}
2^{2u}, & \text{for} \:\:u < u^* \\
\frac{2^{-u(r - 2)}}{\delta} & \text{for} \:\:u \geq u^*.
\end{cases}
\end{equation*}

Therefore, we obtain
\begin{equation*}
T_2 \lesssim \frac{\delta^{-2/r}}{n}.
\end{equation*}

The two bounds on $T_1$ and $T_2$ respectively, along with~\eqref{eq:bdslow1} lets us obtain
\begin{equation}\label{eq:bdslow2}
\sum_{i = 1}^{n} (\hat{\theta}^{(r - 1,\lambda)}_i - \theta^*_i)^2_{+} \lesssim n V^2 \delta^2 + \sigma^2 \delta^{-1/r} (\log n)^2  + \frac{n \sigma^4 \log^2 n}{\lambda^2} + \frac{\lambda^2 \delta^{-2/r}}{n}.
\end{equation}

Now the above bound holds for any $\delta > 0$, hence we can optimize the bound over $\delta.$ Note that the first two terms do not involve $\lambda.$ Let us minimize the sum of the first two terms; we can do this by setting $$\delta \coloneqq \delta^* = C_r \big(\frac{\sigma^2 (\log n)^2}{n V^2}\big)^{r/(2r + 1)}.$$

Then the sum of the first two terms scales like 
\begin{equation}\label{eq:bdslow3}
(n V^2)^{1/(2r + 1)} (\sigma^2 (\log n)^2)^{2r/(2r + 1)}
\end{equation}

We will now handle the sum of the last two terms in the bound in~\eqref{eq:bdslow2}, these are the terms which involve $\lambda$ and will inform us of a good choice of $\lambda.$
We will show that with an optimal choice of $\lambda$, this sum of the last two terms is essentially of the same order as the expression in~\eqref{eq:bdslow3}.

We will plug in the optimized choice $\delta^*$ here. Let us denote the effective number of pieces 
$$k^* = (\delta^*)^{-1/r}.$$

Then the sum of the last two terms in~\eqref{eq:bdslow2} can be written as 
$$\frac{n \sigma^4 \log^2 n}{\lambda^2} + \frac{\lambda^2 (k^*)^2}{n}.$$

The above suggests that we minimize the sum of the above two terms by equating them. This will mean that we need to choose $$\lambda = C_r \big(\frac{n^2}{(k^*)^2} \sigma^4 (\log n)^2\big)^{1/4} = C_r n^{r/(2r + 1)} V^{-1/(2r + 1)} \sigma^{1 + 1/(2r + 1)} (\log n)^{1/2 + 1/(2r + 1)}.$$

By setting this choice of $\lambda$, the sum of the two terms involving $\lambda$
scale like
$$k^* \sigma^2 \log n = (\delta^*)^{-1/r} \sigma^2 \log n.$$

This is dominated by the sum of the first two terms as can be seen from the second term in~\eqref{eq:bdslow2}. This finishes the proof. 
\end{proof}

\section{Proof of Theorem~\ref{thm:fast} (Fast Rate)}\label{sec:fastrate}

\begin{proof}
We are given that there exists an interval partition $\pi^*$ of $[n]$ with intervals $I_1,I_2,\dots,I_k$ such that $\theta^*_{I_j}$ is a polynomial of degree $r \geq 0$ for each $j = 1,\dots,k.$ Since $I_1,I_2,\dots,I_k$ forms a partition of $[n],$ for any index $i \in [n]$, one of these intervals contains $i.$ Let us denote this interval by $J_i.$

Let us bound the positive part of $\hat{\theta}^{(r,\lambda)}_i - \theta^*_i.$ The negative part can be bounded similarly. The bound as given by Theorem~\ref{thm:maingeneral} is that with high probability,
\begin{align*}
\hat{\theta}^{(r,\lambda)}_i - \theta^*_i &\leq \min_{J \in \mathcal{I}: i \in J} \left(Bias_{+}^{(r)}(i,J,\theta^*) + SD^{(r)}(i,J,\lambda)\right) \leq Bias_{+}^{(r)}(i,J_i,\theta^*) + SD^{(r)}(i,J_i,\lambda) \\& \leq \frac{C_r \sigma \sqrt{\log n}}
{\sqrt{Dist(i,\partial J_i)}} + \frac{C_r \sigma^2 \log n}{\lambda} + \frac{2\lambda}{|J_i|}.
\end{align*}
because by definition, $Bias_{+}^{(r)}(i,J_i,\theta^*) = 0.$

Squaring and adding over all indices in $i$, we get 

\begin{equation}\label{eq:bdfast1}
\sum_{i = 1}^{n} (\hat{\theta}^{(r,\lambda)}_i - \theta^*_i)^2_{+} \lesssim \sigma^2 \log n \underbrace{\sum_{i = 1}^{n} \frac{1}{Dist(i,\partial J_i)}}_{T_1} + \frac{n \sigma^4 \log^2 n}{\lambda^2} + \lambda^2 \underbrace{\sum_{i = 1}^{n} \frac{1}{|J_i|^2}}_{T_2}
\end{equation}

As in the proof of Theorem~\ref{thm:slow}, we have
$$T_1 \lesssim k \log \frac{n}{k}.$$

As for $T_2$, we have to use the minimum length condition that each of the $|J_i|$ have length at least $c \frac{n}{k}.$ Therefore,
$$T_2 = \sum_{i = 1}^{n} \frac{1}{|J_i|^2} = \sum_{l = 1}^{k} \sum_{i \in I_l} \frac{1}{|I_l|^2} = \sum_{l = 1}^{k}\frac{1}{|I_l|} \lesssim \frac{k^2}{n}.$$

Therefore, we get the bound
\begin{equation}\label{eq:bdfast2}
\sum_{i = 1}^{n} (\hat{\theta}^{(r,\lambda)}_i - \theta^*_i)^2_{+} \lesssim \sigma^2 k \log n \log \frac{n}{k} + \frac{n \sigma^4 \log^2 n}{\lambda^2} + \lambda^2 \frac{k^2}{n}.
\end{equation}

We can now choose $$\lambda = C_r \big(\frac{n^2 \sigma^4 (\log n)^2}{k^2}\big)^{1/4}$$ to obtain the final bound
\begin{equation*}
\sum_{i = 1}^{n} (\hat{\theta}^{(r,\lambda)}_i - \theta^*_i)^2_{+} \lesssim \sigma^2 k \log n \log \frac{n}{k} + \sigma^2 k \log n.
\end{equation*}

To obtain the $\ell_1$ loss bound we again start from 

\begin{align*}
\sum_{i = 1}^{n} (\hat{\theta}^{(r,\lambda)}_i - \theta^*_i)_{+} &\lesssim \sigma \sqrt{\log n} \underbrace{\sum_{i = 1}^{n} \frac{1}{\sqrt{Dist(i,\partial J_i)}}}_{T_1} + \frac{n \sigma^2 \log n}{\lambda} + \lambda \sum_{i = 1}^{n} \frac{1}{|J_i|} 
\\& \lesssim \sigma \sqrt{\log n} \sum_{l = 1}^{k} \big(\frac{1}{\sqrt{1}} + \dots + \frac{1}{\sqrt{|I_l|}}\big) + \frac{n \sigma^2 \log n}{\lambda} + \lambda \sum_{l = 1}^{k} \sum_{i \in I_l} \frac{1}{|J_i|} \\& \lesssim \sigma \sqrt{\log n} \sum_{l = 1}^{k} \sqrt{|I_l|} + \frac{n \sigma^2 \log n}{\lambda} + \lambda k \\& \leq \sigma \sqrt{n k \log n} + \frac{n \sigma^2 \log n}{\lambda} + \lambda k
\end{align*}
where in the last inequality we used Jensen's inequality. Setting 
$$\lambda = \big(\frac{n \sigma^2 \log n}{k}\big)^{1/2}$$ we get the final bound
$$\sum_{i = 1}^{n} (\hat{\theta}^{(r,\lambda)}_i - \theta^*_i)_{+} \lesssim \sigma \sqrt{n k \log n}.$$

\end{proof}

	\end{appendix}


	
	\bibliographystyle{imsart-number} 
	\bibliography{references}       
	
\end{document}